\documentclass[oneside,english]{amsart}
\usepackage{lmodern}
\usepackage[T1]{fontenc}
\usepackage[latin9]{inputenc}
\usepackage{geometry}
\usepackage{color}
\usepackage{mathrsfs}
\usepackage{amsbsy}
\usepackage{amstext}
\usepackage{amsthm}
\usepackage{amssymb}
\usepackage[normalem]{ulem}
\usepackage{dsfont}
\usepackage{xcolor}

\DeclareMathOperator*{\supess}{ess\,sup}

\makeatletter
\numberwithin{equation}{section}
\numberwithin{figure}{section}
\theoremstyle{plain}
\newtheorem{thm}{\protect\theoremname}[section]
\theoremstyle{remark}
\newtheorem{rem}[thm]{\protect\remarkname}
\theoremstyle{definition}

\theoremstyle{plain}
\newtheorem{lem}[thm]{\protect\lemmaname}
\theoremstyle{plain}
\newtheorem{cor}[thm]{\protect\corollaryname}
\theoremstyle{plain}
\newtheorem{prop}[thm]{\protect\propositionname}

\makeatother

\usepackage{babel}
\providecommand{\corollaryname}{Corollary}
\providecommand{\definitionname}{Definition}
\providecommand{\lemmaname}{Lemma}
\providecommand{\remarkname}{Remark}
\providecommand{\theoremname}{Theorem}
\providecommand{\propositionname}{Proposition}

\newcommand{\abbr}[1]{{\sc\lowercase{#1}}}

\begin{document}
\global\long\def\Crt{{\rm Crt}}%
\global\long\def\dist{{\rm \mathsf{dist}}}%
\global\long\def\Es{E_{\star}}%
\global\long\def\GS{\mathsf{GS}}%
\global\long\def\Rs{R_{\star}}%
\global\long\def\hRs{R_{\star}}%
\global\long\def\V{{\rm Vol}}%
\global\long\def\bs{\boldsymbol{\sigma}}%
\global\long\def\bsigma{\boldsymbol{\sigma}}%
\global\long\def\bx{\mathbf{x}}%
\global\long\def\BB{\mathbf{B}}%
\global\long\def\be{\mathbf{e}}%
\global\long\def\bxs{\mathbf{x}_{\star}}%
\global\long\def\qs{q_{\star}}%
\global\long\def\bz{\mathbf{z}}%
\global\long\def\by{\mathbf{y}}%
\global\long\def\sfv{{\mathsf{v}}}%
\global\long\def\sfu{{\mathsf{u}}}%
\global\long\def\bv{\mathbf{v}}%
\global\long\def\bu{\mathbf{u}}%
\global\long\def\bw{\mathbf{w}}%
\global\long\def\bn{\mathbf{n}}%
\global\long\def\bM{\mathbf{M}}%
\global\long\def\bJ{\mathbf{J}}%
\global\long\def\bG{\mathbf{G}}%
\global\long\def\indic{\mathds{1}}%
\global\long\def\diag{{\rm diag}}%
\global\long\def\sol{\mathsf{Sol}}%
\global\long\def\SN{\mathbb{S}^{N-1}}%
\global\long\def\BN{\mathbb{B}^{N}}%
\global\long\def\SNt{\mathbb{S}^{N-2}}%
\global\long\def\sp{{\rm sp}}%
\global\long\def\AA{\mathbb{A}}%
\global\long\def\E{\mathbb{E}}%
\global\long\def\P{\mathbb{P}}%
\global\long\def\R{\mathbb{R}}%
\global\long\def\V{\mathbb{V}}%
\global\long\def\T{\mathbb{T}}%
$ $
\global\long\def\BN{\mathbb{B}^{N}}%
\global\long\def\N{\mathsf{N}}%
\global\long\def\cpt{{\rm \mathsf{CP}}}%
\global\long\def\Cs{\mathscr{C}}%
\global\long\def\DD{\mathscr{D}}
\global\long\def\calE{\mathcal{E}}%
\global\long\def\oN{{\bar N}}% 
 
\global\long\def\epsilon{\varepsilon} 
\global\long\def\b{\beta}%
\global\long\def\BJ{\mathbf{J}}%
\global\long\def\ba{\mathbf{a}}%
\global\long\def\bb{\mathbf{b}}%
\global\long\def\bc{\mathbf{c}}%
\global\long\def\bcs{\mathbf{c}_{\star}}%
\global\long\def\grad{\nabla_{\mathrm{sp}}}%
\global\long\def\gradt{\bar \nabla}%
\global\long\def\Hess{\nabla_{\mathrm{sp}}^{2}}%
\global\long\def\Hesst{\bar\nabla^{2}}%
\global\long\def\HessT{\nabla_{\scriptscriptstyle T}^{2}}%
\global\long\def\cF{\mathcal{F}}%
\global\long\def\xim{\xi_{(m)}}
\global\long\def\xin{\xi_{(n)}}
\global\long\def\bxim{\bar{\xi}_{(m)}}
\global\long\def\bxik{\bar{\xi}_{(k)}}

\global\long\def\Dp{D'}

\title{
% On 
Disordered Gibbs measures and Gaussian conditioning}
%\author{Amir Dembo and Eliran Subag}
\author{Amir Dembo}\thanks{Email: adembo@stanford.edu} 
\author{Eliran Subag}\thanks{Email: eliran.subag@weizmann.ac.il (corresponding author)}
\begin{abstract}
We study the law of a random field $f_N(\bs)$ evaluated at a random sample from the Gibbs measure associated to 
a Gaussian field $H_N(\bs)$.  In the high-temperature regime, we show that bounds on the probability that $f_N(\bs)\in A$ for $\bs$ randomly sampled from the Gibbs measure can be deduced from similar bounds for deterministic $\bs$ under the conditional Gaussian law given that $H_N(\bs)/N=E$ for $E$ close to 
the derivative $F'(\beta)$ of the free energy (which is 
the typical value of $H_N(\bs)/N$ under the Gibbs measure).  
In the more challenging low-temperature regime we restrict to \abbr{$k$-RSB} spherical spin glasses,  
proving a similar result,  now
%to the high-temperature regime,  but 
with a more elaborate conditioning.  Namely,  with $q_i$ denoting the locations 
of the non-zero atoms of the Parisi measure,  
in addition to specifying that
% the energy 
$H_N(\bs)/N=E$, 
%at $\bs$,  
here 
one needs to also condition on the energy and its gradient at
%a sequence of 
points  $\bx_1,\ldots,\bx_k$ such that
$\langle \bx_i,\bx_j\rangle/N=q_{i\wedge j}$ and $\langle \bx_i,\bs\rangle/N\approx q_{i}$.  Like in
the high-temperature phase,   the energy and gradient values on which one conditions are 
also specified by the model's Parisi measure.  We apply our general results to two important problems from
statistical physics.  
That is,  computing the Franz-Parisi potential at any temperature and,  reducing
certain asymptotics of Langevin dynamics with initial conditions distributed according to the Gibbs measure,  
to the more manageable problem of studying dynamics with non-random initial conditions 
and conditional disorder.
% is one of the main motivations of this work.
\end{abstract}

\maketitle

\section{Introduction and main results}

For any $N\geq1$, let $(\Sigma_{N},\cF_{N},\mu_{N})$ be a probability (configuration)
space,  with a random \emph{Hamiltonian} $H_{N}:\Sigma_{N}\to\R$ which is a centered 
Gaussian field 
on an auxiliary probability space,  denoted hereafter by $\P$
(our standing assumption is that $\P$-a.s.\  any function on $\Sigma_{N}$
we encounter,  is $\cF_{N}$-measurable).  
For any inverse-temperature $\beta\geq0$,  define the partition function and free
energy by
\begin{equation*}
\label{eq:FNbeta}
Z_{N,\beta}:=\int_{\Sigma_{N}}e^{\beta H_{N}(\bs)}d\mu_{N}(\bs),\qquad F_{N,\beta}:=\frac{1}{N}\log Z_{N,\beta}
\end{equation*}
and the (random) Gibbs measure by its density
\[
\frac{dG_{N,\beta}}{d\mu_{N}}(\bs)=\frac{e^{\beta H_{N}(\bs)}}{Z_{N,\beta}}.
\]

Suppose $f_{N}:\Sigma_{N}\to\R$ is another random field,  possibly depending on $H_{N}$ in a complicated way. 
In this work we study the asymptotic 
behavior as $N \to \infty$ of such $f_{N}(\bs)$ evaluated at a random sample $\bs$ from the
Gibbs measure $G_{N,\beta}$.  Precisely,  for Borel $A\subset\R$ we wish to understand
random variables of the form 
\begin{equation}
G_{N,\beta}\left(f_{N}(\bs)\in A\right)=\frac{\int_{\{\bs:\,f_{N}(\bs)\in A\}}e^{\beta H_{N}(\bs)}d\mu_{N}(\bs)}{\int_{\Sigma_{N}}e^{\beta H_{N}(\bs)}d\mu_{N}(\bs)}.\label{eq:G(finA)}
\end{equation}

%We are interested in the asymptotic behavior as $N\to\infty$. 
One often expects $f_{N}(\bs)$ to asymptotically concentrate under the
Gibbs measure around a specific value (i.e.\ its median). Namely,
that for some value $E\in\R$ and any $\epsilon>0$,
\begin{equation}\label{eq:observable_concentration}
\lim_{N\to\infty}\E G_{N,\beta}\Big(
\big|f_N(\bs)-E\big|>\epsilon
\Big)=0.
\end{equation}
The classical example for this
is the normalized ``energy'' $f_{N}=\frac{1}{N}H_{N}$,  whose average under the Gibbs measure is the derivative of the free energy,
\begin{equation}\label{eq:ddbetaF}
\frac{d}{d\beta}F_{N,\beta} = \frac{\int_{\Sigma_{N}}\frac1NH_N(\bs) e^{\beta H_{N}(\bs)}d\mu_{N}(\bs)}{\int_{\Sigma_{N}}e^{\beta H_{N}(\bs)}d\mu_{N}(\bs)}.
\end{equation}
Respectively, in the large $N$ limit, if
for some deterministic,  differentiable function $F$ and any $\beta$,
\begin{equation}
\lim_{N\to\infty}\E F_{N,\beta}=F(\beta),\label{eq:Fconvergence}
\end{equation}
then under the Gibbs measure $\frac{1}{N}H_{N}(\bs)$ 
concentrates around $E=F'(\beta)$ in the sense of \eqref{eq:observable_concentration},  see \cite{AuffingerChenConcentration}. 

What about more general functions $f_{N}$,  possibly depending on $H_{N}$
in a complicated way? From a statistical physics perspective, this question is fundamental: 	we wish to understand the behavior of a general ``observable'' $f_N$ at a configuration $\bs$ drawn from the thermal equilibrium distribution. Indeed, below we discuss two very well-known objects from spin glass theory that fit into this general framework. On the other hand, mathematically this question is closely related to understanding how the random function $H_N$ looks like around a point $\bs$ sampled from a given height. Indeed,  the latter amounts to understanding $f_N(\bs)=\bar f_N(\bs, H_N(\cdot))$ for deterministic functionals $\bar f$, while sampling $\bs$ from the Gibbs measure approximates the uniform measure on the set of points with energy around a given level thanks to concentration of $\frac1N H_N(\bs)$.

A naive approach may suggest to modify the definition of the free energy by replacing $\beta H_N(\bs)$ by $\beta H_N(\bs)+\alpha N f_N(\bs)$, so that similarly to \eqref{eq:ddbetaF} the $\frac{d}{d\alpha}|_{\alpha=0}$ derivative is the average of  $f_N(\bs)$ under $G_{N,\beta}$.  One then hopes to show that $f_N(\bs)$ concentrates around the derivative of the $N\to\infty$ limit of this modified free energy,  as in the case of $\frac1NH_N(\bs)$.  Even if we ignore the issue of concentration,  for many choices of $f_N$ of interest,  it is impossible to analytically compute
such modified free energy.  In contrast, our main results provide a general tool for arbitrary $f_N$ (modulo regularity conditions) to deduce concentration under the Gibbs measure around a specific value $E$, from analogous statements about deterministic $\bs$, under a modified Gaussian disorder arising from a suitable conditioning. Excluding specific examples like $f_N=\frac1N H_N$ above, these results are the first to provide a general principle for concentration under the Gibbs measure for the low temperature phase of spin glasses.

We shall consider two settings.
%this question in two situations. 
The first 
%will 
allows for a very general choice of the
configuration space and Hamiltonian,  at the price of restricting to
the easier regime of small $\beta$.  We shall see that the behavior
under the Gibbs measure can be described by the conditional law of $f_{N}(\bs)$
given that the energy is equal to its typical value $\frac{1}{N}H_{N}(\bs)=F'(\beta)$. The proof 
%will 
exploits the fact that in this regime $\E F_{N,\beta}=\frac1N \log \E Z_{N,\beta}+o(1)$. The argument we 
employ is related to a technique used in \cite{AchlioptasCoja-Oghlan} for random constraint satisfaction problems.  We present the result in Section \ref{subsec:hightemp} and its short proof in Section \ref{sec:HighTempPf},  setting the stage for our main theorem 
(for large $\beta$) by demonstrating how the conditional Gaussian field naturally arises.

In the second setting we consider the more challenging problem of large values of $\beta$. 
Here we shall restrict
%will be able to control
%also  high values of $\beta$ --- a much more difficult problem ---
%when restricting to the 
to spherical spin glass models with $k$-step replica symmetry breaking (\abbr{RSB}),  for some $k$ finite.  As in the case of small $\beta$, we 
%will 
translate questions about the behavior under the Gibbs measure
to ones about $f_{N}(\bs)$ for deterministic $\bs$,  under an appropriate
conditional Gaussian law.  However,  large $\beta$ requires a more elaborate
conditioning that involves additional points $\bx_1,\ldots,\bx_k$ with a specific geometry.  Such $\bx_i$ are conditioned to be critical points on appropriate spheres and further to specifying the values of $H_N$ at $\bs$ we 
specify its values and `radial' derivatives at all $\bx_1,\ldots,\bx_k$.  We present this result in Section \ref{subsec:LowTemp}. 

For mean-field spin glasses, the Gibbs measure $G_{N,\beta}$ is well understood in the large $N$ limit. Namely, its asymptotic structure is explicitly characterized by Ruelle probability cascades, see \cite{PanchenkoBook}. By itself, however, the latter is insufficient to analyze $f_N(\bs)$ under $G_{N,\beta}$. Indeed, thanks to the concentration of the normalized energy $\frac1N H_N(\bs)$ mentioned above, $G_{N,\beta}$ essentially only depends on the structure of the Hamiltonian $H_N$ at energy levels close to $F'(\beta)$,  while the value of $f_N(\bs)$, even for $\bs$ sampled from  $G_{N,\beta}$, may depend (and does in interesting examples) in a complicated way on the structure of $H_N$ outside of such regions.  Our proof for the low-temperature phase of large $\beta\geq0$ exploits the pure states decompositions of Talagrand \cite{TalagrandPstates} and Jagannath \cite{JagannathApxUlt} to relate the Gibbs measure to the aforementioned critical points $\bx_1,\ldots,\bx_k$.   Thanks to the ultrametricity property proved by Panchenko \cite{ultramet}, the pure states are organized in a tree structure, which can be given a geometric interpretation by an embedding in the interior of the sphere in $\R^N$.  In \cite{FElandscape} it was shown that each vertex $\bx$ in the tree roughly maximizes the energy over the sphere of radius $\|\bx\|$ and our Proposition \ref{prop:PS}, 
which is of independent interest,  shows that for
$k$-\abbr{RSB} models,  there exists such a decomposition where the vertices are also critical points in an appropriate sense  (see Section \ref{sec-pf-prop:PS} for its proof). This allows us to 
relate averages over 
the Gibbs measure to ones over
the relevant critical points 
and express the latter  
in terms of the conditional Gaussian law by combining a novel,  multi-level
 Kac-Rice formula (see Proposition \ref{prop:multilvlKR}),
	with the available information on the complexity of such critical points
	from a recent work of Huang and Sellke \cite{HuangSellke2023} (which relates the annealed 
	complexity to the Parisi measure).

It is plausible to carry out 
such a program also for other disordered mean-field models,  involving multi-species,  
structured external fields (signals),  or Ising spins (after improving our 
understanding of the complexity considerations associated with the pure-state decomposition).

Franz and Parisi introduced in \cite{FranzParisiPotential} a certain ``potential'' as a tool to relate dynamical 
properties to the statics of the spherical spin glass model.  The idea here is to sample a point $\bs$ 
from the Gibbs measure and calculate,  at an exponential scale,  the Gibbs probability of sampling another 
point $\bs'$ whose overlap with $\bs$ is approximately $r$.   In Section  \ref{subsec:FPpotential} we 
compute the Franz-Parisi potential 
%\cite{FranzParisiPotential} 
at any temperature,  by applying both our result for the high temperature phase and 
for the low temperature phase.  To the best of our knowledge,  this is {\em the first
rigorous computation of the Franz-Parisi potential at low temperature,} for any spherical model. 
Restricting for simplicity the low temperature phase to 1-\abbr{rsb} models,  our formula \eqref{eq:Lambda} 
involves a maximization over the possible values of the overlap $\rho$ between $\bs'$ and $\bx_1$.  
We thus expect the latter parameter,  and in particular the continuity in $r$ of the argmax in
\eqref{eq:Lambda},  or the lack thereof, to be highly relevant for the asymptotic behavior of certain natural dynamics 
which are reversible \abbr{wrt} $G_{N,\b}$.  Similarly,  the formula for the Franz-Parisi 
potential of $k$-\abbr{rsb} models involves the maximization over the overlaps between
$\bs'$ and $\bx_i$, $1 \le i \le k$.

Another important application of this work is to the study
% in \cite{DS2} 
of the Langevin dynamics $(\bx_t, t \ge 0)$ at inverse-temperature $\widetilde\beta$
for spherical mixed $p$-spin Hamiltonians $H_N$.  That is,
\begin{equation}\label{diffusion}
 \bx_t=\bx_0-\int_0^t b_\ell'(\vert\vert \bx_u\vert\vert^2/N)\bx_u du  + \frac12\widetilde\beta \int_0^t \nabla H_N(\bx_u) du + \BB_t,
 \end{equation}
where $\BB_t$ is a standard $N$-dimensional Brownian motion and the soft constraint potentials $b_\ell(\cdot)$ concentrate at one
as $\ell \to \infty$,  thereby converting the $\R^N$-valued diffusion \eqref{diffusion} to the spherical Langevin dynamic,
reversible \abbr{wrt} the random Gibbs measure 
$G_{N,\widetilde\beta}$.   Such disordered Langevin dynamics
often exhibit very different behavior at various time-scales,  including {\em aging} behavior 
(where the older it gets,  the longer the system takes to forget its current state),  and there is 
much interest in these dynamics in out-of-equilibrium statistical physics (cf.  \cite{BKM,Crisanti1993,Cugliandolo1993,LesHouches} and the references therein).  In 
particular,  for Gibbs initial conditions,  namely,  with $\bx_0 =\bs$ drawn from $G_{N,\beta}$,  possibly at $\beta \ne \widetilde\beta$.
For $\bx_0$ uniformly distributed on the sphere,  
the thermodynamic 
limit as $N \to \infty$ (and then $\ell \to \infty$),  while $t,s$ are fixed, 
of the empirical covariance, the integrated response functions and the normalized potential,  namely,
\begin{equation}\label{eq:cov}
C_N(s,t) := \frac{1}{N} \langle \bx_s,\bx_t\rangle,   \quad 
\chi_N(s,t):=\frac{1}{N} \langle \bx_s, \BB_t \rangle \quad {\rm and} \quad \widetilde{H}_N(s):=\frac{1}{N} H_N(\bx_s),
\end{equation}
% and the corresponding response function, 
respectively,  satisfy the 
%celebrated 
\abbr{CKCHS}-equations of Cugliandolo-Kurchan \cite{Cugliandolo1993}  
%(who consider instead $C(2\cdot,2\cdot)$ and $R(2\cdot,2\cdot)$), 
and Crisanti-Horner-Sommers \cite{Crisanti1993}.   The \abbr{CKCHS}-equations have 
been rigorously derived in \cite{BDG1},  but prior to our work
there has been no such result for \eqref{diffusion} with Gibbs initial conditions.  It was
\emph{not even known what the thermodynamic limit 
%of \eqref{diffusion} 
might then be},  except for 
pure $p$-spins at the replica symmetric high temperature 
%regime 
($\beta < \beta_c$),  where using the non-rigorous replica method, 
\cite{BBM,BF} predict such limit equations for $(C_N,\chi_N,\widetilde{H}_N)$.  The challenge to 
% even 
merely predicting what the limit ought to be at low temperatures,  stems from the lack of any representation 
of $(C_N,\chi_N,\widetilde{H}_N)$ (or their expected value),  as explicit functionals
of the random Hamiltonian.  We show in Section \ref{subsec:Langevin} how our results reduce
having Gibbs initial conditions within the finite \abbr{rsb} phase,  to the much easier study of \eqref{diffusion} 
for non-random $\bx_0$ with suitable linear constraints imposed on the Hamiltonian.  Indeed,  taking 
advantage of this reduction allows us to 
rigorously derive in \cite{DS2} the thermodynamic limit of $(C_N,\chi_N,\widetilde{H}_N)$ 
for any Gibbsian initial conditions within the $1$-\abbr{rsb} phase (the same 
applies for the $k$-\abbr{rsb} phase,  apart from having then a more complex limiting system of 
equations and a technically more involved derivation).

\subsection{\label{subsec:hightemp}High temperature phase}

Assume that the normalized variance $\frac{1}{N}\E\big[H_{N}(\bs)^{2}\big]$ is bounded uniformly
in $N\geq1$ and $\bs\in\Sigma_{N}$. For simplicity, we further assume that it is essentially
constant over space,  in the sense that for any $\epsilon>0$,
\begin{equation}
\varlimsup_{N\to\infty}\frac{1}{N}\log\mu_{N}\Big(\bs:\,\Big|\frac{1}{N}\E\big[H_{N}(\bs)^{2}\big]-1\Big|>\epsilon\Big)=-\infty.\label{eq:variance-bd}
\end{equation}
We consider small values of $\beta$, for
which the analysis is simpler. %This may expressed in several ways,
%but most relevant to our discussion is that deviations of the partition
%function $Z_{N,\beta}$ from its expectation are small at exponential
%scale. 
Namely,  we work here in the setting where \abbr{whp}, 
\[
\frac{1}{N}\log Z_{N,\beta}=\frac{1}{N}\log\E Z_{N,\beta}+o(1).
\]
Of course,
\[
\E Z_{N,\beta}=\int_{\Sigma_{N}}\E \big[e^{\beta H_{N}(\bs)}\big] d\mu_{N}(\bs)=\int_{\Sigma_{N}}e^{\frac{1}{2}\beta^{2}\E\big[H_{N}(\bs)^{2}\big]}d\mu_{N}(\bs)
\]
and (\ref{eq:variance-bd}) thus implies that, for all $\beta$, 
\begin{equation*}
\lim_{N\to\infty}\frac{1}{N}\log\E Z_{N,\beta}=\frac{1}{2}\beta^{2}.\label{eq:EZconvergence}
\end{equation*}
\begin{rem}\label{rem:conditional}
We define the critical inverse-temperature $\beta_{c}$ as the unique value
such that\footnote{Let $\beta'<\beta$ and $\beta''=\sqrt{\beta^{2}-\beta'^{2}}$.  Note
that $\beta H_{N}(\bs)$ has the same law as $\beta'H^{(1)}_{N}(\bs)+\beta''H^{(2)}_{N}(\bs)$
for two  independent copies $H_N^{(i)}(\cdot)$, $i=1,2$ of $H_N(\cdot)$. Hence,  from Jensen's inequality
and \eqref{eq:variance-bd},
\[
\E F_{N,\beta}\leq\E F_{N,\beta'}+\frac{1}{N}\log\E Z_{N,\beta''} +o(1).
\]
It follows that if $\beta$ satisfies the equality in (\ref{eq:bcdef}),  then so
does $\beta'$.  By H\"older's inequality,  $F_{N,\beta}$ is convex
and therefore equality as in (\ref{eq:bcdef}) holds also 
at the critical $\beta=\beta_c$. }
\begin{equation}
\lim_{N\to\infty}\E F_{N,\beta}=\frac{1}{2}\beta^{2}\iff\beta\leq\beta_{c}.\label{eq:bcdef}
\end{equation}
For any fixed $\bs \in \Sigma_N$ and Borel $A \subset \R$,  the sub-probability measure on $\R$
\[
\P\big(f_{N}(\bs)\in A,\,\frac{1}{N}H_{N}(\bs)\in\cdot\big)
\]
is absolutely continuous \abbr{wrt} the Lebesgue measure allowing us to define the conditional probability
\begin{equation}\label{eq:reg-cond}
p_N(\bs,E) := \P\big(f_{N}(\bs)\in A\,\big|\,\frac{1}{N}H_{N}(\bs)=E \big)
\end{equation}
as the corresponding Radon-Nikodym derivative,   normalized by the positive
Gaussian density of $\frac{1}{N}H_{N}(\bs)$ at $E$.  We  assume hereafter that 
$p_N$ of \eqref{eq:reg-cond} are measurable on $\Sigma_N \times \R$ equipped with
the product measure of $\mu_N$ and Lebesgue.
\end{rem}
\begin{thm}
\label{thm:high-temp} Assume (\ref{eq:variance-bd}) and let $\beta\leq\beta_{c}$.
Suppose the $\P$-a.s.\ $\cF_{N}$-measurable random field $f_{N}(\bs)$ and
Borel $A\subset\R$ are such that $p_N(\bs,E)$ of \eqref{eq:reg-cond} is measurable on
$\Sigma_N \times \R$.  If 
\begin{equation}
\lim_{\epsilon\to0}\varlimsup_{N\to\infty}\frac{1}{N}\log\int_{\Sigma_{N}}\int_{\beta-\epsilon}^{\beta+\epsilon}
%e^{-\frac{1}{2}N(\beta-E)^{2}}
\P\big(f_{N}(\bs)\in A\,\big|\,\frac{1}{N}H_{N}(\bs)=E \big) dE d\mu_{N}(\bs)<0,\label{eq:bd_Gaussian_average}
\end{equation}
then 
\begin{equation}\label{eq:ht-result}
\varlimsup_{N\to\infty}\frac{1}{N}\log\E \Big[ G_{N,\beta}\left(f_{N}(\bs)\in A\right) \Big] <0.
\end{equation}
\end{thm}

Of course, (\ref{eq:bd_Gaussian_average}) follows whenever 
\begin{equation}\label{eq:HighTempCondUnif}
\lim_{\epsilon\to0}\varlimsup_{N\to\infty}\frac{1}{N}\log 
\supess_{\bs\in\Sigma_{N},\,|E-\beta|<\epsilon}\P\Big(f_{N}(\bs)\in A\,\Big|\,\frac{1}{N}H_{N}(\bs)=E\Big)<0\,.
\end{equation}

\begin{rem} Often the probabilities in \eqref{eq:bd_Gaussian_average} do not depend on 
$\bs\in\Sigma_{N}$ so it suffices to verify the relevant exponential decay only
at one fixed,  convenient choice of $\bs$. For example, 
for the mixed $p$-spin models that we discuss in the sequel,  with
either Ising or spherical spins,  the variance $\E[H_{N}(\bs)^{2}]$ is constant on $\Sigma_{N}$ 
and the free energy converges \cite{GuerraToninelli,FreeEnergyConvergence}.  
In the absence of an external field,  $\beta_{c}>0$ is the inverse-temperature where
replica symmetry is broken (excluding the special case of the spherical $2$-spin),
and $H_{N}(\bs)$ is invariant under rotations.  Thus,  if $f_{N}(\bs)$ is such that $\bs\mapsto(H_{N}(\bs),f_{N}(\bs))$
is also invariant,  for example as for the Franz-Parisi potential we
consider in Section \ref{subsec:FPpotential},  then the
conditional probabilities in (\ref{eq:bd_Gaussian_average}) are indeed constant
over $\Sigma_{N}$.
\end{rem}

The proof of Theorem \ref{thm:high-temp} in Section \ref{sec:HighTempPf},  is  
under the above general setting.  However,  understanding the Gibbs measure for large $\beta$
is an extremely challenging problem even for specific models,  so in the rest of the 
paper we always assume that $G_{N,\beta}$ corresponds to some spherical spin glass model.

\subsection{\label{subsec:LowTemp}Low temperature phase: spherical spin glasses}

As mentioned above,  we now restrict our attention to the spherical mixed $p$-spin
models, for which we are able to exploit results relating
the Gibbs measure to critical points \cite{ABA2,A-BA-C,geometryMixed,2nd,
geometryGibbs,pspinext,2ndarbitraryenergy}
and the generalized Thouless-Anderson-Palmer 
(\abbr{TAP}) approach \cite{TAPChenPanchenkoSubag,TAPIIChenPanchenkoSubag,FElandscape,TAP}.
Specifically,  we take the configuration space 
% of the spherical spin glass models is
\begin{equation*}
\Sigma_{N}=\SN:=\Big\{\bs=(\sigma_{1},\ldots,\sigma_{N})\in\R^{N}:\,\|\bs\|=\sqrt{N}\,\Big\}\label{eq:SN}
\end{equation*}
with its Borel $\sigma$-algebra $\cF_{N}$ and uniform measure
$\mu_{N}$.  For deterministic $\gamma_{p}\geq 0$,  we have the corresponding \emph{mixture}
\begin{equation}\label{def:xi}
\xi(t):=\sum_{p\geq 2} \gamma_{p}^{2}t^{p} \,,
\end{equation}
and we assume hereafter 
that 
$\xi(1+\epsilon)<\infty$ for some $\epsilon=\epsilon_\xi>0$. 
The model with $\xi(t)=t^p$ is called the pure $p$-spin model,  
whereas a mixture $\xi(\cdot)$ is called generic if
	\begin{equation*}
		\sum_{p\,\text{odd}}p^{-1}\indic\{\gamma_{p}>0\}=\sum_{p\,\text{even}}p^{-1}\indic\{\gamma_{p}>0\}=\infty.\label{eq:generic}
	\end{equation*}
The mixed $p$-spin Hamiltonian on $\SN$ corresponding to such $\xi$ is given by
\begin{equation}
H_{N}(\bs)=\sum_{p=2}^{\infty}\gamma_{p}N^{-\frac{p-1}{2}}\sum_{i_{1},\dots,i_{p}=1}^{N}J_{i_{1},\dots,i_{p}}\sigma_{i_{1}}\cdots\sigma_{i_{p}},\label{eq:Hamiltonian}
\end{equation}
where $J_{i_{1},\dots,i_{p}}$ are \abbr{iid} standard normal variables,  yielding the covariance function
\[
\E H_{N}(\bs)H_{N}(\bs')=N\xi(\langle\bs,\bs'\rangle/N),
\]
where $\frac{1}{N}\langle\bs,\bs'\rangle:=\frac{1}{N}\sum_{i\leq N}\sigma_{i}\sigma_{i}'$
is called the overlap of $\bs$ and $\bs'$.  

The free energy of such spherical models is given by the Parisi formula
\cite{ParisiFormula,Parisi},
\begin{equation*}
\lim_{N\to\infty}\E F_{N,\beta}=F(\beta)=\min_{x}\mathcal{P}_{\xi,\beta}(x)\label{eq:Parisi}
\end{equation*}
rigorously proved by Talagrand \cite{Talag}  and Chen \cite{Chen}. Here the minimum is taken
over all $[0,1]\to[0,1]$ distribution functions (i.e., non-decreasing and right-continuous) such that $x(\hat q)=1$ for some $\hat q<1$ and 
the Crisanti-Sommers \cite{Crisanti1992} functional $\mathcal{P}_{\xi,\beta}(x)$ is given by
\begin{equation}\label{eq:CrisantiSommers}
\mathcal{P}_{\xi,\beta}(x)=\frac12\Big(
\beta^2\int_0^1x(t)\xi'(t)dt+\int_0^{\hat q}\frac{dt}{\int_t^1 x(s)ds}+\log(1-\hat q)
\Big).
\end{equation}
The unique minimizer of the strictly convex functional
(\ref{eq:CrisantiSommers}) is called the Parisi distribution and denoted by $x_{P}=x_{\beta,P}$. The corresponding measure $\nu_{P}=\nu_{\beta,P}$ given by $\nu_{P}([0,q])=x_{P}(q)$ is called the Parisi measure.

A zero-temperature analogue of the Parisi formula expressing the ground state energy as
\begin{align}
	\lim_{N\to\infty}\frac{1}{N}\E\Big[ \sup_{\bx\in\SN}\big\{ H_{N}(\bx)\big\} \Big] = \min_{\alpha,c}\mathcal{P}_{\xi,\infty}(\alpha,c)
	\label{eq:ZTParisi}
\end{align}
was also proved in \cite{ChenSen,JagannathTobascoLowTemp}. Here the minimum is taken over all non-decreasing, right-continuous and integrable functions $\alpha:[0,1)\to[0,\infty)$ and real,  positive values $c>0$, 
where 
\[
\mathcal{P}_{\xi,\infty}(\alpha,c) := \frac12\Big[
\xi'(1)c + \int_0^1 \xi'(t)\alpha(t)dt + \int_0^1\frac{dt}{\int_t^1 \alpha(s)ds+c}\,
\Big]\,.
\]

We denote the ground state energy at radius $\sqrt{Nq}>0$ and (twice) its derivative by 
\begin{align}
\Es(q) & :=\lim_{N\to\infty}\frac{1}{N}\E\Big[ \sup_{\bx\in\sqrt{q}\SN}\big\{ H_{N}(\bx)\big\} \Big],\qquad \hRs(q):=2\frac{d}{dq}\Es(q).
\label{eq:GS}
\end{align}
Denoting by $(\alpha_q,c_q)$ the minimizer of \eqref{eq:ZTParisi} for the mixture $\hat\xi_q(t):= \xi(qt)$, we have that
\begin{equation}\label{eq:hRTformula}
\Es(q) = \mathcal{P}_{\hat\xi_q,\infty}(\alpha_q,c_q),\qquad \hRs(q) = \frac{c_q}{q}(\hat\xi_q''(1)+\hat\xi_q'(1)) +\frac1q \int_0^1(s\hat\xi_q''(s)+\hat\xi_q'(s))\alpha_q(s)ds.
\end{equation}
The first equality follows from \eqref{eq:ZTParisi},  since up to space scaling (by $\sqrt{q}$),
% of the space, 
the restriction of $H_N(\bx)$ to $\sqrt{q}\SN$ is the Hamiltonian on $\SN$
for mixture $\hat\xi_q(t)$ (of the form $\sum_{p\geq2}\gamma_pq^{\frac p2}H_N^p(\bs)$ 
with independent pure $p$-spin Hamiltonians $H_N^p(\bs)$).
For the second equality,  fix $\lambda,q_1,q_2 \in [0,1]$ and note that since
$(\lambda q_1+(1-\lambda) q_2)^{\frac p2} \le \lambda q_1^{\frac p2}+(1-\lambda) q_2^{\frac p2}$
for all $p \ge 2$, it follows that
\begin{align*}
	\E \max_{\bs} \sum_{p\geq2}\gamma_p(\lambda q_1+(1-\lambda) q_2)^{\frac p2}H_N^p(\bs)
	& \leq
	\E \max_{\bs} \sum_{p\geq2}\gamma_p(\lambda q_1^{\frac p2}+(1-\lambda) q_2^{\frac p2})H_N^p(\bs)\\
	&\leq 
	\lambda \E \max_{\bs} \sum_{p\geq2}\gamma_p q_1^{\frac p2}H_N^p(\bs)+(1-\lambda)\E \max_{\bs} \sum_{p\geq2}\gamma_p q_2^{\frac p2}H_N^p(\bs),
\end{align*}
from which we deduce that $q \mapsto \Es(q)$ is convex.
Further,  comparing 
% by the Parisi formula 
 \eqref{eq:ZTParisi} with the \abbr{lhs} of \eqref{eq:GS} we see that
 %for any $q'$, 
 $\Es(q')\leq \mathcal{P}_{\hat\xi_{q'},\infty}(\alpha_q,c_q)$,  with equality if $q'=q$ (by the first equality
 of \eqref{eq:hRTformula}).
% while for $q'=q$ we have equality. 
Thus,\footnote{Applying with $f(q')=\Es(q')$, $g(q')=\mathcal{P}_{\hat\xi_{q'},\infty}(\alpha_q,c_q)$ and $x_0=q$ the fact that if $f(x)\le g(x)$ are real functions with $f(x)$ convex, $g(x)$ differentiable at $x_0$ and $f(x_0)=g(x_0)$ then $f'(x_0)$ exists and is equal to $g'(x_0)$.}
 \[
 \hRs(q) = 2\frac{d}{dq}\Es(q)=2\Big[\frac{d}{dq}\mathcal{P}_{\hat\xi_q,\infty}(\alpha,c)\Big]\Big|_{(\alpha,c)=(\alpha_q,c_q)}
 \]
which coincides with the \abbr{rhs} of the second equality in \eqref{eq:hRTformula}.
Recall also that by \cite[Proposition 11]{FElandscape}, for any positive $q\in \mbox{supp}(\nu_{\beta,P})$,
\[
\alpha_q(t)=\beta x_P(qt)\quad\mbox{and}\quad c_q=\frac{\beta}{q}\int_q^1 x_P(s)ds
\]
with $\Es(q)$ and $\Rs(q)$ thus expressed directly in terms of the Parisi distribution 
$x_{P}$.

%We denote by $q_P=q_P(\beta)=\max\text{supp}(\nu_{P})<1$ the rightmost point
%in the support of the Parisi measure. Note that, as can be verified
%from the Parisi formula, $q_P>0$ if and only if $\beta>\beta_{c}$.
%We shall see (in Lemma \ref{lem:ddq}) that for $\beta>\beta_c$,
%\[
%\Rs(q_P)=\frac{1}{\b(1-q_P)}+\b(1-q_P)\xi''(q_P).
%\]
%Although the quantities above depend on $\beta$,  we often omit this dependence,  
%in order to lighten our notation.

Given a probability measure $\nu$ on $[0,1]$ with cumulative distribution $x(q)=\nu([0,q])$ such that $x(\hat q)=1$ for some $\hat q<1$, define the functions
\begin{equation}\label{eq:phi}
	\begin{aligned}
		\Phi_{\nu}(t)&=\beta^2 \xi'(t)-\int_0^t\frac{ds}{(\int^1_s x(r)dr)^2},\\
		\phi_{\nu}(s)&=\int^s_0\Phi_{\nu}(t)dt.
	\end{aligned}
\end{equation}
Talagrand \cite[Proposition 2.1]{Talag} proved that the Parisi measure $\nu_{\beta,P}$ is the unique measure such that
\begin{equation}\label{eq:Parisi_characterization_finite_beta}
	\mbox{supp}(\nu)\subset \mathcal{S}_{\nu}:=\Big\{ s\in[0,1]: \phi_{\nu}(s)=\sup_{t\in[0,1]}\phi_{\nu}(t) \Big\}.
\end{equation}
For the Parisi measure,  we denote $\mathcal{S}_{P}:=\mathcal{S}_{\nu_{P}}$ and
by $q_P:=q_P(\beta)=\max\text{supp}(\nu_{P})<1$ the rightmost point in its support.  Note that, as can be verified
from the Parisi formula, $q_P>0$ if and only if $\beta>\beta_{c}$.   Recall that a model $\xi(\cdot)$ is 
$k$-\abbr{rsb} at $\b$ for some $k \ge 1$,  if the Parisi measure $\nu_P=\nu_{\b,P}$ is supported 
on $0$ and $k$ distinct positive points $q_i$.  Generalizing  \cite[Definition 4]{HuangSellke2023}
(beyond $k=1$ and zero-temperature),  we say that such $\xi(\cdot)$ is {\em strictly} $k$-\abbr{rsb} 
at $\b$ if there exist  $0=q_0 < q_1 <\cdots <q_k=q_P$ such that
\begin{equation}\label{eq:kRSB}
	\mathcal{S}_P\cap [0,q_P] = \mbox{supp}(\nu_P)=\{q_0,\ldots,q_k\} 
\end{equation}
(while we do not rely on it,  we note in passing that for generic mixtures
$\mathcal{S}_P$ is the set of all multi-samplable overlaps,  as shown in 
\cite[Theorem 10]{FElandscape}). 

The problem of computing the Parisi measure, and in particular identifying whether it is $k$-\abbr{RSB}, has been studied in \cite{AuffingerZhou25,JagannathTobascoBdsCplxSph,Zhou2025}.  If $\xi(\cdot)$ is known to be $k$-\abbr{rsb},  one can easily check whether it is also strictly $k$-\abbr{rsb} by using \eqref{eq:kRSB}. 
Generically,  one expects most $k$-\abbr{rsb} models to be strict,  but at phase transitions we 
can find $k$-\abbr{rsb} models which are not strictly $k$-\abbr{rsb}.  For example, if $\xi''(0)=0$ then at the critical $\beta_c$ the maximal $q \in \mathcal{S}_P$ is strictly positive while $\nu_P=\delta_0$; see \cite[Remark 13]{FElandscape} for a proof of this fact and a related discussion on the special case of the $2$-spin model.
% the strict $k$-\abbr{rsb} condition is equivalent to having no multi

Assume henceforth that our mixture $\xi(t)$ is $k$-\abbr{rsb} for some finite $k$. 
Denoting the Euclidean gradient of $H_{N}(\bx)$
in $\R^{N}$ by $\nabla H_{N}(\bx)$, we define directional derivatives by
\[
\partial_{\bu} H_{N}(\bx):=\left\langle \nabla H_{N}(\bx),\bu\right\rangle \,,
\]
and with an arbitrary matrix $M_i=M_i(\vec\bx)\in\R^{(N-i)\times N}$
whose rows form an orthonormal basis of $\big(\sp\{\bx_1,\ldots,\bx_i\}\big)^{\perp}$,
we define $\gradt H_{N}(\bx_i):=M_{i}\nabla H_{N}(\bx_i)$.  Note that
$\gradt H_{N}(\bx_i)$ depends also on the points $\bx_1,\ldots,\bx_{i-1}$ (whose choice will be clear from the context).  In particular,  denoting hereafter $a\wedge b=\min\{a,b\}$ and
$\{\be_i\}$ the standard basis of $\R^N$,  
in Theorem \ref{thm:low temp 1} we use the vectors $\vec\bx^\be_k=(\bx_1^{\be},\ldots,\bx_k^{\be})$ where
$\bx^{\be}_i := \sum_{j=1}^i \sqrt{N(q_j-q_{j-1})} \be_j$ are such that $\frac1N\langle\bx_i,\bx_j\rangle=q_{i\wedge j}$.

Given
two vectors $\vec E=(E_i)_{i=1}^k$ and $\vec R=(R_i)_{i=1}^k$ of real numbers, another real number $E$, a collection of points $\vec{\bx}=(\bx_i)_{i=1}^k$ in $\R^N$ such that\footnote{For such $\vec{\bx}$ to exist, we must have that $N\geq k$.   
	Henceforth, we always implicitly assume that indeed $N\geq k$.} $\frac1N\langle\bx_i,\bx_j\rangle=q_{i\wedge j}$  and a point $\bs\in\SN$, we set $\bx_0={\bf 0}$ and
let $\cpt(\bs,\vec{\bx},E,\vec E,\vec R)$ be
the event that 
\begin{equation}
	\forall i\leq k :\quad \Big(\frac{1}{N}H_{N}(\bs),\,\frac{1}{N}H_{N}(\bx_i),\,\frac{\partial_{\bx_i-\bx_{i-1}}H_{N}(\bx_i)}{\|\bx_i-\bx_{i-1}\|^2},\,\gradt H_{N}(\bx_i)\Big)=\big(E,E_i,R_i,{\bf 0}\big).  \label{eq:cpt}
\end{equation}

With $\BN\subset\R^N$ denoting the ball of radius $\sqrt {N r}$ for some $1 \le r < 1 + \epsilon_\xi$  
and $C^\infty(\BN)$ the class of real smooth functions on it,  the statistics we wish to 
study under the Gibbs measure  are of the form
$f_N(\bs)=\bar f_N(\bs,H_N(\cdot))$ for some deterministic 
$\bar f_N(\bs,\varphi):\SN\times C^\infty(\BN)\to\R$ which is 
rotationally invariant.   That is,
$\bar f_N(\bs,\varphi(\cdot)) = \bar f_N(O\bs,\varphi(O^{T}\cdot))$ for any orthogonal $O$. 
We further assume that
% the following continuity property of $\bar f_N$:
for any $\bs\in\SN$,  $\varphi,\psi_i\in C^\infty(\BN)$ and $k\geq1$,
\begin{equation}
	\lim_{t_{i}\to0}\bar f_{N}\Big(\bs,\varphi+\sum_{i\leq k}t_{i}\psi_{i}\Big)=\bar f_{N}(\bs,\varphi).\label{eq:flim}
\end{equation}
We note in passing that
$\partial_{\bx_i-\bx_{i-1}}H_{N}(\bx_i)$ is measurable \abbr{wrt}  $H_{N}(\bx_i)$  
iff $i=1$ and the model is pure $p$-spin for some $p$ (see the covariance calculations in \cite[Appendix A]{geometryMixed}),
in which case we omit the condition
$\frac{\partial_{\bx_1}H_{N}(\bx_1)}{\|\bx_1\|^2}=R_1$ from the event $\cpt(\bs,\vec{\bx},E,\vec E,\vec R)$.
Henceforth,  \abbr{wlog} the Gaussian vector which determines
the event $\cpt(\bs,\vec{\bx}^\be_k,E,\vec E,\vec R)$ has a strictly positive density.  We thus 
define,  analogously to \eqref{eq:reg-cond},  the conditional probability
\[
p_N(\bs,(E,\vec E,\vec R)) := \P\big(\, f_{N}(\bs)\in  A\,\Big|\,\cpt(\bs,\vec{\bx}_k^{\be},E,\vec E,\vec R)\,\big) \,,
\]
as the density with respect to the underlying Lebesgue measure of the corresponding sub-probability 
measure,  normalized by the relevant Gaussian vector density at $(E,\vec E,\vec R,\bf 0)$ and in particular 
to consider the (essential) supremum of $p_N(\cdot,\cdot)$ over
\begin{align}
\label{eq:Band}
B(\vec \bx,\delta)&:=\left\{ \bs\in\SN:\,|\langle\bs,\bx_i\rangle-\langle\bx_i,\bx_i\rangle|\leq N\delta,\,\;\; \forall 1\leq i\leq k \right\}\,,	\\
\label{eq:V}
V(\epsilon)& :=\left\{ (E,\vec E,\vec R):\,|E-F'(\beta)|,\,| E_i-\Es(q_i)|,\,|R_i-\Rs(q_i)|<\epsilon,\,\;\; \forall 1\leq i\leq k\right\} \,.
\end{align}

\begin{thm}
\label{thm:low temp 1} Consider the spherical mixed $p$-spin model
with a generic mixture $\xi(t)$ and $\beta>\beta_{c}$ at which $\xi(t)$ is strictly $k$-\abbr{rsb}.
% such that \eqref{eq:kRSB} holds.
Suppose that $f_N(\bs)=\bar f_N(\bs,H_N(\cdot))$ for some deterministic rotationally invariant $\bar f_N(\bs,\varphi)$ which satisfies the continuity property \eqref{eq:flim} and that $f_N(\bs)$ is $\P$-a.s.\  a measurable function on $\SN$.  If for some Borel 
%measurable 
$A\subset\R$
\begin{equation}
\lim_{\epsilon\to0}\varlimsup_{N\to\infty}\supess_{B(\vec\bx_k^{\be},\epsilon)\times V(\epsilon)}\frac{1}{N}\log\P\big(\, f_{N}(\bs)\in  A\,\Big|\,\cpt(\bs,\vec{\bx}_k^{\be},E,\vec E,\vec R)\,\big)<0
\label{eq:bd_Gaussian_average-11}
\end{equation}
%(where the conditional probabilities are defined similarly to Remark \ref{rem:conditional}),
then 
\begin{equation}
\lim_{N\to\infty}\E G_{N,\beta}\left(f_{N}(\bs)\in A\right)=0.\label{eq:GAverageLowT}
\end{equation}
\end{thm}
Note that $\vec \bx_k^{\be}$,  the nominal values of  
$(E,\vec E,\vec R)$ in \eqref{eq:bd_Gaussian_average-11},  and the overlaps
$\{ \langle \bs,\bx^{\be}_i\rangle,  i \le k\}$ that determine the narrow band $B(\vec \bx_k^{\be},\epsilon)$
of \eqref{eq:Band},  are explicit functions of the Parisi measure $\nu_{\beta,P}$.

\begin{rem} It is necessary to have an extra conditioning in Theorem \ref{thm:low temp 1},
beyond conditioning only on  $\frac1NH_N(\bs)= E\approx F'(\beta)$,  as in Theorem \ref{thm:high-temp} 
for high temperature.  For instance  
consider the observable $f_N(\bs) = \|M_0 \nabla H_{N}\|/\sqrt{N}$ for $M_0=\mathbf{I}_N-\frac1N\bs\bs^T$ (projecting the gradient to the orthogonal space).  Conditional only on $\frac1NH_N(\bs)= E$, 
we have that $f_N(\bs)\overset{{\rm d}}{=} \sqrt{\xi'(1)} \|X\|/\sqrt{N}$ for 
standard Gaussian $X \in \R^{N-1}$,  hence 
$\P( |f_N(\bs) - \sqrt{\xi'(1)}| > \delta | \frac1NH_N(\bs)= E) = e^{-\Omega(N)}$.
On the other hand,  by a covariance computation similar to that of Section \ref{subsec-cond-law},  one can verify that 
for some $U_\beta\neq\sqrt{\xi'(1)}$ and almost every $\beta$,  our condition \eqref{eq:bd_Gaussian_average-11} holds for any open set $A$ not containing $U_\beta$.  Even without any calculation, the reader could easily convince themselves that for $\beta$ very large so that $F'(\beta)$ is close to $\Es(1)$,  such a constant $U_\beta$ must be close to zero.
\end{rem}
\begin{rem}
		The conclusion \eqref{eq:GAverageLowT} provides an $o(1)$ bound, weaker than the exponential decay \eqref{eq:ht-result} in the high-temperature case. The reason is that we rely on the pure state decomposition 
		\cite{	JagannathApxUlt,TalagrandPstates} that controls the Gibbs measure up only over a set of measure $1-o(1)$. Strengthening the latter decomposition to cover $1-e^{-\Theta(N)}$ of the measure would result in a similar strengthening for \eqref{eq:GAverageLowT}, by a straightforward modification of our proof.
	\end{rem}
\begin{rem}\label{rem:pure}
	The pure $p$-spin models with mixture $\xi(t)=t^p$,  $p\geq3$ are strictly $1$-\abbr{rsb}  
	% satisfy \eqref{eq:kRSB} with $k=1$,  
	whenever $\beta>\beta_c$.  While they are not generic,  one can derive
	Theorem \ref{thm:low temp 1} for such pure models at all large enough $\beta$.
	Indeed,  genericity is used in the proof of the theorem mainly to invoke the pure states decomposition of Proposition \ref{prop:PS},  but for
the pure models such a decomposition 
% for the Gibbs measure
 follows at large $\beta$ 
	from the decomposition of \cite{geometryGibbs}.  
	In the generic 
	%non-pure 
	case the collection of Gaussian random variables in \eqref{eq:cpt} is non-degenerate 
	(namely,
	they are not linearly dependent).  In contrast,  for pure models
	%the pure case,  for which  $k=1$,  
	the radial derivative $\partial_{\bx_1}H_{N}(\bx_1)/\|\bx_1\|^2$ is measurable \abbr{wrt} $H_{N}(\bx_1)$
	and,
	% . In this case, 
	as noted after \eqref{eq:flim}, we omit the condition $\partial_{\bx_1}H_{N}(\bx_1)/\|\bx_1\|^2=R_1$ from the definition of $\cpt(\bs,\vec{\bx}^\be_1,E,\vec E,\vec R)$. This entails several technical modifications to the proof of Theorem \ref{thm:low temp 1}, which are straightforward and left to the interested reader. 
	In Theorem \ref{thm:low temp pure} below, we only state the analogous result for the pure $p$-spin models, denoting the corresponding event by $\cpt(\bs,\vec{\bx}_1^{\be},E,\vec E)$.
	Specifically,  in the pure case,  since $k=1$ the required complexity computations as in Section \ref{sec:complexity} can be taken directly from \cite{A-BA-C},
% Moreover,  
the multi-level ($m\geq1$ or $n\geq2$) Kac-Rice formulas of Section \ref{sec:KacRice} are 
%also 
not required,
%in the pure case,  
while 
the conditional laws of  Lemma \ref{lem:conditional_law_on_band}
% which 
are required only for $m=1$ and thus derived
%are given by the same formulas as the lemma 
by a trivial modification of the proof for the generic case.
\end{rem}

\begin{thm}
	\label{thm:low temp pure} Consider the spherical pure $p$-spin model 
	with mixture $\xi(t)=t^p$ and suppose that $f_N(\bs)=\bar f_N(\bs,H_N(\cdot))$ is as in Theorem \ref{thm:low temp 1}.
	For any $p\ge3$ there exists $\beta_p$ such that if $\beta>\beta_p$ and for some Borel 
	%measurable 
	$A\subset\R$
	\begin{equation*}
		\lim_{\epsilon\to0}\varlimsup_{N\to\infty}\supess_{B(\vec\bx_1^{\be},\epsilon)\times V(\epsilon)}\frac{1}{N}\log\P\big(\, f_{N}(\bs)\in  A\,\Big|\,\cpt(\bs,\vec{\bx}_1^{\be},E,\vec E)\,\big)<0\,,
	\end{equation*}
	then 
	\begin{equation*}
		\lim_{N\to\infty}\E G_{N,\beta}\left(f_{N}(\bs)\in A\right)=0.
	\end{equation*}
\end{thm}

In Section \ref{sec:appl} we illustrate the wide applicability of our main results
in the context of spherical mixed $p$-spin models,  by considering Langevin dynamics initialized
according to a Gibbs measure,  and by computing the Franz-Parisi potential
\cite{FranzParisiPotential},  including at the challenging low-temperature phase.
In Section \ref{sec:putestates} we state the pure states decomposition of \cite{JagannathApxUlt,TalagrandPstates} and our Proposition \ref{prop:PS}, by which its vertices can be chosen to be critical points (on appropriate spheres).  In Section \ref{sec:PfSketch} we outline  the proof of Theorem \ref{thm:low temp 1} about the low-temperature phase of spherical spin glasses. Section \ref{sec:HighTempPf} is devoted to the proof of Theorem \ref{thm:high-temp} about the high temperature phase. In Section \ref{sec:complexity} we collect several results from \cite{geometryMixed,HuangSellke2023,FElandscape} about the complexity of critical points and prove some consequences of theirs.  In Section \ref{sec:KacRice} we use the 
multi-level Kac-Rice formula to express averages over critical points by certain complexities and the corresponding conditional expectations and probabilities given $\cpt(\bs,\vec{\bx}_n^{\be},E,\vec E,\vec R)$.
In Section \ref{sec-pf-prop:PS} we prove the aforementioned Proposition \ref{prop:PS} on the pure states decomposition.  Finally,  in Section \ref{sec:LowTempPf} we combine the results of Sections \ref{sec:putestates}, \ref{sec:complexity} and \ref{sec:KacRice} to prove Theorem \ref{thm:low temp 1} about the low temperature phase
and in Appendix \ref{sec:Appendix} extend it to get Corollary \ref{cor:low temp PS} 
(which we use in our study of Langevin dynamics).
%in \cite{DS2}.

{\bf Acknowledgment}
We thank the anonymous referees for their comments, which significantly improved
the presentation of this paper.
This research was funded in part by \abbr{NSF} grant DMS-2348142 (A.D.),  
by \abbr{ISF} grant 2055/21 (E.S.), \abbr{ERC} grant PolySpin 101165541 (E.S.) and by a research grant from the Center for Scientific Excellence
at the Weizmann Institute of Science (E.S.). E.S. is the incumbent of the Skirball Chair in New Scientists.

% \vspace{2cm}

\section{\label{sec:appl} Applications}
Section \ref{subsec:Langevin} shows how 
Theorem \ref{thm:high-temp} and Corollary \ref{cor:low temp PS} 
are utilized
to convert the thermodynamic limits of Langevin dynamic \eqref{diffusion} 
initialized according to a Gibbs measure,  into such limits for non-random initial conditions
with a suitably modified Gaussian disorder,  while in Section \ref{subsec:FPpotential} we apply 
Theorems \ref{thm:high-temp} and \ref{thm:low temp 1}, to compute the Franz-Parisi potential 
at any temperature.

\subsection{\label{subsec:Langevin} Langevin dynamics} For mixture $\xi(t)$ of 
 \eqref{def:xi} such that $\xi(1+\epsilon)<\infty$,  consider the Langevin dynamic \eqref{diffusion} at 
some $\widetilde\beta>0$,  for $t \in [0,T]$ and $\bx_0=\bs$ drawn from $G_{N,\beta}$
for some $\beta>0$.
%,  possibly $\beta \ne 2 \beta_\star$).  
As shown in \cite[Corollary 4.6]{DS2},
having $1<\rho<1+\epsilon$ with $b_\ell'(\cdot)$ locally Lipschitz on $[0,\rho)$ and
$(\rho - r) b_\ell'(r) \to \infty$ as $r \uparrow \rho$,  guarantees the existence of unique
strong solutions of \eqref{diffusion} 
in $C(\R^+,\R^N)$,  with  $C_N(t,t) < \rho$.  
%(and the dynamics \eqref{diffusion} is then reversible \abbr{wrt} $G_{N,2\beta_\star}$)
Fixing hereafter such $b_\ell(\cdot)$ and $\xi(\cdot)$,  our approach applies for any generic mixture
and $\beta>\beta_c$ at which $\xi(t)$ is strictly $k$-\abbr{rsb},  but for clarity we restrict 
hereafter to $k=1$,  in which case we amend $(C_N,\chi_N,\widetilde{H}_N)$ of \eqref{eq:cov} by
\begin{equation}\label{eq:bnq} 
q_N(s):=\frac{1}{N} \langle \bx_s, \bx_\star \rangle\,,
\end{equation} 
for some $\bx_\star\in \sqrt{q_1} \SN$ chosen in the sequel (in the $k$-\abbr{rsb} case 
one should amend $(C_N,\chi_N,\widetilde{H}_N)$ by $k$ such functions).  Also,  aiming at a more uniform 
presentation,  we set $\bx_\star:={\bf 0}$ whenever $\beta \le \beta_c$.  Then,  denoting by $\|U_N\|_T$ the sup-norm 
of a generic function $U_N : [0,T]^2 \to \R$,  for $N$,  $T$,  $H_N(\cdot)$ and a strong 
solution of \eqref{diffusion} starting at $\bx_0=\bs \in \SN$,  we set the error between 
${\bf U}_N:=(C_N,\chi_N, \widetilde{H}_N,q_N)$ and 
a proposed non-random limit ${\bf U}_\infty=(C,\chi,H,q)$,  to be 
\begin{align*}
% f_N(\bs) :=
 \bar f_{N} (\bs,\bx_\star,H_N(\cdot) ) &:= 
\E
%_{\bs,H_N} 
\Big[  \|C_N-C\|_T  \wedge 1 +\|\chi_N-\chi\|_T \wedge 1 
+\| \widetilde{H}_N -  H\|_T \wedge 1 + \|q_N-q\|_T \wedge 1 \Big] \,,
%\label{eq:err} 
\end{align*}
where the expectation
% in \eqref{eq:err} 
is only over the Brownian motion $\BB_t$ (and the induced solution of \eqref{diffusion}).
Setting $A=(\delta,\infty)$ we utilize Theorem \ref{thm:high-temp} when $\beta \le \beta_c$ and 
Corollary \ref{cor:low temp PS} when $\beta>\beta_c$,  to conclude in \cite[Theorem 1.8]{DS2} 
that $\bar f_N(\bs,\bx_\star,H_N(\cdot)) \to 0$ in probability,  for some 
explicit ${\bf U}_\infty$ which is given in \cite[Definitions 1.5 \& 1.1]{DS2},  respectively.

Indeed,  for $\beta \le \beta_c$ the
%,  the measure $G_{N,\beta}$ satisfies the 
high-temperature condition \eqref{eq:bcdef} holds,  with 
%the conditional probability 
$p_N(\bs,E)$ of \eqref{eq:reg-cond} measurable on $\SN \times \R$ and 
the exponential convergence condition \eqref{eq:HighTempCondUnif} is established in \cite[(2.3)]{DS2}.

Turning to $\beta>\beta_c$,  as noted in \cite[(1.26)]{DS2} the functional $\bar f_N(\bs,\bx_\star,\varphi_N(\cdot))$ 
is 
%deterministic 
rotationally invariant,  its continuity in $\varphi_N$ is established in \cite[Lemma 2.4]{DS2}
and its continuity in $\bx_\star$ is obvious.  Further,  for generic $\xi(t)$ which is strictly 1-\abbr{rsb} 
at $\beta>\beta_c$ we have the pure state decomposition $(\bx^\star_i,B_i)_{i \le d_N}$ of Proposition \ref{prop:PS}.
Fixing non-random $\bx^\star_0 \in \sqrt{q_1} \SN$ and setting $B_0:= \SN \setminus \cup_i B_i$, 
the partition $(B_i)_{i=0}^{d_N}$ induces the measurable mapping 
$\bx_\star (\bs,H_N)=\bx^\star_i \in \sqrt{q_1} \SN$.  With $G_{N,\beta}(B_0) \to 0$, 
we have thanks to \eqref{eq:energyontree} and properties (2),  (5) of Proposition \ref{prop:PS} that 
with probability approaching one $\bs\in B(\bx_\star,\epsilon_N)$ and 
$\bx_\star \in \sqrt{q_1}\SN$ such that $\gradt H_{N}(\bx_\star)=0$ and
\[
\Big(\frac{1}{N}H_{N}(\bs),\,\frac{1}{N}H_{N}(\bx_\star),\,\frac{\partial_{\bx_\star}H_{N}(\bx_\star)}{Nq_1}
\, \Big)\in V(\epsilon_{N})
\] 
for some  $\epsilon_{N} \to 0$,  with the exponential
convergence condition \eqref{eq:bd_Gaussian_average-PS} provided yet again by \cite[(2.3)]{DS2}.  

The thermodynamic limits for ${\bf U}_N$ are derived in \cite{BDG1} (for $\bx_0$ of \abbr{iid} components),
and in \cite{DemboSubag2020} (with a conditional disorder,  albeit not the one required here),  only for
finite mixtures and without an exponential in $N$ convergence rate.  Consequently,  
\cite[Sections 4 \& 5]{DS2} must extend the proofs in \cite{BDG1,DemboSubag2020} to 
cover different affine conditioning on $H_N(\cdot)$,  as well as handling generic (infinite) mixtures 
and producing exponential in $N$ convergence rates.  The non-random limit 
${\bf U}_\infty$ for Gibbsian initial conditions is novel,  even \abbr{wrt} the physics literature.  The
derivation of \cite[(2.3)]{DS2} is therefore supplemented 
in \cite[Sections 1.2,  3 \& 6]{DS2} by various properties of ${\bf U}_\infty$ 
which are beyond our scope in this work.

\subsection{\label{subsec:FPpotential} Computing the Franz-Parisi potential}
The Franz-Parisi potential involves sampling a point $\bs$ from the Gibbs measure and calculating, 
at an exponential scale,  the Gibbs probability of sampling another point $\bs'$ whose overlap 
with $\bs$ is approximately $r$,  where $\bs$ and $\bs'$ are allowed to be sampled 
from the Gibbs measure at two different inverse-temperatures $\beta$ and $\beta'$ 
(see  \cite{FranzParisiPotential}).
In doing so,  there are two sources of randomness,  the disorder $H_N$  (chosen by law $\P$)
and $\bs$ which is chosen according to $G_{N,\beta}$,  while the
Gibbs probability to sample $\bs'$ with a given overlap with $\bs$ is
determined once we fix $H_N$ and $\bs$ (i.e., it is measurable \abbr{wrt} them).  
One then aims to understand the typical behavior of this variable
as a function of $r$,  which can be interpreted as the typical rate function for
the overlap $\langle \bs,\bs'\rangle/N$  (for $N \gg 1$,  see below). 
%\footnote{Note that, as random variables measure w.r.t. the disorder $H_N(\bs)$, this is different from the probability to sample $(\bs,\bs')$ with overlap roughly $r$ under $G_{N,\beta}\times G_{N,\beta'}$. }
Specifically,  the Franz-Parisi potential is the random variable 
\begin{equation}\label{eq:FP-def}
V_{N,\beta,\beta'}(r):=f_{N,r}(\bs)+\frac12\log(1-r^2)-F_{N,\beta'},
\end{equation}
with $\bs$ a random point sampled from $G_{N,\beta}$ and
\begin{equation}\label{eq:fFP}
		f_{N,r}(\bs):=\frac{1}{N}\log\int_{S(\bs,r)}e^{\beta' H_N(\bs')}d\bs',
\end{equation}
where $d\bs'$ denotes integration \abbr{wrt} the uniform measure on 
%the sphere $S(\bs,r)$. 
\[
S(\bs,r):=\big\{\bs'\in \SN:\,\langle \bs',\bs \rangle/N=r \big\},
\]
and the logarithmic term in \eqref{eq:FP-def}
is the limit of $\frac1N\log$ of the Hausdorff measure of $S(\bs,r)$. We note in passing that   
since $H_N(\cdot)$ is $O(\sqrt N)$-Lipschitz with overwhelmingly high probability (see, e.g.\ \cite[Lemma 6.1]{TAPChenPanchenkoSubag} or \cite[Corollary C.2]{geometryMixed}),  one also has that
\begin{equation}\label{eq:Vpotential}
V_{N,\beta,\beta'}(r) = \frac1N\log G_{N,\beta'}\Big(\big\{\bar \bs:\, |\langle \bar \bs,\bs \rangle/N -r|<\epsilon \big\} 
\Big)+ o_{N,\epsilon}(1),
\end{equation}
where $o_{N,\epsilon}(1)$ denotes a random variable which is smaller than $\delta$ with probability at least $1-e^{-N\kappa}$ provided that $N>N_0$ and $\epsilon<\epsilon_0$, for some $\kappa=\kappa(\xi)>0$, 
$N_0=N_0(\xi,\delta)$ and $\epsilon_0=\epsilon_0(\xi,\delta)$, uniformly in $r\in[-r_0,r_0]$ for any fixed $r_0\in(0,1)$.  To see that, let $P_\perp$ be the projection to the orthogonal space to $\bs$ and note if we project any point  
$\bar \bs$ from the set in \eqref{eq:Vpotential} to 
\begin{equation}\label{eq:sigmabar}
	\bs'=\sqrt{N(1-r^2)}P_\perp(\bar \bs)/\|P_\perp(\bar \bs)\|+r\bs\in S(\bs,r),
\end{equation}
the distance $\|\bs'-\bar\bs\|$ is bounded by $\epsilon\sqrt{N}C(r_0)$ for a constant $C(r_0)$ depending on $r_0$, assuming that $\epsilon<(1-r_0)/2$. Using this and the co-area formula, the Gibbs probability can be related to $V_{N,\beta,\beta'}(r)$ to prove \eqref{eq:Vpotential}.

\noindent
{\bf I.  High temperature phase} We start with the analysis in the easier case of $\beta\leq \beta_c$. Define the mixture
\begin{equation}\label{eq:xiqtilde}
	\tilde \xi_q(t):=\xi(q+(1-q)t)-\xi(q)
\end{equation}
and denote by $\tilde F(\beta,r)$ the limiting free energy of the spherical model with mixture $\tilde \xi_{r^2}(t)$ at inverse-temperature $\beta$,  as defined in \eqref{eq:Fconvergence}.
\begin{cor}\label{cor:FP_hightemp} For any spherical model $H_N(\bs)$,  
	$\beta\leq\beta_c$,   $|r| < 1$,  $\beta', \delta>0$ and $f_{N,r}(\bs)$ as in \eqref{eq:fFP},
	\[
	\varlimsup_{N\to\infty}\frac{1}{N}\log\E G_{N,\beta}\Big(
	\big| f_{N,r}(\bs)- \beta\beta'\frac{\xi(r)}{\xi(1)}-\tilde F(\beta',r) \big| > \delta \Big)<0.
	\]
\end{cor}
A similar result to Corollary \ref{cor:FP_hightemp} was recently proved in \cite{alaoui2023shattering} for the spherical models and \cite{AuffingerElAlaouiSellke_shattering} for Ising models,  both using a technique from \cite{AchlioptasCoja-Oghlan} (while stated there only
for $\beta=\beta'$ and,  in \cite{alaoui2023shattering},  for pure models,  their argument works similarly in the setting below). 
\begin{proof}

As $F_{N,\beta}$ concentrates around its mean,  which is given by the Parisi formula,  for understanding
$V_{N,\beta,\beta'}(r)$ it suffices to analyze $f_{N,r}(\bs)$.
Further,  since $S(\bs,r)$ is a sphere of co-dimension 1,  we can think of the restriction of $H_N(\bs')$ to it,  as a new spherical Hamiltonian.  More precisely,  by the well-known formulas for the conditional law of jointly Gaussian variables  one can easily check that,  conditionally on $\frac{1}{N}H_{N}(\bs)=E$,  for $\bs' \in S(\bs,r)$,
\[
H_N(\bs') = \frac{\xi(r)}{\xi(1)}NE+\tilde H_N(\bs'),
\]
where $\tilde H_N(\bs')$ is a centered Gaussian process with covariance
\begin{align}
\E \big[
\tilde H_N(\bs_1')\tilde H_N(\bs_2')
\big]  &= N\xi\Big({\textstyle\frac1N} \langle \bs_1', \bs_2' \rangle \Big) -N\frac{\xi(r)^2}{\xi(1)}
\nonumber \\
&= N \tilde \xi_{r^2}\Big( \frac{\langle \bs_1'-r\bs,  \bs_2'-r\bs\rangle}{\|\bs_1'-r\bs\|\|\bs_2'-r\bs\|}\Big)+N\xi(r^2)-N\frac{\xi(r)^2}{\xi(1)}\,.
\label{eq:tilde-Hn}
\end{align}
Note that $\tilde H_N(\bs')$ has the same law as $\hat H_N(\bs')+\sqrt{N} X$,   for
a centered Gaussian variable $X$ of variance $\xi(r^2)-\xi(r)^2/\xi(1)\geq0$
which is independent of the centered Gaussian process $\hat H_N(\bs')$ on $S(\bs,r)$ whose 
covariance matches the left-most term in \eqref{eq:tilde-Hn}.
In particular,
$\tilde H_N(\bs')$ and $\hat H_N(\bs')$ have the same expected free energy. Further,  
by the preceding and
the concentration of the free energy (see \cite[Theorem 1.2]{PanchenkoBook}),  for any $\delta>0$, 
\begin{equation*}%\label{eq:HighTempCondUnif}
	\varlimsup_{\epsilon\to0}\varlimsup_{N\to\infty}\frac{1}{N}\log\sup_{\bs\in\Sigma_{N},\,|E-\beta|<\epsilon}\P\left( \big| f_{N,r}(\bs)-\beta\beta'\frac{\xi(r)}{\xi(1)}-\tilde F(\beta',r)\big| > \delta \,\Big|\,\frac{1}{N}H_{N}(\bs)=E\right)<0.
\end{equation*}
The corollary follows as a direct consequence of Theorem \ref{thm:high-temp}. 
\end{proof}

\noindent
{\bf II.  Low temperature phase} 
%{\color{cyan} General remarks: 1. below I changed the notation to be consistent with the new Theorem \ref{thm:low temp 1}, replacing $\bar E,R,\bxs$ by $E_1,R_1,\bx_1^{\be}$, etc. 2. A similar argument to the one we discussed in some length is now used in Lemma \ref{lem:Hhat}. I cite the lemma below, and replace previous arguments by you/me.}
Here we treat $\beta>\beta_c$,  for which to the best of our knowledge there is no rigorous computation of the Franz-Parisi potential, for any spherical model.  To this end,  note that
$f_{N,r}(\bs)$ of \eqref{eq:fFP} is of the form $\bar f_N(\bs,H_N(\cdot))$ for a deterministic,  
rotationally invariant $\bar f_N$,  that satisfies \eqref{eq:flim}.  We shall thus use 
Theorem \ref{thm:low temp 1} and work with the event $\cpt(\bs,\vec{\bx}^{\be},E,\vec E,\vec R)$ of \eqref{eq:cpt},  while 
simplifying our presentation by assuming that $\xi(\cdot)$ is not pure.  The straightforward adaptation 
% of our arguments 
to the pure case (and $\beta$ large enough so that a pure states decomposition exists),
by omitting the conditioning on $\partial_{\bx_1}H_{N}(\bx_1)$, is left to the reader 
(see Remark \ref{rem:pure}).
We note in passing that in the final formula \eqref{eq:Lambda},  this would
amount to removing the 3rd column and row from the matrix $C(\qs)$ below,  
as well as removing the 3rd entry from the vectors $\sfv$ and $\sfu$.  While for simplicity we 
consider only the case of strictly $1$-\abbr{RSB},
% condition of \eqref{eq:kRSB} with $k=1$,  
the same principles apply,  with obvious modifications,  to the more involved case of strictly 
$k$-\abbr{rsb} for $k\geq2$.
% which is not treated here.

Proceeding to evaluate the value around which $f_{N,r}(\bs)$ concentrates,  set
$\bx_1=\bx_1^{\be}=(\sqrt{Nq_1},0,\ldots,0)$,  and fixing the point
$\bs_1 =(\sqrt{N q_1},\sqrt{N(1-q_1)},0,\ldots,0)  \in B(\bx_1,0)$,  for $|\rho| \le q_1$,  let
\[
S(\bs_1,r;\bx_1,\rho) := \big\{\bs'\in \SN:\,\langle \bs',\bs_1 \rangle/N=r,\, \langle \bs',\bx_1 \rangle/N=\rho\big\}\,.
\]
The set $S(\bs_1,r;\bx_1,\rho)$  is non-empty if and only if 
\begin{equation}\label{eq:tau}
%	\tau_b = \tau_b(\qs,r,\rho):= \frac{\rho^2}{\qs}+ \frac{(r-b \rho)^2}{1-b^2 \qs}
%	=r^2 +\frac{(\rho- b r \qs)^2}{\qs-b^2\qs^2} \,,
\tau = \tau (q_1,r,\rho):= \frac{\rho^2}{q_1}+ \frac{(r-\rho)^2}{1- q_1}
=r^2 +\frac{(\rho- r q_1)^2}{q_1-q_1^2} \le 1\,,
\end{equation} 
in which case it is a sphere of 
%dimension $N-3$ and
radius $\sqrt {N(1-\tau)}$ 
in $\tilde{\bs}' := (\sigma'_3,\ldots,\sigma'_N)$,  centered at 
% $\bs_1$,  where 
\[
\bs_o:=\frac{r-\rho}{1- q_1}\bs_1 + \frac{\rho- rq_1}{q_1- q_1^2} \bx_1 \,.
% \bs_b:=\frac{r-b \rho}{1-b^2 \qs}\bs_\star + \frac{\rho-b r\qs}{\qs- b^2\qs^2} \bxs \,.
\]
Alternatively,  $S(\bs_1,r;\bx_1,\rho)$ is non-empty over the interval
\begin{equation}\label{eq:J}
	J=J(q_1,r):=\Big\{\rho:\,|\rho- rq_1|\leq \sqrt{q_1-q_1^2}\sqrt{1-r^2}\,\Big\}
	% J_b(\qs,r):=\Big\{\rho:\,|\rho-b r\qs|\leq \sqrt{\qs-b^2\qs^2}\sqrt{1-r^2}\,\Big\}.
\end{equation}
and consequently
\begin{equation}\label{eq:unionrho}
	S(\bs_1,r) = S(\bs_1,r;\bx_1,J),  \quad \text{for} \quad  
	S(\bs_1,r;\bx_1,I) := \bigcup_{\rho\in I} S(\bs_1,r;\bx_1,\rho).
\end{equation}
The event $\cpt_1:=\cpt(\bs_1,\bx_1,E, E_1, R_1)$ is the intersection of
\begin{align}
\label{eq:vec2}
\Big(\frac{\partial }{\partial x_3}H_{N}(\bx_1),\ldots,\frac{\partial }{\partial x_N}H_{N}(\bx_1)\Big)&=0\,,\\
\label{eq:vec1}
\bigg(\frac{1}{N} H_{N}(\bs_1),\,\frac{1}{N} H_{N}(\bx_1),\,
\frac{1}{\sqrt{N q_1}} \frac{\partial}{\partial x_1} H_{N}(\bx_1),\,
\frac{1}{\sqrt N} \frac{\partial }{\partial x_2} H_{N}(\bx_1)\bigg)& =\big(E,E_1, R_1,0\big) \,.
\end{align}
The covariances between these variables are easily computed (see the proof of Lemma \ref{lem:conditional_law_on_band} for a similar computation).  
Specifically,  the covariance matrix of the vector in \eqref{eq:vec2} is  $\xi'(q_1) \mathbf{I}_{N-2}$ (where $\mathbf{I}_{N-2}$ 
 denotes the identity matrix),  while the vector in  \eqref{eq:vec1} has the covariance matrix 
$N^{-1} C(q_1)$,  for
\[
C(q_1):=
\begin{pmatrix}
	\xi(1) & \xi(q_1) &  \xi'(q_1) & \sqrt{1-q_1}\xi'(q_1)\\
	\xi(q_1) & \xi(q_1) & \xi'(q_1) & 0\\
	 \xi'(q_1) & \xi'(q_1) & \xi''(q_1)+\xi'(q_1)/q_1 & 0\\
	\sqrt{1-q_1}\xi'(q_1) & 0 & 0 & \xi'(q_1)
\end{pmatrix},
\]
% and $C(q_1)$
 which is invertible (whenever $\xi(t)$ is not pure).\footnote{Since the variables \eqref{eq:vec1} are jointly Gaussian, $C(q_1)$ is invertible if no linear combination of them is deterministically zero.  Recall that 
 $0<q_1<1$,  so working directly with the formula \eqref{eq:Hamiltonian} for the Hamiltonian, one can easily verify that this is indeed the case for our choice of $\bs_1$ and $\bx_1$.}   Further, for any 
$\bs'\in S(\bs_1,r;\bx_1,\rho)$, 
the covariance of $H_N(\bs')$ and the vector in \eqref{eq:vec2} is $\xi'(\rho) \tilde{\bs}'$,
%where $\tilde{\bs}'$ is the vector obtained from $\bs'$ by removing its first two entries,  
whereas that
%the covariance 
of $H_N(\bs')$ and the vector in  \eqref{eq:vec1} is then 
\[
\sfv=\sfv(q_1,r,\rho):=\Big(\xi(r),\,\xi(\rho),\,\xi'(\rho)\rho/q_1,\,\xi'(\rho)\frac{r-\rho}{\sqrt{1-q_1}}\Big)\,.
\]

Since the vectors in \eqref{eq:vec2} and in \eqref{eq:vec1} are independent, 
we have that for any $\bs' \in S(\bs_1,r;\bx_1,\rho)$,
\begin{equation}\label{eq:conditionalexpectation}
\frac{1}{N} \E [ H_N(\bs') \big| \, \cpt_1 \, ] = \sfv(q_1,r,\rho)  C(q_1)^{-1} \big(E,E_1, R_1,0\big)^{\top}.
\end{equation}
Further,  conditionally on $\cpt_1$,  we have that 
$H_N(\cdot) = \tilde H_N(\cdot) + \E[H_{N}(\cdot)\,|\,\cpt_1]$ for a centered Gaussian process 
$\tilde H_N(\cdot)$ such that for any $\bs'_1,  \bs_2' \in S(\bs_1,r;\bx_1,\rho)$,
\[
\frac1N\E \big[ \tilde H_N(\bs'_1) \tilde H_N(\bs'_2) \big]  
= \xi(\tfrac1N\langle\bs'_1,\bs'_2\rangle) -\frac{\xi'(\rho)^2}{\xi'(q_1)}\frac{1}{N} \langle\tilde{\bs}'_1 ,\tilde{\bs}'_2 \rangle   - \sfv C(q_1)^{-1}\sfv^{\top} \,.
\]
Recall that then $\langle\bs'_1,\bs'_2\rangle =N \tau +  \langle\tilde{\bs}'_1 ,\tilde{\bs}'_2 \rangle$
for $\tau=\tau(q_1,r,\rho)
$ of \eqref{eq:tau},  and consequently
\begin{align}\label{eq:covcond}
% &\hspace{-.2cm}
\frac1N\E \big[ \tilde H_N(\bs'_1) \tilde H_N(\bs'_2) \big] 
% & = \xi\left(\tau+\tfrac1N\langle\tilde{\bs}'_1,\tilde{\bs}'_2\rangle\right) 
% -\frac{\xi'(\rho)^2}{\xi'(\qs)}\frac{1}{N} \langle\tilde{\bs}'_1 ,\tilde{\bs}'_2 \rangle - \sfv_1 C_1(\qs)^{-1}\sfv_1^{\top}
% \nonumber  \\
& = \xi_{q_1,\tau,\rho}\Big( \frac{\langle \tilde{\bs}'_1, \tilde{\bs}'_2 \rangle}{\|\tilde{\bs}'_1\| \|\tilde{\bs}'_2\|} \Big) 
+\xi\left( \tau \right)
- \sfv C(q_1)^{-1}\sfv^{\top},
\end{align}
so up to mapping $\bs \mapsto \tilde \bs$ and scaling,  
$\tilde H_N(\cdot)$ is the sum of the Hamiltonian for the mixture
\begin{equation}\label{def:xi-r-rho}
\xi_{q_1,\tau,\rho}(t) : =  \xi\big(\tau+(1-\tau)t\big) -\xi(\tau) - \frac{\xi'(\rho)^2}{\xi'(q_1)}(1-\tau)t 
\end{equation}
and an independent centered Gaussian variable of variance $N A$,  where
$A:=\xi\left( \tau \right)	- \sfv C(q_1)^{-1}\sfv^{\top}$.
\begin{rem}\label{rem:xiq1taurho}
Indeed,  $\xi_{q_1,\tau,\rho}(t)=\sum_{p} (\xi^{(p)}_{q_1,\tau,\rho}(0)/p!) t^p$ is a mixture,  as
%	Second, to see that indeed $\xi_{q_1,\tau,\rho}(t)=\sum_{p=0}^\infty \alpha_p t^p$ for $\alpha_0=0$ and some $\alpha_p\geq0$ for $p\geq1$, we note the following. For $p=0$, $\alpha_0 = 
$\xi_{q_1,\tau,\rho}(0)=0$ and
$
%\alpha_p= 
\xi^{(p)}_{q_1,\tau,\rho}(0)
%/p!
= (1-\tau)^p \xi^{(p)}(\tau) 
%/p!
\ge 0$
for all $p\geq2$,  while with $\sqrt{\tau q_1}\geq \rho$ we have that 
%Finally, for $p=1$, 
	\begin{equation*}
		% \alpha_1=
		 \xi'_{q_1,\tau,\rho}(0)= (1-\tau)(\xi'(\tau)-\xi'(\rho)^2/\xi'(q_1))
		 %\geq (1-\tau)(\xi'(\tau)-\xi'(\sqrt{\tau q_1})^2/\xi'(q_1))
		 \geq 0
	\end{equation*}
(since 
$\xi'(\cdot) \ge 0$ is non-decreasing and by Cauchy-Schwarz also
$\xi'(a)\xi'(b)\ge \xi'(\sqrt{ab})^2$ for all $a,b\in[0,1]$).
% since $\sqrt{\tau q_1}\geq \rho$ and for $a,b\in[0,1]$,
%	\begin{align*}
%	\xi'(a)\xi'(b)&=\sum_{i,j}\gamma^2_{i}\gamma^2_{j}a^{i-1}b^{j-1}=\frac12\sum_{i,j}\gamma^2_{i}\gamma^2_{j}(a^{i-1}b^{j-1}+a^{j-1}b^{i-1})\\
%	&\geq \frac12\sum_{i,j}\gamma^2_{i}\gamma^2_{j}2(a^{\frac{i-1+j-1}{2}}b^{\frac{i-1+j-1}{2}})= \xi'(\sqrt{ab})^2\,.
%	\end{align*}
Moreover,  $A \ge 0$ since there are $N-2 \gg 1$ points $\bx_i$ such that the entries of the 
non-negative definite covariance matrix of $\{N^{-1/2} \tilde H_N(\bx_i)\}$ 
are $\xi_{q_1,\tau,\rho}(1) {\bf 1}_{i=j}+A$.
\end{rem}
\begin{thm}\label{thm:FP-lt}
Consider the spherical mixed $p$-spin model
with a generic mixture $\xi(t)$,  for $\beta>\beta_{c}$
at which $\xi(t)$ is strictly $1$-\abbr{rsb}.
% such that  \eqref{eq:kRSB} holds with $k=1$.
Let
$\sfu := C(q_1)^{-1} (F'(\beta),\Es(q_1),\Rs(q_1),0)^{\top}$,   $\tau=\tau(q_1,r,\rho)$ of
\eqref{eq:tau} and $F(\beta ; q_1,\tau,\rho)$ the limiting free energy,  as in \eqref{eq:Fconvergence},  
 for the spherical spin model that corresponds to
$\xi_{q_1,\tau,\rho}(\cdot)$ of \eqref{def:xi-r-rho}.
Then,  for $f_{N,r}(\bs)$ of \eqref{eq:fFP} and any $|r| <1$,  $\beta',\delta>0$,
\[
	\lim_{N\to\infty}\E G_{N,\beta}\big( |f_{N,r}(\bs)- \Lambda(\beta,\beta',r)|> \delta\big)=0,
\]
for the limiting Franz-Parisi potential
	\begin{equation}\label{eq:Lambda}
	\begin{aligned}
	\Lambda(\beta,\beta',r) &:= \sup_{\rho\in J(q_1,r)} \{ \lambda(\beta,\beta',r,\rho) \} \,, \\
\lambda(\beta,\beta',r,\rho)&:=
\beta' \langle \sfv(q_1,r,\rho), \sfu  \rangle + F(\beta' ; q_1,\tau,\rho)
	+\frac12 \log\Big(\frac{1-\tau(q_1,r,\rho)}{1-r^2}\Big).
	\end{aligned}
	\end{equation}
\end{thm}
%Recall Remark \ref{rem:pure}, and note that here too the proof can be easily modified to handle
%the pure case (provided $\beta$ is large enough so that a pure states decomposition exists). 

\subsection{\label{FP-lt} Proof of Theorem \ref{thm:FP-lt}}
The following remark specifies our later use of Anderson's inequality.
\begin{rem}\label{rem:Anderson}
	Anderson's inequality \cite{Anderson}, states that for any symmetric $f:\R^d\to[0,\infty)$ of
	%, unimodal (i.e., having 
	convex super-level sets,  any symmetric,  convex $K\subset \R^d$,
     all $c\in[0,1]$ and $y\in\R^d$,
\[
\int _{K}f(x+cy)\,\mathrm {d} x\geq \int _{K}f(x+y)\,\mathrm {d} x \,.
\]
 As a consequence (stated as \cite[Corollary 3]{Anderson}),  
 if $X\sim\N(0,\Sigma_1)$, $Y\sim\N(0,\Sigma_2)$ with $\Sigma_1-\Sigma_2$ non-negative definite, then $\P(X\in K)\leq \P(Y\in K)$ for any symmetric, convex $K$. Suppose that $(W(t))_{t\in T}$ is an a.s.\ continuous, centered Gaussian process on some closed $T\subset \R^n$, $U$ is a Gaussian vector and $W(\cdot)$ and $U$ are jointly Gaussian. Then, in distribution, $W(\cdot) = \tilde W(\cdot) + \E[W(\cdot)\,|\,U]$ with both summands independent and $\tilde W(\cdot)$ a centered Gaussian process such that 
$
 \E  W(x)  W(y) - \E \tilde W(x) \tilde W(y)
 $ is non-negative definite.  Of course,  $\tilde W(\cdot)$ and $\E[W(\cdot)\,|\,U]$ are a.s.\ continuous. 
 Fixing a dense collection $\{x_i\} \subset T$ of distinct points, 
 $s \geq 0$ and finite $I$,  since the set
 $K= \{(s_i)_{i\in I}\in \R^I: \sup_{i\neq j} |s_i-s_j|/|x_i-x_j|\leq s \}$ is convex and symmetric, 
 we have that
	\begin{align*}
	\P\Big( \sup_{i\neq j\in I}\frac{|W(x_i)-W(x_j)|}{|x_i-x_j|}\leq s  \Big) &
	=\P\Big( \big(W(x_i)\big)_{i\in I}\in K \Big)\\
	&\leq \P\Big( \big(\tilde W(x_i)\big)_{i\in I}\in K \Big)=\P\Big( \sup_{i\neq j\in I}\frac{|\tilde W(x_i)-\tilde W(x_j)|}{|x_i-x_j|}\leq s  \Big).
	\end{align*}
	Taking $|I| \uparrow \infty$ 
	%with $(x_i)_{i\in I}$ dense in the limit, 
	we see that the Lipschitz constant of $\tilde W(\cdot)$ is stochastically dominated by that of $W(\cdot)$.
\end{rem}
Turning to prove Theorem \ref{thm:FP-lt},  by
Theorem \ref{thm:low temp 1},  we only need to show that 
	for any $\epsilon_{N}\to0$,
	% for arbitrary $\delta>0$,
		\begin{equation*}
		\varlimsup_{N\to\infty}\sup_{B(\bx_1,\epsilon_N)\times V(\epsilon_N)}\frac{1}{N}\log\P\left(\,
|f_{N,r}(\bs)- \Lambda(\beta,\beta',r)|> \delta \,\Big|\,\cpt(\bs,\bx_1,E, E_1, R_1)\,\right)<0.
	\end{equation*}
Further, $f_{N,r}(\bs)$ is a free energy of a model on $S(\bs,r)$ whose conditional variance is bounded by the unconditional variance $N\xi(1)$.   Hence by concentration (see \cite[Theorem 1.2]{PanchenkoBook}), it suffices to show that  for any $\epsilon_{N}\to0$,
\begin{equation}\label{eq:cond-exp-lim}
	\varlimsup_{N\to\infty}\sup_{B(\bx_1,\epsilon_N)\times V(\epsilon_N)} \left|\E\big[\,f_{N,r}(\bs)\,\big|\,\cpt(\bs,\bx_1,E, E_1, R_1)\,\big] - \Lambda(\beta,\beta',r)\right|= 0.
\end{equation}

Since the law of $H_N(\bs)$ is invariant under rotations of $\R^N$,  the conditional expectation in \eqref{eq:cond-exp-lim} depends on $\bs$ only through $\langle \bs,\bx_1\rangle/N$.  Hence,  it is enough to 
 restrict the supremum there to $\bs =\bs(q):= (\sqrt{Nq},\sqrt{N(1-q)},0,\ldots,0)$ 
for $q$ such that $|\langle \bs(q),\bx_1\rangle-\langle \bx_1,\bx_1\rangle|/N=|\sqrt{qq_1}-q_1|\leq \epsilon_{N}$. 
We thus set analogously to \eqref{eq:V}, 
\begin{equation*}\label{eq:Vq}
\widetilde{V}(\epsilon) :=
\left\{ (q,E,E_1,R_1):\,|q-q_1|,\,|E-F'(\beta)|,\,| E_1-\Es(q_1)|,\,|R_1-\Rs(q_1)|<\epsilon \right\},
\end{equation*}
and 
%the corresponding 
events $\cpt_1^q:=\cpt(\bs_1,O_q\bx_1,E, E_1, R_1)$,
for $\cpt(\cdot)$ of \eqref{eq:cpt} and the orthogonal matrix $O_q\in\R^{N\times N}$
that maps $\bs(q)$ to $\bs_1=\bs(q_1)$ 
and acts as the identity on $\sp\{\bs(q),\bs_1\}^{\perp}$.  Now,  by rotational invariance
\[
\E\big[\,f_{N,r}(\bs(q))\,\big|\,\cpt(\bs(q),\bx_1,E, E_1, R_1)\,\big] = 
\E\big[\,f_{N,r}(\bs_1)\,\big|\,   \cpt_1^q \,\big] \,,
\]
 and with the nominal event $\cpt_\star := \cpt(\bs_1,\bx_1,F'(\beta), \Es(q_1), \Rs(q_1))$,
we get \eqref{eq:cond-exp-lim} by showing that
%for any $\epsilon_{N} \to 0$
\begin{align}\label{eq:cond-exp-lim2}
	\varlimsup_{N\to\infty} &
	\sup_{\widetilde V(\epsilon_N)} \left|\E\big[\,f_{N,r}(\bs_1)\,\big|\,\cpt_1^q \,\big] - 
	\E\big[\,f_{N,r}(\bs_1)\,\big|\,\cpt_\star \,\big] \right| = 0,\\
\label{eq:cond-exp-lim3}
	\lim_{N\to\infty}  & \E\big[\,f_{N,r}(\bs_1)\,\big|\,\cpt_\star \,\big] = \Lambda(\beta,\beta',r)\,.
\end{align}

Turning to show \eqref{eq:cond-exp-lim2},  let
\begin{align*}
	C_N &:= \beta' \sup_{\widetilde V(\epsilon_N)}
	\sup_{\bs_1'\in S(\bs_1,r)}\frac1N\left|\E\big[H_N(\bs_1')\,\big|\,\cpt_1^q ] -\E\big[H_N(\bs_1')\,\big|\,\cpt_\star ]  \right|,\\
	D_N&:= \beta'^2 \sup_{\widetilde V(\epsilon_N)}
	\sup_{\bs_1',\bs_2'\in S(\bs_1,r)}\frac1N\left|\mbox{Cov}\big(H_N(\bs_1'),H_N(\bs_2')\,\big|\,\cpt_1^q\big) -\mbox{Cov}\big(H_N(\bs_1'),H_N(\bs_2')\,\big|\,\cpt_\star \big)  \right|
\end{align*}
%uniformly in $q,q'\in \{q:|q-q_1|\leq \epsilon_{N}\}$ and $(E,E_1,R_1), (E',E_1',R_1')\in V(\epsilon_{N})$, 
where
$\mbox{Cov}\big(\, \cdot ,\cdot\,\big|\,\cpt_1^{q} \big)$ denotes the conditional covariance. 
% and by abuse of notation $\cpt_1^{q'}$ is defined similarly to $\cpt_1^q$ with $q',E',E_1',R_1'$. 
Note that for all $N$ the space on which the free energy of $H_N(\bs')$ is computed is $S(\bs_1,r)$
and by Gaussian interpolation (see \cite[Section 1.2]{PanchenkoBook}),
\begin{align*}
	&\left|\E\big[f_{N,r}(\bs_1)\,\big|\,\cpt^q_1 ]-\E\big[f_{N,r}(\bs_1)\,\big|\,\cpt_\star ] \right|\leq C_N+D_N \,.
\end{align*}
%\footnote{
Indeed,  the term $C_N$ bounds 
the worst-case effect from the difference in the mean of the two conditional fields.  Thereafter,   let 
%the factor 
$\nu(\bs) := e^{\beta' \E[H_N(\bs)|\cpt_\star]}$ and consider
the interpolation $g_N^t=\sqrt{t} g_N^1+\sqrt{1-t} g_N^0$
for independent,  centered,  Gaussian fields $g_N^0$,  $g_N^1$ on $S(\bs_1,r)$,
setting  
$\varphi(t)=\frac1N\E\log\int_{S(\bs_1,r)}e^{g_N^t(\bs)}\nu(\bs) d\bs$.  After
Gaussian integration by parts,  one has
\[
\varphi'(t)= 
%\frac{1}{2N}\E \big\langle g^1_N(\bs)/\sqrt{t}-g^0_N(\bs)/\sqrt{1-t} \big\rangle=
	\frac{1}{2} \E
\langle C_1(\bs_1',\bs_1')-C_0(\bs_1',\bs_1')\rangle_t -\frac{1}{2} \E \langle C_1(\bs_1',\bs_2') - C_0(\bs_1',\bs_2') \rangle_t \,,	
\]
where  $\langle\cdot\rangle_t$ denotes the product Gibbs averaging of $\bs_1',\bs_2'$ under
Hamiltonian $g_N^t$ and $C_i(\bs_1',\bs_2'):=\frac1N\mbox{Cov} (g_N^i(\bs_1'),g_N^i(\bs_2'))$.  Hence, 
$|\varphi(1)-\varphi(0)| \le \sup_{\bs_1',\bs_2'} |C_1(\bs_1',\bs_2') - C_0(\bs_1',\bs_2')|$,  yielding in our case 
the term $D_N$ (as the worst-case over $\widetilde V(\epsilon_N)$).
%}
We further claim that  
\begin{align}\label{eq:CNDNlim}
	\varlimsup_{N\to\infty} C_N=\varlimsup_{N\to\infty} D_N=0 \,.
\end{align}
To see this, first note that setting $T:=\mbox{span}(X)$ for $X=\{\bs_1',\bs_2',\bs_1,O_q\bx_1\}$,  
$\{\partial_{\bx}H_N(O_q\bx_1)\}_{\bx\in T^{\perp}}$ and 
$\{H_N(\by)\}_{\by\in X}\cup\{\partial_{\bx}H_N(O_q\bx_1)\}_{\bx\in T}$ are independent of each other.  Hence, 
letting $Y=\{\by_1,\by_2,\by_3\}$  be unit vectors such that $\mbox{span}(Y)=\{\bz\in T:\,\bz\perp O_q\bx_1\}$, the conditional law of 
$H_N(\bs_1')$ and $H_N(\bs_2')$ given $\cpt_1^q$ is the same as their conditional law given that 
\begin{equation*}
	\Big(\frac{1}{N}H_{N}(\bs_1),\,\frac{1}{N}H_{N}(O_q \bx_1),\,\frac{\partial_{O_q\bx_1}H_{N}(O_q\bx_1)}{\|O_q\bx_1\|^2},\,(\partial_{\by_j} H_{N}(O_q\bx_1))_{j\leq3}\Big)=\big(E,E_1,R_1, {\bf 0}\big).
\end{equation*}
The conditional mean and covariance as in the definition of $C_N$ and $D_N$ above therefore depend continuously on $(E,E_1,R_1)$ and the values $\xi$, $\xi'$ and $\xi''$ at the normalized inner products
$\langle\bx,\by\rangle/N$ for $\bx,\by\in X\cup Y$, uniformly on $\widetilde{V}(\epsilon)$ for any small $\epsilon$
(in a manner which is independent of $N$;
	see also the more involved setting,  with $k \ge 1$,  in our proof of Lemma \ref{lem:Hhat}). From this, \eqref{eq:CNDNlim} easily follows.  

Moving next to \eqref{eq:cond-exp-lim3}, 
we fix $k\geq1$ and 
partition the interval $J$ of \eqref{eq:J} to $k$ equal sub-intervals $J_i$,  writing in view of \eqref{eq:unionrho}, 
\begin{align*}
	f_{N,r}(\bs_1)  = \frac1N\log \Big( \sum_{i=1}^{k} e^{N f_{N,r}(\bs_1;J_i)} \Big)
	\quad   \text{for} \quad f_{N,r}(\bs_1;J_i):= 
	\frac{1}{N} \log \Big( \int_{S(\bs_1,r;\bx_1,J_i)}\!\!\!\!\!\!\!\!\!\!\!\! e^{\beta' H_N(\bs')}d\bs' \Big).
	% \quad S(\bs_1,r;\bx_1,t_1,t_2):=\bigcup_{\rho\in [t_1,t_2]} S(\bs_1,r;\bx_1,\rho)\,,
\end{align*}
Each $f_{N,r}(\bs_1;J_i)$
%since each of terms above
is a free energy of a Hamiltonian whose conditional
variance given $\cpt_\star$ is bounded everywhere by the unconditional variance $N\xi(1)$.
Hence, from concentration (see \cite[Theorem 1.2]{PanchenkoBook}),
\[
\E\big[\,f_{N,r}(\bs_1)\,\big|\,\cpt_\star \,\big] = \max_{i=1}^k \E\Big[\,   f_{N,r}(\bs_1;J_i)
% \frac1N\log \Big(  \int_{S(\bs_1,r;\bx_1,J_i)}\!\!\!\!\!\!\!\!\!\!\!\!\!\!\!\!\!e^{\beta' H_N(\bs')}d\bs' \Big)
 \,\Big|\,\cpt_\star \,\Big] + o_N(1)\,.
\]
In addition,  there exist $\delta>0$ and closed interval $J' \subset J^o$ (the interior of $J$)
% Let $J'$ be a closed interval with $\min J'> \min J$ and $\max J'< \max J$
such that
% if $J_i \subset J\setminus J'$ then, 
for all large $N$,
\[
\max_{J_i \subset J \setminus J'} \E\Big[\,   f_{N,r}(\bs_1;J_i)
%\frac1N\log \Big(  \int_{S(\bs_1,r;\bx_1,J_i)}\!\!\!\!\!\!\!\!\!\!\!\!\!\!\!\!\!e^{\beta' H_N(\bs')}d\bs' \Big) 
\,\Big|\,\cpt_\star \,\Big] \le \E\Big[\,   f_{N,r}(\bs_1;J \setminus J')\,\Big|\,\cpt_\star \,\Big] 
\le \Lambda(\beta,\beta',r) - \delta
\]
% The existence of such $J'$ can be easily deduced 
(since decreasing $J \setminus J'$ makes
$\frac1N\log (\mbox{Vol}(S(\bs_1,r;\bx_1,J \setminus J'))$
% can be made
 as negative as we wish).
 % for $J_i \subset J\setminus J'$ and suitable $J'$).
To complete the proof of \eqref{eq:cond-exp-lim3},  it thus suffices to show that
for $\lambda(\beta,\beta',r,\rho)$ of  \eqref{eq:Lambda},  some $C<\infty$, 
all $k$ large and any $i$ such that $J_i \not\subset J\setminus J'$
\begin{equation}\label{eq:Eliran1}
\varlimsup_{N\to\infty} \sup_{\rho\in J_i} \Big| \E \Big[\,   f_{N,r}(\bs_1;J_i)
% \frac1N\log \Big(  \int_{S(\bs_1,r;\bx_1,J_i)}\!\!\!\!\!\!\!\!\!\!\!\!\!\!\!\!\!e^{\beta' H_N(\bs')}d\bs' \Big) 
\,\Big|\,\cpt_\star \,\Big] -  \lambda(\beta,\beta',r,\rho) \Big| \leq C/k \,.
\end{equation}
%\[
%\varlimsup_{N\to\infty} \bigg| \E\Big[\,   \frac1N\log \Big(  \int_{S(\bs_1,r;\bx_1,J_i)}\!\!\!\!\!\!\!\!\!\!\!\!\!\!\!\!\!e^{\beta' %H_N(\bs')}d\bs' \Big) \,\Big|\,\cpt_\star \,\Big] - \sup_{\rho\in J_i} \lambda(\beta,\beta',r,\rho) \bigg| \leq C/k\,,
%\]
To establish \eqref{eq:Eliran1},  note that by the co-area formula 
\[
f_{N,r}(\bs_1;J_i) =\frac{1}{N} \log \Big( \int_{J_i} \mbox{Jac}(s) 
\int_{S(\bs_1,r;\bx_1,s)}
\!\!\!\!\!\!\!\!\!\!\!\!\!\!\!\! e^{\beta' H_N(\bs')}d\bs' ds\Big) ,
\]
where the inner integration is \abbr{wrt} the uniform measure on $S(\bs_1,r;\bx_1,s)$ and for 
$\tau$ as in \eqref{eq:tau},
\[
\mbox{Jac}(s) \propto \Big(\frac{1-\tau(q_1,r,s)}{1-r^2}\Big)^{\frac{N-4}{2}}\qquad\mbox{and}\qquad\ \  \int_J \mbox{Jac}(s)ds=1\,.
\]
Moreover,
\[
\lim_{N \to \infty} \frac1N\log\mbox{Jac}(s) =   \lim_{N\to\infty}\frac{1}{N}\log\Big(
\frac{\mbox{Vol}(S(\bs_1,r;\bx_1,s))}{\mbox{Vol}(S(\bs_1,r))} \Big) =
\frac12 \log\bigg(\frac{1-\tau(q_1,r,s)}{1-r^2}\bigg),
\]
uniformly over $s \in J_i$ (since $J_i \not\subset J\setminus J'$ and $k$ is large,  so  
$J_i$ is bounded away from $J  \setminus J^o$).  Consequently,
\begin{align*}
\varlimsup_{N\to\infty} \sup_{\rho\in J_i} \Big|\,& \E\big[\,   f_{N,r}(\bs_1;J_i) \big|\,\cpt_\star \,\big] 
- \frac12 \log\Big(\frac{1-\tau(q_1,r,\rho)}{1-r^2}\Big) \\
& - 
\E\Big[\,   \frac1N\log \Big(  \int_{S(\bs_1,r;\bx_1,\rho)}\!\!\!\!\!\!\!\!\!\!\!\!\!\!\!\!\!e^{\beta' H_N(\bs')}d\bs' \Big) \,\Big|\,\cpt_\star \,\Big]
   \Big| \leq \frac{C}{k} \Big(1+\frac{\beta'}{\sqrt{N}}\E\Big[\,L_N \,\Big|\,\cpt_\star \,\Big]\Big),
\end{align*}
where $L_N$ denotes the Lipschitz constant of $H_N(\bs)$ over the ball of radius $\sqrt N$.
% (and the integration is \abbr{wrt} the uniform measure on $S(\bs_1,r;\bx_1,\rho)$). 

Recall that, conditionally on $\cpt_\star$,
we have that 
$H_N(\bs') = \tilde H_N(\bs') + \E[H_{N}(\bs')\,|\,\cpt_\star]$ for a centered Gaussian process 
$\tilde H_N(\bs')$. Also recall \cite[Lemma 6.1]{TAPChenPanchenkoSubag} that for large enough $C'$, $\E L_N \leq C'\sqrt N$.  By Anderson's inequality,  see Remark \ref{rem:Anderson}, 
$\E \tilde L_N \leq C'\sqrt N$ where $\tilde L_N$ is  the Lipschitz constant of the conditional Hamiltonian $\tilde H_N(\bs')$.  
Moreover,  the conditional expectation $\E[H_{N}(\bs')\,|\,\cpt_\star]$ that we computed 
in \eqref{eq:conditionalexpectation},  is also $C'\sqrt N$-Lipschitz in  $\bs'$
(if $C'$ is sufficiently large).

Finally, to prove \eqref{eq:cond-exp-lim3}, it remains to show that 
\begin{align}\label{eq:cond-exp-lim4}
	\varlimsup_{N\to\infty} \sup_{\rho\in J_i} \bigg|\,& \E\Big[\,   \frac1N\log \Big(  \int_{S(\bs_1,r;\bx_1,\rho)}\!\!\!\!\!\!\!\!\!\!\!\!\!\!\!\!\!e^{\beta' H_N(\bs')}d\bs' \Big) \,\Big|\,\cpt_\star \,\Big] - \Big(\beta' \langle \sfv(q_1,r,\rho), \sfu  \rangle + F(\beta' ; q_1,\tau,\rho)\Big)   \bigg| =0\, .
\end{align}
Indeed,  $\E [ H_N(\bs') \,\big| \, \cpt_\star ]=N \langle \sfv(q_1,r,\rho), \sfu  \rangle$
for any $\bs'\in S(\bs_1,r;\bx_1,\rho)$ (see \eqref{eq:conditionalexpectation}).
Moreover,  recall \eqref{eq:covcond},  that mapping the process $\tilde H_N(\bs')$ from $S(\bs_1,r;\bx_1,\rho)$ 
to the sphere $\mathbb S^{N-3}$ while preserving angles, yields 
a mixed model with mixture $\xi_{q_1,\tau,\rho}(t)+A$ for some constant $A\geq 0$ (see Remark \ref{rem:xiq1taurho}), up to scaling by a factor of $\sqrt{N/(N-2)}$.  As explained after \eqref{eq:xiqtilde},  
the expected free energy for the 
%Hamiltonian with 
mixture $\xi_{q_1,\tau,\rho}(t)+A$ is the same as that for the mixture $\xi_{q_1,\tau,\rho}(t)$.  
 Hence, \eqref{eq:cond-exp-lim4} follows.

\section{\label{sec:putestates}Pure states decomposition}

In \cite{TalagrandPstates} Talagrand proved an asymptotic pure states decomposition for sequences of random measures on the sphere $\SN$ as $N\to\infty$, assuming they satisfy the Ghirlanda-Guerra identities \cite{GhirlandaGuerra} and that the annealed overlap distribution converges to a law which charges the rightmost point in its support. In particular,  his result applies to generic spherical models,  if the rightmost point in the support of the Parisi measure is charged.  That is,  when $\nu_P(q_P)>0$ (which is part of our assumption of finite \abbr{rsb}).
% \eqref{eq:kRSB}).  
Using ultrametricity \cite{ultramet},  Jagannath \cite{JagannathApxUlt} proved that a pure states decomposition exists for generic models also when $\nu_P(q_P)=0$.  Below we state  these and some additional properties of 
% from \cite{FElandscape}} a straightforward corollary of
those decompositions 
under the strict $k$-\abbr{RSB} condition.
% of \eqref{eq:kRSB}.
To this end,  for any $m \le k$,  let
\begin{equation}\label{eq:AA}
	\AA_{m}:=\Big\{\vec\bx=(\bx_i)_{i=1}^m\in\R^{N\times m}:\,\frac1N\langle\bx_i,\bx_j\rangle=q_{i\wedge j},\,\forall i,j\leq m\Big\} \,.
\end{equation}
In particular,  if $\vec\bx \in\AA_{m}$,  then 
\begin{equation}
\label{eq:v+y}
\bx_m \in \Big\{\bx_{m-1}+\by:\,\by \in \sqrt{q_{m}-q_{m-1}} \SN,\,\langle\by,\bx_i\rangle=0,\forall i < m \Big\}\subseteq \sqrt{q_m} \SN \,.
\end{equation}
Recall that in \eqref{eq:cpt},  given $\vec\bx=(\bx_i)_{i=1}^k\in \AA_k$ we   set,  per $i \le k$, 
an arbitrary matrix $M_i=M_i(\vec\bx)\in\R^{N-i\times N}$
whose rows form an orthonormal basis of $\big(\sp\{\bx_1,\ldots,\bx_i\}\big)^{\perp}$ and $\gradt H_{N}(\bx_i):=M_{i}\nabla H_{N}(\bx_i)$.   We likewise define $M_i$ and $\gradt H_{N}(\bx_i)$ for 
any $1 \le i \le m \le k$ and $\vec\bx\in \AA_{m}$.  Then,  denoting by
$\bar V(\epsilon)$ the projection of $V(\epsilon)$ of \eqref{eq:V} to its last two components,  
 we set 
\begin{equation}
	\label{eq:xi_condition}
	\begin{aligned}
		\Cs_{m}(\epsilon):=\Big\{\vec \bx \in \AA_m:\,&\exists (\vec E, \vec R)\in \bar V(\epsilon)\mbox{ such that }\forall i\leq m\\
		&\Big(\frac{1}{N}H_{N}(\bx_i),\,\frac{\partial_{\bx_i-\bx_{i-1}}H_{N}(\bx_i)}{\|\bx_i-\bx_{i-1}\|^2},\,\gradt H_{N}(\bx_i)\Big)=\big(E_i,R_i,0\big) \Big\}\,.
	\end{aligned}
\end{equation} 
From the results of Section \ref{sec:KacRice},  the size of $\Cs_{m}(\epsilon)$ has finite expectation and is therefore a.s.\ finite.

For any vertices $\bu$ and $\bv$ of a rooted tree we denote by $|\bu|$ 
the depth of $\bu$ and by $\bu\wedge\bv$ the least common ancestor, 
writing $\bu\leq\bv$ if $\bu$ is an ancestor of $\bv$ or $\bu=\bv$.
Also, if $\T_N$ is a tree with vertex set $\V_N\subset \R^N$ rooted at the origin, for any vertex ${\bf 0}\neq
\bv\in\V_N$ we let $\vec\bv :=(\bv_1,\ldots,\bv_m=\bv)$  record the
	path from the root to $\bv$ (excluding the root).
% using the notation $\vec \bx_i^\star$ for a leaf $\bv=\bx_i^\star$.
\begin{prop}\label{prop:PS}
%\begin{cor}\!\!\emph{(}\cite[Corollary 14]{FElandscape}\emph{)}\label{cor:purestates}
Consider the spherical mixed $p$-spin model with a generic mixture $\xi(t)$,  
	and $\beta>\beta_{c}$ at which the model is strictly $k$-\abbr{rsb}  
	% such that \eqref{eq:kRSB} holds with
	 for some $k$ finite.
For any positive $c_N \to 0$,  
there exist positive $\delta_N,\epsilon_{N}\to0$ and integer $d_N\to\infty$,  such that the following holds. There exists a 
full regular tree $\T_N$ of degree $d_N$ and depth $k$ with
(random) vertex set $\V_N \subset \R^N$ 
and disjoint  (random) sets $B_i\subset \SN$  associated
with the leaves $\{ \bx^{\star}_i,  i \le d_N^k := (d_N)^k \}$ of $\T_N$,  so with probability
tending to $1$  as $N \to \infty$:
% and $\bu,\bv\in \V$:
	\begin{enumerate}
		\item\label{enum:PS1}
		 $\displaystyle\vphantom{\Big(} \sum_{i \le d_N^k} G_{N,\beta}\big(
		%\bigcup_{i \le d_N^k} 
		B_i\big)>1-\epsilon_{N}$\quad and\quad $\displaystyle \min_{i\leq d_N^k}G_{N,\beta}(B_i)>c_N$.
		\item\label{enum:PS2} $\displaystyle\vphantom{\Big(} 
		%\bv\leq \bx^\star_i\implies
		B_i \subset  B(\vec \bx^\star_i,\delta_N)$  of \eqref{eq:Band}, for all $i \le d_N^k$.
		\item\label{eq:Pt3} $\displaystyle \vphantom{\Big(} \sup_{\bu,\bv \in \V_N}
		\big|\frac1N \langle \bu,\bv\rangle -  q_{|\bu\wedge\bv|} \big| < \epsilon_{N}$.
		\item\label{enum:PS4} $\displaystyle 
		\min_{i,j \le d_N^k}
		\vphantom{\Big(} G_{N,\beta}^{\otimes2}\Big(
		\big|\frac1N \langle \bs_1,\bs_2\rangle - q_{| \bx^{\star}_i\wedge\bx^{\star}_j|}  \big|<\epsilon_{N}\,\Big|\,\bs_1\in B_i,\,\bs_2\in B_j
		\Big)>1-\epsilon_{N}$.
\item\label{enum:PS7} $\displaystyle\vphantom{\Big(} \{ \vec\bx_i^{\star}, i \le d_N^k\} \subset \Cs_{k}(\epsilon_N)$.
	\end{enumerate}	
In addition,  as $N \to \infty$, 
\begin{equation}\label{eq:energyontree}
 \sup_{\bv\in \V_N}\Big|\frac{1}{N}H_N(\bv)-\Es(q_{|\bv|})\Big| \longrightarrow 0 \,,
\quad\mbox{in probability.}  
 \end{equation}
Moreover,  upon ordering the leaves of ${\mathbb T}_N$ so that $i \mapsto G_{N,\beta}(B_i)$ 
be non-increasing (and padding it with zeros at all $i > d_N^k$),  these infinite sequences 
of weights,  converge weakly as $N \to \infty$ to a Poisson-Dirichlet distribution of 
parameter $1-\nu_P(q_P)$.
\end{prop}

\begin{proof}[Proof of Proposition \ref{prop:PS}, except Property \eqref{enum:PS7}]
Using the results of \cite{JagannathApxUlt}, 
even without the full power of \eqref{eq:kRSB},  the pure state decomposition of 
\cite[Corollary 14]{FElandscape} has properties \eqref{enum:PS1}--\eqref{enum:PS4}
apart from having $\|\cdot\|^2/N$ instead of $q_{|\cdot|}$ in both \eqref{eq:Pt3} and \eqref{enum:PS4}  and the convergence of the weights to a Poisson-Dirichlet distribution.\footnote{For this, in \cite{FElandscape} it is assumed that $|\mbox{supp}(\nu_P)|>1$.  This must hold in our setting,  since for $\beta>\beta_c$ we have from the Parisi formula that $\nu_P\neq \delta_0$,  and from \eqref{eq:kRSB} also that 
	$0\in \mbox{supp}(\nu_P)$.}
Going through its construction in the presence of \eqref{eq:kRSB} allows one to 
verify,  using  \cite[Point (4') in p.49]{FElandscape},  that property \eqref{enum:PS4} must also hold as stated here
(for some other non-random $\epsilon_N \to 0$).  Upon suitably modifying the sequences $\{\epsilon_N\}$,
this implies in turn that \eqref{eq:Pt3} must
hold as stated,  first in case $\bu=\bv$,  and hence by the triangle inequality,  also in its full generality.
In Proposition \ref{prop:PS} we may take $c_N \to 0$ 
slowly enough (which in turn implies it with arbitrary $c_N\to0$) for \cite[Corollary 15]{FElandscape} to hold as well.  Then, 
 as $\bv \mapsto H_N(\bv)$ is $O(\sqrt N)$-Lipschitz \abbr{whp},
combining the statement \cite[(1.48)]{FElandscape} of \cite[Corollary 15]{FElandscape} 
with Property \eqref{eq:Pt3} at $\bu=\bv$ and the Borell-TIS inequality yields  \eqref{eq:energyontree}.
% (recall our definition \eqref{eq:GS} of $E_\star(\cdot)$).  
\end{proof}

To summarize,  
% from \cite{FElandscape} 
we saw that the points in our embedding of the
% tree of 
pure states decomposition,  have energy which is (approximately) maximal over the relevant spheres.
% of radii near $\sqrt{q_m N}$.  
Property \eqref{enum:PS7} of Proposition \ref{prop:PS},  proved
in Section \ref{sec-pf-prop:PS},  goes further to show that such embedding 
$\V_N$ exists where also
 \begin{equation*}\label{eq:Am}
 \bu \le \bv \quad \Longrightarrow \quad \frac1N \langle \bu,\bv \rangle = q_{|\bu|} 
\end{equation*}
(or equivalently,  with $\vec \bv \in \AA_{|\bv|}$ for any $\bv \in \V_N$),  and 
where each $\bx_m \in \V_N$ of depth $m$ is \emph{also a critical point of the energy} 
on the subset \eqref{eq:v+y} of the relevant sphere (as $\gradt H_{N}(\bx_{m})=0$).

\section{\label{sec:HighTempPf}High temperature: proof of Theorem \ref{thm:high-temp}}

First note that since we work with $\beta\leq\beta_{c}$, $\E F_{N,\beta}\to F(\beta)=\frac{1}{2}\beta^{2}$
and the derivative exists and is equal to $F'(\beta)=\beta$. For
$\beta<\beta_{c}$ this follows by definition. For $\beta=\beta_{c}$
it follows since the derivative from the
left is $\beta_{c}$, by convexity the derivative of $F(\beta)$
from the right exists and is at least $\beta_{c}$, and for any $\beta$
(and in particular on some right neighborhood of $\beta_{c}$) by
Jensen's inequality $\E F_{N,\beta}\leq\frac{1}{2}\beta^{2}$.

The argument from this point up to \eqref{eq:Lpm_ratio}
is standard, and is taken almost with no change from \cite{AuffingerChenConcentration}. Let $\epsilon>0$ and define
the subsets
\begin{align*}
	L_{N}^{+}(\epsilon) & :=\left\{ \bs\in\Sigma_{N}:\,H_{N}(\bs)/N-F'(\beta)>\epsilon\right\} ,\\
	L_{N}^{-}(\epsilon) & :=\left\{ \bs\in\Sigma_{N}:\,H_{N}(\bs)/N-F'(\beta)<-\epsilon\right\} .
\end{align*}
For any $\lambda>0$,
\begin{align*}
	\frac{1}{N}\log\int_{L_{N}^{\pm}(\epsilon)}e^{\beta H_{N}(\bs)}d\mu_{N} & \leq\frac{1}{N}\log\int_{\Sigma_{N}}e^{(\beta\pm\lambda)H_{N}(\bs)\mp\lambda N(F'(\beta)\pm\epsilon)}d\mu_{N}\\
	& =F_{N,\beta\pm\lambda}\mp\lambda(F'(\beta)\pm\epsilon).
\end{align*}

Defining
\[
\delta_{N}^{\pm}(\lambda):=\frac{F_{N,\beta\pm\lambda}-F_{N,\beta}}{\lambda}\mp F'(\beta),
\]
we have that
\[
F_{N,\beta\pm\lambda}\mp\lambda(F'(\beta)\pm\epsilon)=F_{N,\beta}-\lambda(\epsilon-\delta_{N}^{\pm}(\lambda)).
\]

Also define
\[
\delta_{\infty}^{\pm}(\lambda):=\frac{F(\beta\pm\lambda)-F(\beta)}{\lambda}\mp F'(\beta)\overset{\lambda\to0}{\longrightarrow}0.
\]

Choose some $\lambda>0$ sufficiently small so that $a(\epsilon,\lambda):=\epsilon-\delta_{\infty}^{+}(\lambda)\vee\delta_{\infty}^{-}(\lambda)>0$.
Define the event 
\begin{equation*}
	\mathcal{E}_{N}(\tau):=
	\big\{  |F_{N,\beta\pm\lambda}-F(\beta\pm\lambda)|\vee |F_{N,\beta}-F(\beta)|<\tau \big\}\,.
\end{equation*}  
 Since we assume that 
$\sup_{N,\bs} \frac{1}{N}\E\big[H_{N}(\bs)^{2}\big]$ is finite,  
from the well-known concentration of the free energy \cite[Theorem 1.2]{PanchenkoBook},  for any $\tau>0$, 
\begin{equation}\label{eq:Etaucomp}
\varlimsup_{N\to\infty}\frac{1}{N}\log\P(\mathcal{E}_{N}(\tau)^{c})<0 \,.
\end{equation}
Taking hereafter $2\tau<\lambda a(\epsilon,\lambda)$,  on $\mathcal{E}_{N}(\tau)$
\[
\frac{1}{N}\log\int_{L_{N}^{\pm}(\epsilon)}e^{\beta H_{N}(\bs)}d\mu_{N}\leq F_{N,\beta}-\lambda a(\epsilon,\lambda)+ 2 \tau
\]
and consequently also
\begin{equation}
	\varlimsup_{N\to\infty}\frac{1}{N}\log\E\bigg(\indic_{\mathcal{E}_{N}(\tau)}\cdot\frac{\int_{L_{N}^{+}(\epsilon)\cup L_{N}^{-}(\epsilon)}e^{\beta H_{N}(\bs)}d\mu_{N}(\bs)}{\int_{\Sigma_{N}}e^{\beta H_{N}(\bs)}d\mu_{N}(\bs)}\bigg)\leq2\tau-\lambda a(\epsilon,\lambda)<0.\label{eq:Lpm_ratio}
\end{equation}
In view of \eqref{eq:Etaucomp},  \abbr{wlog} we can multiply the $[0,1]$-valued 
$G_{N,\beta}\left(f_{N}(\bs)\in A\right)$ in \eqref{eq:ht-result} by $\indic_{\mathcal{E}_{N}(\tau)}$.  Then, 
by
\eqref{eq:G(finA)} and \eqref{eq:Lpm_ratio},  we complete the proof upon showing that for
$L_{N}(\epsilon):=\Sigma_{N}\setminus(L_{N}^{+}(\epsilon)\cup L_{N}^{-}(\epsilon))$,
\begin{equation}
	\varlimsup_{N\to\infty}\frac{1}{N}\log\E\bigg(\indic_{\mathcal{E}_{N}(\tau)}\cdot\frac{\int_{L_{N}(\epsilon)}e^{\beta H_{N}(\bs)}\indic\{f_{N}(\bs)\in A\}d\mu_{N}(\bs)}{\int_{\Sigma_{N}}e^{\beta H_{N}(\bs)}d\mu_{N}(\bs)}\bigg)<0.\label{eq:ratio_required_bd}
\end{equation}
Now, 
% using that $F'(\beta)=\beta$ and 
 applying \eqref{eq:variance-bd},
%for  $v_N(\bs):=N^{-1} \E[H_N(\bs)^2]$
\begin{equation*}
	\label{eq:ELN}
	\begin{aligned}
		\E\int_{L_{N}(\epsilon)}e^{\beta H_{N}(\bs)} & \cdot\indic\big\{ f_{N}(\bs)\in A\big\} d\mu_{N}(\bs)
		%=\int_{\Sigma_N}  \int_{F'(\beta)-\epsilon}^{F'(\beta)+\epsilon} e^{N(\beta E-\frac{E^2}{2v_N(\bs)})} 
		%\sqrt{\frac{N}{2\pi v_N(\bs)}}
		%\P\left(f_{N}(\bs)\in A\,\Big|\,\frac{1}{N}H_{N}(\bs)=E \right) dE d\mu_{N}(\bs) 
		\\
		& =\int_{\Sigma_{N}}\int_{F'(\beta)-\epsilon}^{F'(\beta)+\epsilon}e^{N(\beta E-\frac{1}{2}E^{2})+o(N)}\P\Big(f_{N}(\bs)\in A\,\Big|\,\frac{1}{N}H_{N}(\bs)=E\Big)dE\,d\mu_{N}(\bs),
	\end{aligned}
\end{equation*}
where the $o(N)$ term is uniform in $\bs$ and $E$.  Since on $\mathcal{E}_{N}(\tau)$
the denominator in (\ref{eq:ratio_required_bd}) is at least $e^{(F(\beta)
-\tau)N}$ and 
\[
\beta F'(\beta) - \frac{1}{2} F'(\beta)^2 - F(\beta) = 0  \,,
\]
if we take small enough $\epsilon$ and $\tau$ then (\ref{eq:ratio_required_bd})
follows from (\ref{eq:bd_Gaussian_average}).
\qed

\section{\label{sec:PfSketch}Proof sketch for Theorem \ref{thm:low temp 1} on the low temperature phase}

In this section we outline the main steps in the proof of Theorem \ref{thm:low temp 1}.

\subsection{\label{subsec:PfSketch_purestates}Using the pure state decomposition of Proposition \ref{prop:PS}}  
 First, 
define the measures (which may have mass larger than $1$), 
\begin{equation}\label{eq:barPi}
\bar\Pi^\epsilon_{N,\beta}(\,\cdot\,):= 
\sum_{\vec\bx\in \mathscr{C}_{k}(\epsilon)} G_{N,\beta}\big(\,\cdot\,\cap\, B(\vec\bx,\epsilon)\big) \le 
\sum_{\vec\bx\in \mathscr{C}_{k}(\epsilon)} G_{N,\beta}\big(\,\cdot\,\big| \, B(\vec\bx,\epsilon)\big) =:\Pi^\epsilon_{N,\beta}(\,\cdot\,)\,.
\end{equation}
Then,  denoting by $\mathcal{B}(\SN)$ the set of Borel measurable subsets of $\SN$, 
it follows from properties \eqref{enum:PS1}, \eqref{enum:PS2} and \eqref{enum:PS7} 
of Proposition \ref{prop:PS},  that
for any fixed $\epsilon>0$, 
\begin{equation}\label{eq:SG}
	\lim_{N\to\infty} \P \Big(\forall M\in \mathcal{B}(\SN):\, \bar\Pi_{N,\beta}^\epsilon(M)\geq G_{N,\beta}(M)-\epsilon_{N}\Big)=1.
\end{equation}
%For technical reasons, it will be more convenient to work with the function
%\[
%\bar S^\epsilon_{N,\beta}(M):=  S^\epsilon_{N,\beta}(M)\wedge 1,
%\]
%where $a\wedge b:=\min\{a,b\}$. Since $G_{n,\beta}$ is a probability measure, $S^\epsilon_{N,\beta}$ satisfies \eqref{eq:SG} if and only if $\bar S^\epsilon_{N,\beta}$ does.
Hence,  to prove \eqref{eq:GAverageLowT}, it suffices to show  that
\begin{equation}
	\lim_{\epsilon\to0}\varlimsup_{N\to\infty}\frac{1}{N}\log \E \Big( \Pi^{\epsilon}_{N,\beta}\big(f_{N}(\bs)\in A \big)\Big)<0. \label{eq:PiAverage}
\end{equation}
More generally, for (sufficiently regular) $\alpha_k:\AA_k\to[0,1]$ such that the law of $\alpha_k(\vec\bx)$ is independent of $\vec\bx$, consider the limit
\begin{equation}\label{eq:genE}
	\lim_{\epsilon\to0}\varlimsup_{N\to\infty}\frac{1}{N}\log  \E  \Big( \sum_{\vec\bx\in \mathscr{C}_{k}(\epsilon)} \alpha_k(\vec{\bx}) \Big)\,.
\end{equation}
Indeed, by \eqref{eq:barPi}, the expectation of \eqref{eq:PiAverage} is the particular case of  
\begin{equation}\label{eq:specific_alpha}
	\alpha_k(\vec{\bx}) = G_{N,\beta}\big(\{\bs:\,f_N(\bx)\in A \}\big| \, B(\vec\bx,\epsilon)\big)\,.
\end{equation}

\subsection{Applying the Kac-Rice formula}

The advantage of dealing with \eqref{eq:genE} rather than \eqref{eq:GAverageLowT} directly, is that to analyze such sums we may apply 
the Kac-Rice formula. Indeed, by considering the sequence of gradients
\begin{equation}\label{eq:k_gradt}
	\vec\bx \mapsto \big(\gradt H_N(\bx_1),\ldots,\gradt H_N(\bx_k)\big)\,,
\end{equation}
we obtain a mapping from the $\Dp$-dimensional manifold $\AA_k$ to $\R^{\Dp}$ as required for the application of the Kac-Rice formula,  with $\Dp:=N-1+\cdots+N-k=k(N-\frac{k+1}{2})$.
Further, consider two vector-valued random fields on $\AA_k$ induced by the real-valued random field $H_N$,
\begin{align}
\vec\bx &\mapsto \big(H_N(\bx_1),\ldots,H_N(\bx_k)\big)\,, \label{eq:k_HN}\\
\vec\bx &\mapsto \bigg(\frac{\partial_{\bx_1}H_{N}(\bx_1)}{\|\bx_1\|^2},\ldots,\frac{\partial_{\bx_k-\bx_{k-1}}H_{N}(\bx_k)}{\|\bx_k-\bx_{k-1}\|^2}\bigg)\,. \label{eq:k_RN}
\end{align}
Then,  by definition, $\Cs_{m}(\epsilon)$ of \eqref{eq:xi_condition} is the set of points in $\AA_k$ such that \eqref{eq:k_gradt} is equal to $0\in\R^{\Dp}$ and \eqref{eq:k_HN} normalized by $N$ and \eqref{eq:k_RN} are within distance $\epsilon$ from $(\Es(q_i))_{i\le k}$ and $(\Rs(q_i))_{i\le k}$ of \eqref{eq:GS} in sup-norm.

\subsection{Warm-up: the $1$-RSB case} Before tackling the general case we explain what to do when
$k=1$.  Here, we write $\vec\bx=\bx\in\AA_1$. If $\alpha_1(\bx)\equiv1$, then \eqref{eq:genE} is  the `usual' annealed complexity  of critical points at radius $q_1$, with energy and radial derivative in a given interval \cite{ABA2,A-BA-C,geometryMixed}. By \eqref{eq:complexity} below, it is equal in this case to 
\begin{equation}\label{eq:complexity_k1}
	\lim_{\epsilon\to0}\varlimsup_{N\to\infty}\frac{1}{N}\log  \Big( \E\{|\mathscr{C}_{1}(\epsilon)|\} \Big) = \Theta_{\xi(q_1\cdot)}(\Es(q_1),q_1\Rs(q_1))\,,
\end{equation}
for the explicit function  $\Theta_\xi(E,R)$ of \eqref{eq:Theta}.

To obtain an upper bound for \eqref{eq:genE} with general $\alpha_1$ using the Kac-Rice formula, one essentially needs to augment the complexity calculation by an upper bound for the expectation of $ \alpha_{1}(\bx)$, under the same conditioning that appears in the Kac-Rice formula.\footnote{The expectation in the Kac-Rice formula also includes the determinant of the Hessian.  However,  this determinant concentrates so well that at the exponential scale in which we work here,  it is not difficult to deal with.} Precisely, denoting by $\Delta(\alpha_1)$ the limit
\begin{equation}\label{eq:CE1}
	\lim_{\epsilon\to0}\varlimsup_{N\to\infty}\frac{1}{N}\log \Big(  \sup_{\substack{| E_1-\Es(q_1)|<\epsilon\\ |R_1-\Rs(q_1)|<\epsilon}}\E \Big\{ \alpha_1(\bx) \,\Big|\,
	\Big(\frac{1}{N}H_{N}(\bx),\,\frac{\partial_{\bx}H_{N}(\bx)}{\|\bx\|^2},\,\gradt H_{N}(\bx)\Big)=\big(E_1,R_1,0\big) \Big\}\Big) \,,
\end{equation}
one can show that
\begin{equation}\label{eq:k1_basicbound}
	\eqref{eq:genE}\le \Theta_{\xi(q_1\cdot)}(\Es(q_1),q_1\Rs(q_1))+\Delta(\alpha_1)\,.
\end{equation}

For the case of interest to us \eqref{eq:specific_alpha}, we shall see that under the  conditioning as in \eqref{eq:CE1}, the Gibbs measure $G_{N,\beta}(\,\cdot\,\big| \, B(\bx,\epsilon))$ on $B(\bx,\epsilon)$ is effectively  in the high-temperature (\abbr{RS}) phase. Thus, by a similar argument to that used for Theorem \ref{thm:high-temp} (which needs to be slightly strengthened to account for supremum), for \eqref{eq:specific_alpha} we show in Section \ref{sec:LowTempPf} that if \eqref{eq:bd_Gaussian_average-11} holds then $\Delta(\alpha_1)<0$.

\subsection{The case of general $k\ge1$}

To bound \eqref{eq:genE} one may in principle apply the Kac-Rice formula directly on the space $\AA_k$,  
but it is more natural to instead
use an inductive argument on $k$.  Specifically, 
for any $\vec\by=(\by_1,\ldots,\by_m)\in \AA_m$ and $m\leq n\leq k$,  we define,   similarly to \eqref{eq:AA} and  \eqref{eq:xi_condition},
\begin{equation}
	\label{eq:AmnCmny}
\begin{aligned}
	\AA_{m,n}(\vec\by)&:=\Big\{\vec\bx=(\bx_i)_{i=1}^{n}\in\AA_{n}:\,(\bx_1,\ldots,\bx_m)=(\by_1,\ldots,\by_m)\Big\}\\
		\Cs_{m,n}(\epsilon,\vec\by)&:=\Big\{\vec \bx \in \AA_{m,n}(\vec\by):\,\exists (\vec E, \vec R)\in \bar V(\epsilon)\mbox{ such that }\forall m+1\leq i\leq n,\\
		&\hphantom{:=\Big\{\vec \bx \in \AA_{m,n}(\vec\by):\,}\Big(\frac{1}{N}H_{N}(\bx_i),\,\frac{\partial_{\bx_i-\bx_{i-1}}H_{N}(\bx_i)}{\|\bx_i-\bx_{i-1}\|^2},\,\gradt H_{N}(\bx_i)\Big)=\big(E_i,R_i,0\big) \Big\}\,.
\end{aligned}
\end{equation}
We may then write
\begin{equation}\label{eq:KRrep}
	\begin{aligned}
		\sum_{\vec\bx\in \mathscr{C}_{k}(\epsilon)} \alpha_k(\vec{\bx}) = \sum_{\by\in \mathscr{C}_{1}(\epsilon)} \alpha_{1,k}^{\epsilon}(\by)
		= \sum_{\by\in \mathscr{C}_{1}(\infty)} \indic\{\by\in \mathscr{C}_{1}(\epsilon)\} \cdot \alpha_{1,k}^{\epsilon}(\by)\,,
	\end{aligned}
\end{equation}
where $\mathscr{C}_{1}(\infty)$ is the set of critical points of $H_N$ over $\AA_1=\sqrt{q_1}\SN$ and
\begin{equation}\label{eq:alpha1k}
		\alpha_{1,k}^{\epsilon}(\by) :=	\sum_{\vec\bx\in \mathscr{C}_{1,k}(\epsilon,\by)} \alpha_k(\vec\bx)\,.
\end{equation}
While \eqref{eq:KRrep} holds with $\alpha_{1,k}^{\epsilon}(\by)$ defined
by summing over $\{\vec\bx\in \mathscr{C}_{k}(\epsilon):\,\bx_1=\by\}$ in \eqref{eq:alpha1k},  
such $\alpha_{1,k}^{\epsilon}$ being zero outside $\mathscr{C}_{1}(\infty)$, hence
lacks the continuity over $\AA_1$ needed for applying the Kac-Rice formula.

%To obtain an upper bound for \eqref{eq:genE} using the Kac-Rice formula and the representation \eqref{eq:KRrep}, one essentially needs to augment the well-known computation of $\E\{|\mathscr{C}_{1}(\epsilon)|\}$ (see \cite{A-BA-C,ABA2,geometryMixed}) by an upper bound for the expectation of $ \alpha_{1,k}^{\epsilon}(\by)$, conditional on an appropriate event.

Fix $\by\in \sqrt{q_1}\SN$.  Define
\[
\varphi(\bz) = \by+\sqrt{\frac{N(1-q_1)}{N-1}}\theta(\bz)
\]
for an arbitrary linear isometry $\theta:\R^{N-1}\to\by^{\perp}=\{\bx:\bx\perp\by\}$, and note that for this choice of the scale factor $\varphi(\SNt)=(\by+\by^{\perp})\cap\SN$. 
By abuse of notation, for vectors $\vec\bz = (\bz_1,\ldots,\bz_{k-1})$ define $\varphi(\vec\bz)=(\varphi(\bz_i))$ by element-wise operation.
To make the dependence in $H_N$ and $q=(q_i)_{i=1}^k$ explicit, we use the notation $\mathscr{C}_{k}(\infty;H_N,q)$ and $\mathscr{C}_{1,k}(\infty,\by;H_N,q)$, where here taking $\epsilon=\infty$ the values of $\vec E$ and $\vec R$ are irrelevant. 
%Let $\mathscr{C}^{(1)}$ denote the set of points $\mathscr{C}_{1,k}(\infty,\by)$ for the original Hamiltonian $H_N$ and parameters $(q_i)_{i\le k}$ and let $\mathscr{C}^{(2)}$ denote
%$\mathscr{C}_{k-1}(\infty)$ for the Hamiltonian  $H_N^\by$ and modified $(\hat q_{i+1})_{i\le k-1}$ with $\hat q_i= \frac{q_{i}-q_1}{1-q_1}$. 
% There's a note in the dropbox folder that verifies the following:
Define $\hat q=(\hat q_i)_{i=1}^{k-1}$, where $\hat q_i= \frac{q_{i}-q_1}{1-q_1}$.
One can verify that $\vec{\bx}\in\mathscr{C}_{1,k}(\infty,\by;H_N,q)$ with 
\[
\Big(\frac1NH_N(\by),\,\frac{1}{N}H_{N}(\bx_i),\,\frac{\partial_{\bx_i-\bx_{i-1}}H_{N}(\bx_i)}{\|\bx_i-\bx_{i-1}\|^2}\Big) = (\bar E,\,\bar E_i,\, \bar R_i)
\]
if and only if $\vec\bx=\varphi(\vec\bz)$ for  $\vec\bz\in\mathscr{C}_{k-1}(\infty;H_N^\by,\hat q)$  such that
\[
\Big(\frac{1}{N-1}H^{\by}_{N}(\bz_i),\,\frac{\partial_{\bz_i-\bz_{i-1}}H^{\by}_{N}(\bz_i)}{\|\bz_i-\bz_{i-1}\|^2}\Big) = \sqrt{\frac{N}{N-1}}\big(\bar E_{i+1}-\bar E,\, (1-q_1)\bar R_{i+1}\big)\,.
\]
Hence, setting $\hat E=(E_{i+1}-E_1)_{i=1}^{k-1}$ and $\hat R=((1-q_1)R_{i+1})_{i=1}^{k-1}$, we have that
\[
\alpha_{1,k}^{\epsilon}(\by) =	\sum_{\vec\bx\in \mathscr{C}_{1,k}(\epsilon,\by;H_N,q,\vec E,\vec R)} \alpha_{k}(\vec\bx)  
\le
\sum_{\vec\bz\in \mathscr{C}_{k-1}(4\epsilon,\by;H^\by_N,\hat q,\hat E,\hat R)} \hat\alpha_{k-1}(\vec\bz)\,,
\]
where $\hat \alpha_{k-1}(\vec\bz):=\alpha_{k}(\varphi(\vec\bz))$ and where we included $\vec E$ and $\vec R$ in the notation to make the dependence on them explicit. The sum on \abbr{rhs} is of the same form as in \eqref{eq:genE},  with $k$ reduced by $1$ and modified Hamiltonian and parameters.

Proposition \ref{prop:multilvlKR} states an upper bound on \eqref{eq:genE} 
which is proved by an inductive argument using the above idea. Specifically,  it yields a generalization of \eqref{eq:k1_basicbound} for $k\ge1$:
\begin{equation}\label{eq:kge1_basicbound}
	\eqref{eq:genE}\le \sum_{i=0}^{k-1} \Theta_{\xim}\Big(\Es(q_{i+1})-\Es(q_{i}),(q_{i+1}-q_i)\Rs(q_{i+1})\Big) +\Delta(\alpha_k)\,,
\end{equation}
where $\xim$ are explicit mixtures defined by \eqref{eq:xiq} and \eqref{eq:xim} and, generalizing \eqref{eq:CE1}, $\Delta(\alpha_k)$ is defined as the limit
\begin{equation*}
	\lim_{\epsilon\to0}\varlimsup_{N\to\infty}\frac{1}{N}\log \bigg( \sup_{\bar V(\epsilon)}\E \Big\{ \alpha_k(\vec\bx) \,\Big|\,
	 \cpt(\vec \bx,\vec E,\vec R) \Big\}\bigg)\,,
\end{equation*}
for the event
\begin{equation*}
	\cpt(\vec{\bx},\vec E,\vec R) := \bigcap_{1\leq i \le k} \Big\{
	\Big(\frac{1}{N}H_{N}(\bx_i),\,\frac{\partial_{\bx_i-\bx_{i-1}}H_{N}(\bx_i)}{\|\bx_i-\bx_{i-1}\|^2},\,\gradt H_{N}(\bx_i)\Big)=\big(E_i,R_i,0\big) \Big\}\,.
\end{equation*}

By the same argument mentioned above for the case $k=1$, also for $k\ge1$ we show in Section \ref{sec:LowTempPf} that if \eqref{eq:bd_Gaussian_average-11} holds then $\Delta(\alpha_k)<0$.

\subsection{Vanishing of the annealed complexities}

Recall the zero-temperature Parisi formula \eqref{eq:ZTParisi}. Analogously to the set $\mathcal{S}_{P}=\mathcal{S}_{P,\xi,\beta}$ of \eqref{eq:Parisi_characterization_finite_beta} which was defined for finite $\beta>\infty$, a zero-temperature analogue $\mathcal{S}_{P,\xi,\infty}\subset [0,1]$ can be defined from the optimality condition for the Parisi measure, see \eqref{eq:Parisimeasure_ZT}.

Using the characterization \eqref{eq:Parisimeasure_ZT} of the Parisi measure, it was proved in \cite{FElandscape} that the mixtures $\xim$ as in \eqref{eq:kge1_basicbound} are strictly 1-\abbr{rsb} in the sense that $\mathcal{S}_{P,\xim,\infty}=\{0,1\}$, see Corollary \ref{cor:Parisidistinfty} and \eqref{eq:xim_strict1RSB}. 

Making the dependence on the mixture explicit by writing $\Es(q)=\Es(\xi,q)$ and $\Rs(q)=\Rs(\xi,q)$, in Lemma \ref{lem:EsRs} we show that
	\begin{align*}
		\Es(\xim,1)&=\Es(\xi,q_{m+1})-\Es(\xi,q_{m}),\\
		\Rs(\xim,1)&=(q_{m+1}-q_{m})\Rs(\xi,q_{m+1}).
	\end{align*}
Finally, by a result of Huang and Sellke \cite{HuangSellke2023},  which we restate in Lemma \ref{lem:strict1RSB},  
if  $\mathcal{S}_{P,\xim,\infty}=\{0,1\}$, then $\Theta_{\xi}(\Es(1),\hRs(1))=0$. 

By combining these facts, we conclude in Corollary \ref{cor:Thetastar} that each of the complexities in the sum of \eqref{eq:kge1_basicbound} is equal to zero. Since, as stated above, if \eqref{eq:bd_Gaussian_average-11} holds then $\Delta(\alpha_k)<0$,  in Section \ref{sec:LowTempPf} we conclude from this that 
\eqref{eq:genE} is strictly negative in the setting of Theorem \ref{thm:low temp 1},  thereby completing its proof.

\section{\label{sec:complexity}Complexity of critical points and related results}

In this section we collect several results from \cite{geometryMixed,HuangSellke2023,FElandscape} about the complexity of critical points and prove some consequences of theirs that will be used later. 
For any mixture $\xi(t)$ and intervals $I,I'\subset \R$ denote the number of critical points with normalized energy in $I$ and normal derivative in $I'$ by
\begin{equation*}
	\Crt_{N,\xi}(I,I'):=\#\Big\{
	\bs\in\SN:\,\forall \bu\perp\bs,\,\langle\nabla H_N(\bs),\bu\rangle=0, \, \frac1N H_N(\bs)\in I, \frac{1}{N}\langle\nabla H_N(\bs),\bs\rangle\in I'
	\Big\}.
\end{equation*}
For a mixture $\xi(t)$  which is not pure,  define the function
\begin{equation}\label{eq:Theta}
	\Theta_\xi(E,R)=\frac12+\frac12\log\Big(\frac{\xi''(1)}{\xi'(1)}\Big)-\frac12(E,R)\Sigma_{\xi}^{-1}(E,R)^{\top}+\Omega(R/\sqrt{\xi''(1)}),
\end{equation}
where  (see, e.g., \cite[Prop. II.1.2]{logpotential})
\begin{equation}\label{eq:Omega}
	\begin{aligned}
		\Omega(t) & =\frac{1}{2\pi}  \int_{-2}^{2} \log|\lambda-t| \sqrt{4-\lambda^2}d\lambda\\
		& = \begin{cases}
			\frac{t^2}{4}-\frac{1}{2} &|t|\leq 2\\
			\frac{t^2}{4}-\frac{1}{2} - \Big[\frac{|t|}{4}\sqrt{t^2-4}-\log\Big(\sqrt{\frac{t^2}{4}-1}+\frac{|t|}{2}\Big)\Big] &|t|>2,
		\end{cases}
	\end{aligned}
\end{equation}
and
\begin{equation}\label{eq:Sigmaxi}
	\Sigma_{\xi}=\left(
	\begin{matrix}
		\xi(1) & \xi'(1) \\
		\xi'(1) & \xi''(1) +\xi'(1)
	\end{matrix}	
	\right).
\end{equation}
In the pure case $\xi(t)=t^p$,  define
\begin{equation*}\label{eq:Thetapure}
\Theta_\xi(E,R)=\begin{cases}
	\frac12+\frac12\log\Big(\frac{\xi''(1)}{\xi'(1)}\Big)-\frac12 E^2/\xi(1)+\Omega(R/\sqrt{\xi''(1)}) &R=pE\\
	-\infty &\mbox{otherwise.}
\end{cases}
\end{equation*}
Theorem 2.8 of \cite{A-BA-C}\footnote{Note that in the pure case, $\langle\nabla H_N(\bs),\bs\rangle = NpH_N(\bs)$.} and Theorem 3.1 of \cite{geometryMixed} state that for any intervals $I,I'\subset \R$,
\begin{equation}\label{eq:complexity}
	\lim_{N\to\infty}\frac1N\log \E \Crt_{N,\xi}(I,I') = \sup_{(E,R) \in I\times I'}\Theta_\xi(E,R).
\end{equation}

For any $q\in[0,1)$, define the mixtures
\begin{equation}
	\label{eq:xiq}
	\xi_q(t)=\xi(t+q)-\xi(q)-\xi'(q)t,  \qquad \bar \xi_q(t) :=\xi_q((1-q)t) \,.
\end{equation}
Assuming the strict $k$-\abbr{rsb} condition \eqref{eq:kRSB} and setting $q_{k+1}=1$,  we define for any 
$1 \le m \leq k$ the mixtures
\begin{equation}\label{eq:xim}
	\begin{aligned}
		\xim (t)&:=\xi_{q_m}((q_{m+1}-q_m)t),\\
		\bxim (t)&:=\bar \xi_{q_m}(t).
	\end{aligned}
\end{equation}
We recall the definitions \eqref{eq:AA} and \eqref{eq:xi_condition} of 
\begin{equation}\label{eq:AA2}
	\AA_{m}:=\Big\{\vec\bx=(\bx_i)_{i=1}^m\in\R^{N\times m}:\,\frac1N\langle\bx_i,\bx_j\rangle=q_{i\wedge j},\,\forall i,j\leq m\Big\} \,.
\end{equation}
and 
\begin{equation}
	\label{eq:xi_condition2}
	\begin{aligned}
		\Cs_{m}(\epsilon):=\Big\{\vec \bx \in \AA_m:\,&\exists (\vec E, \vec R)\in \bar V(\epsilon)\mbox{ such that }\forall i\leq m\\
		&\Big(\frac{1}{N}H_{N}(\bx_i),\,\frac{\partial_{\bx_i-\bx_{i-1}}H_{N}(\bx_i)}{\|\bx_i-\bx_{i-1}\|^2},\,\gradt H_{N}(\bx_i)\Big)=\big(E_i,R_i,0\big) \Big\}\,.
	\end{aligned}
\end{equation} 
They involve 
discrete paths $\vec\bx=(\bx_1,\ldots,\bx_m)$ such that each $\bx_{n+1}$ is a critical point of $H_N(\bx)$ on the
$(N-n-1)$-dimensional sphere 
% of radius $\sqrt{N(q_{n+1}-q_{n})}$
\begin{equation*}\label{eq:nlvlSphere}
\Big\{
\bx_{n}+\by: \|\by\|=\sqrt{N (q_{n+1}-q_{n})}, \,\langle\by,\bx_i\rangle=0,\,\forall i\leq n 
\Big\} \,.
\end{equation*}
In the sequel we show that under an appropriate conditioning, the expected number of critical points on the latter sphere 
%\eqref{eq:nlvlSphere} 
is expressed by the formula \eqref{eq:complexity} for mixture $\xin$ with $I$ and $I'$ appropriately scaled (see \eqref{eq:1lvlComplexity}),  hence by summing over the different levels, one obtains an asymptotic formula for $\E|\Cs_{m}(\epsilon)|$ (see Remark \ref{rem:Cncomplexity}).

Mixtures as in \eqref{eq:xim} were central to \cite{FElandscape}, where certain free energy functionals on spherical bands were studied and a generalized 
%Thouless-Anderson-Palmer 
\abbr{TAP} approach was established (also see \cite{TAPChenPanchenkoSubag,TAPIIChenPanchenkoSubag} for the Ising case). For example, the free energy mixture corresponding to $\bxim$ was used to define in \cite{FElandscape} a  \abbr{TAP} correction that generalizes the classical Onsager correction term. These mixtures were also used before in the 1-\abbr{RSB} case \cite{geometryMixed,geometryGibbs} to study the free energy of spherical bands around critical points.

Most relevant to our current  discussion are certain relations between the Parisi measure of $\xim$ at zero-temperature and that of $\xi$ at inverse-temperature $\beta$, which were established in \cite{FElandscape}.  
Recall the set of overlaps $\mathcal{S}_{P}=\mathcal{S}_{P,\xi,\beta}$ defined by the optimality condition 
of Talagrand in \eqref{eq:Parisi_characterization_finite_beta},  where we now add $\xi$ 
and $\beta$ to the notation,  to make the distinction from its zero-temperature analogue 
%we define in a moment
and between the mixtures $\xi$ and $\bxim$.
\begin{cor}[\cite{FElandscape}]\label{cor:Parisidistbeta}
	If $x_\beta$ is the Parisi distribution of $\xi$ at $\beta$ and $q\in\mathcal{S}_{P,\xi,\beta}$, then
	\[
	\bar x_\beta^q(t) = x_\beta(q+(1-q)t)
	\] 
	is the Parisi distribution of $\bar\xi_q$ of \eqref{eq:xiq} (at $\beta$),
	and 
	\begin{equation}\label{eq:Sq}
	\forall t\in[0,1]:\quad t\in \mathcal{S}_{P,\bar\xi_q,\beta}\iff q+(1-q)t\in \mathcal{S}_{P,\xi,\beta}.
	\end{equation}
\end{cor}
\begin{proof}
	The first claim,  about the relation between the Parisi measures,  is part of \cite[Proposition 11]{FElandscape}. 
	Fixing $\xi$ and $\beta$,  let
	$\phi:=\phi_{\nu_{\beta,P}}$ denote the function defined by \eqref{eq:phi} for the Parisi measure 
	$\nu_{\beta,P}$ and define $\phi_q(s):=\phi (q+(1-q)s)-\phi (q)$. 
	Recall that $\mathcal{S}_{P,\xi,\beta}=\{s\in[0,1]:\phi (s)=\sup_{t\in[0,1]}\phi (t)\}$.  In the proof of 
	\cite[Proposition 11]{FElandscape} it was shown that  $\mathcal{S}_{P,\bar\xi_q,\beta}=\{s\in[0,1]:\phi_q(s)=\sup_{t\in[0,1]}\phi_q(t)\}$ from which \eqref{eq:Sq} follows.
\end{proof}

An analogue of the optimality condition \cite[Proposition 2.1]{Talag} was also proved for the zero-temperature Parisi formula \eqref{eq:ZTParisi}. For any non-decreasing, right continuous and integrable $\alpha:[0,1)\to[0,\infty)$ and $c>0$ define
\begin{align*}
	\Psi_{\alpha,c}(t)&=\xi'(t)-\int_0^t\frac{ds}{(\int_s^1\alpha(r)dr+c)^2},\\
	\psi_{\alpha,c}(t)&=\int_s^1\Psi_{\alpha,c}(t)dt,
\end{align*}
and let $\eta_{\alpha}$ be the measure on $[0,1)$ defined by $\eta_{\alpha}([0,s])=\alpha(s)$. It was proved in \cite{ChenSen} that the minimizer in the Parisi formula \eqref{eq:ZTParisi} is the unique pair $(\alpha,c)$ such that $\Psi_{\alpha,c}(1)=\min_{s\in[0,1]}\psi_{\alpha,c}(s)=0$ and 
\begin{equation}\label{eq:Parisimeasure_ZT}
\mbox{supp}(\eta_{\alpha})\subset \mathcal{S}^{\infty}_{\alpha,c}:=\{s\in[0,1]:\,\psi_{\alpha,c}(s)=0\}.
\end{equation}
Denote by $\mathcal{S}_{P,\xi,\infty}$ the set $\mathcal{S}^{\infty}_{\alpha,c}$ that corresponds
to the minimizer $(\alpha,c)$.

\begin{cor}[\cite{FElandscape}]\label{cor:Parisidistinfty}
	If $x_\beta$ is the Parisi distribution of $\xi$ at $\beta$ and $0<q\in\mathcal{S}_{P,\xi,\beta}$, then
	the minmizer of the Parisi formula \eqref{eq:ZTParisi} for $\hat\xi_q(t)= \xi(qt)$ is
	\[
	(\alpha_q(t),c_q)=\Big(
	\beta x_\beta(qt),\frac{\beta}{q}\int_q^1 x_\beta(s)ds 
	\Big) \,.
	\] 
	Further, 
	\begin{equation}\label{eq:Sq2}
		\forall t\in[0,1]:\quad t\in \mathcal{S}_{P,\hat\xi_q,\infty}\iff qt\in \mathcal{S}_{P,\xi,\beta}.
	\end{equation}
\end{cor}
\begin{proof}
	The first claim, about the relation between the Parisi measures, is part of \cite[Proposition 11]{FElandscape}. 
	Fixing $\xi$ and $\beta$,  recall that with $\phi=\phi_{\nu_{\beta,P}}$ as above, 
	$\mathcal{S}_{P,\xi,\beta}=\{s\in[0,1]:\phi (s)=\sup_{t\in[0,1]}\phi (t)\}$. 
	In the proof of \cite[Proposition 11]{FElandscape}, it was shown that   $\mathcal{S}_{P,\hat\xi_q,\infty}=\{s\in[0,1]:\psi_q(s)=\min_{t\in[0,1]}\psi_q(t)\}$  for $\psi_q(s):=-\phi (qs)/\beta^2$ (using $\phi (q)=0$), 
	from which \eqref{eq:Sq2} follows.
\end{proof}

Assuming \eqref{eq:kRSB}, Corollary \ref{cor:Parisidistbeta} implies that for $1 \le m\leq k$,
the Parisi distribution of $\bxim$ at $\beta$ is
\begin{equation}	
\label{eq:Sbarxim}
\begin{aligned}
	x_\beta^m(t) &= x_\beta(q_m+(1-q_m)t), \qquad  with \\
\mathcal{S}_{P,\bxim,\beta} &= \bigg\{\frac{q_i-q_m}{1-q_m}:\,i=m,\ldots,k  \bigg\}.
\end{aligned}
\end{equation} 
Corollary \ref{cor:Parisidistinfty} with $\xi=\bxim$ and $q=(q_{m+1}-q_{m})/(1-q_{m})$ then implies that
the minimizer of the zero-temperature Parisi formula for $\xim$ is
\begin{align}
	\nonumber (\alpha^m_q(t),c^m_q)&=\Big(
	\beta x_\beta(q_m 
	+ (q_{m+1}-q_m) t
	),\frac{\beta}{q_{m+1}-q_m}\int_{q_{m+1}}^{1} x_\beta(s)ds
	\Big),\quad with  \\
	\label{eq:xim_strict1RSB}\mathcal{S}_{P,\xim,\infty} &= \{0,1\}.
\end{align}
Models such as \eqref{eq:xim_strict1RSB},  where $\mathcal{S}_{P,\xi,\infty}=\{0,1\}$ are called
in \cite{HuangSellke2023} strictly 1-\abbr{RSB}.  The next two results of Huang and Sellke \cite{HuangSellke2023}
about the complexity of critical points of such models will be important to us. 
\begin{lem}[\cite{HuangSellke2023}]\label{lem:strict1RSB}
	If $\xi$ is a mixture such that $\mathcal{S}_{P,\xi,\infty}=\{0,1\}$ and $\Es(1)$ and $\hRs(1)$ are defined by \eqref{eq:GS}, then $$\Theta_{\xi}(\Es(1),\hRs(1))=0.$$
\end{lem}
\begin{proof}
	Assuming $\mathcal{S}_{P,\xi,\infty}=\{0,1\}$,  it is shown in \cite[Lemma 3.18]{HuangSellke2023} 
	that $\Theta_{\xi}(E_0,R_0)=0$ for $(E_0,R_0)$ of \cite[(3.6)]{HuangSellke2023}.  
	Thus,  it suffices to verify that 
	\begin{equation}
		\label{eq:0starequivalence}
		(E_0,R_0)=(\Es(1),\hRs(1)).
	\end{equation} 
	Indeed,  by \cite[Lemma 3.5]{HuangSellke2023},  $E_0=\Es(1)$ and 
	$L,u>0$ of that lemma are such that
	the minimizer of the zero-temperature Parisi formula \eqref{eq:ZTParisi} for $\xi$ is 
	$\alpha(t)\equiv u$ and $c=L-u$.  In the proof of \cite[Lemma 3.5]{HuangSellke2023} it is further 
	shown that $\xi'(1)=(L(L-u))^{-1}$.  We thus get from \cite[(3.6)]{HuangSellke2023} that
	$R_0=(c+u)\xi'(1)+c\xi''(1)$.  In comparison,  
	for $\alpha(t)\equiv u$ we get from \eqref{eq:hRTformula} (at $q=1$),  that also $\hRs(1)=(c+u)\xi'(1)+c\xi''(1)$.
\end{proof}
Relying on \eqref{eq:0starequivalence},  we next re-state \cite[Proposition 3.14]{HuangSellke2023}
in our notation,  while writing $\Crt_{N,\xi}(E,R)$ for $\Crt_{N,\xi}(\{E\},\{R\})$ with singletons.
\begin{prop}[\cite{HuangSellke2023}]\label{prop:HS2}
	Assume that $\mathcal{S}_{P,\xi,\infty}=\{0,1\}$.
	For any sufficiently small $\eta>0$ there exist some $a(\eta),b(\eta)>0$ such that for any $a\leq a(\eta)$, $b\leq b(\eta)$ and $\bs\in\SN$,
	\[
	\lim_{N\to\infty}\sup_{\substack{|E-\Es(1)|\leq a\\ |R-\hRs(1)|\leq a}}\frac1N\log \bigg(
	\P\bigg[
	\frac1N\sup_{\substack{\bx\in\SN:\\ \langle \bx,\bs\rangle=(1-\eta) N}} \{ H_N(\bx) \} \geq \Es(1)-b\,\bigg|\,\bs\in \Crt_{N,\xi}(E,R)
	\bigg]
	\bigg)<0.
	\]
\end{prop}

We observe that the gradient, and thus the radial derivative, are roughly determined at approximate ground states in the following sense. 

\begin{lem}
	\label{lem:rder}Let $q\in(0,1)$ and define
	\[
	L_{N}(q,\epsilon):=\left\{ \bx\in\sqrt{q}\SN:\,\Big|\frac{1}{N}H_{N}(\bx)-\Es(q)\Big|<\epsilon\right\} .
	\]
	Then, for any fixed $\epsilon',\tau>0$, for sufficiently small $\epsilon>0$,
	\begin{align}\label{eq:ddRatGS}
	\varlimsup_{N\to\infty}\frac1N\log\bigg(\P\bigg[\sup_{\bx\in L_{N}(q,\epsilon)}\Big|\frac{1}{\|\bx\|^2}\partial_{\bx}H_{N}(\bx)-\hRs(q)\Big|>\epsilon'\bigg]\bigg)&<0,\\
	\label{eq:gradatGS}
	\varlimsup_{N\to\infty}\frac1N\log\bigg(\P\bigg[\sup_{\bx\in L_{N}(q,\epsilon)} \sup_{\by:\,\by\perp\bx}\Big|\frac{1}{\sqrt{N}\|\by\|}\partial_{\by}H_{N}(\bx)\Big|>\tau\bigg]\bigg)&<0.
	\end{align}
\end{lem}

\begin{rem} We follow the usual convention that the supremum over an empty set is $-\infty$ (but 
this does not really matter,  since by the Borell-TIS inequality, 
$\P(L_{N}(q,\epsilon)=\varnothing) \le e^{-c N}$ for some $c=c(\epsilon)>0$).
\end{rem}

\begin{proof} Fix $\epsilon'>0$ and assume towards contradiction that for any $\epsilon>0$, 
	\begin{align}\label{eq:ddRatGS2}
		\varlimsup_{N\to\infty}\frac1N\log\bigg(\P\bigg[\sup_{\bx\in L_{N}(q,\epsilon)} \frac{1}{\|\bx\|^2}\partial_{\bx}H_{N}(\bx)-\hRs(q) >\epsilon'\bigg]\bigg)=0.
	\end{align}
	By
	\cite[Corollary C.2]{geometryMixed},
	for some constants $C=C(\xi)$ and $c=c(\xi)$, denoting $\BN:=\{\bx:\,\|\bx\|\leq\sqrt{N}\}$,
	\begin{equation}
		\P\Big(\forall\bx,\by\in\BN:\,\|\nabla H_{N}(\bx)-\nabla H_{N}(\by)\|_{{\rm op}}<C\|\bx-\by\|\Big)>1-e^{-cN}.\label{eq:Lip_grad}
	\end{equation}
	Let $t>0$ be a small number to be determined below.  Define $q_{t}:=(\sqrt q+t)^2$. Recall the definition \eqref{eq:GS} of $\Es(q)$.  By the Borell-TIS inequality, for any $\delta>0$, for some $c'=c'(\xi,\delta)>0$,
	\begin{equation}
		\P\left(\max_{\bx\in\sqrt{q}\SN}\frac{1}{N}H_{N}\Big(\bx+t\sqrt{N}\frac{\bx}{\|\bx\|}\Big)=\max_{\bx\in\sqrt{q_{t}}\SN}\frac{1}{N}H_{N} (\bx)<\Es(q_{t})+\delta
		\right)\ge 1-e^{-c'N}.\label{eq:Esqepsilon}
	\end{equation}

On some subsequence of $N\to\infty$,  
with positive probability
% $L_{N}(q,\epsilon)$ are non-empty and
the three events in \eqref{eq:ddRatGS2}-\eqref{eq:Esqepsilon} occur simultaneously.
From  \eqref{eq:Lip_grad} we have that \abbr{whp} for all $\bx \in \BN$, 
	\begin{equation}\label{eq:TaylorApxDerivative}\begin{aligned}
			H_{N}\Big(\bx&+t\sqrt{N}\frac{\bx}{\|\bx\|}\Big) = 
			H_{N}(\bx)+\int_0^{t\sqrt{N}} \partial_{\bx/\|\bx\|} H_{N}\Big(\bx+s\frac{\bx}{\|\bx\|}\Big) ds \\
			&=H_{N}(\bx)+t\sqrt{N} \partial_{\bx/\|\bx\|} H_{N}(\bx)  + \int_0^{t\sqrt{N}} \bigg[\int_0^{s}  \partial_{\bx/\|\bx\|}^2 H_{N}\Big(\bx+w\frac{\bx}{\|\bx\|}\Big) dw\bigg]ds \\
			&
			%\geq H_{N}(\bx)+t\sqrt{N} \partial_{\bx/\|\bx\|} H_{N}(\bx)  - \int_0^{t\sqrt{N}} \bigg[\int_0^{s}  C dw\bigg]ds
			\ge  H_{N}(\bx)+
			t\sqrt{q}N\frac{1}{\|\bx\|^2}\partial_{\bx}H_{N}(\bx)- \frac{N}{2}  C t^2
		\end{aligned}
	\end{equation}
where we used that $t\sqrt{N}\partial_{\bx/\|\bx\|}H_{N}(\bx)=t\sqrt{q}N\frac{1}{\|\bx\|^2}\partial_{\bx}H_{N}(\bx)$.
Thus,  when the events in \eqref{eq:ddRatGS2}-\eqref{eq:Lip_grad} occur
there exists a point $\bx\in L_{N}(q,\epsilon)$ such that
	\begin{equation}\label{eq:HqtLB}
	\frac{1}{N}H_{N}\Big(\bx+t\sqrt{N}\frac{\bx}{\|\bx\|}\Big)\geq 
	\Es(q)-\epsilon
	+t\sqrt{q}\big(\hRs(q)+\epsilon'\big)-\frac{1}{2}Ct^{2}\,.
	\end{equation}
	
Since $\frac{d}{dt}q_{t}|_{t=0}=2\sqrt{q}$ and $\hRs(q):=2\frac{d}{dq}\Es(q)$,
	for all small enough $t>0$,
	\[
	\Es(q_{t})\leq\Es(q)+t\sqrt{q}\big(\hRs(q)+\epsilon'/2\big).		
	\]
	
	Combining  the last two inequalities and canceling like terms, the point as in \eqref{eq:HqtLB} satisfies
		\begin{align*}
		\frac{1}{N}H_{N}\Big(\bx+t\sqrt{N}\frac{\bx}{\|\bx\|}\Big)&\ge
		\Es(q_t)+t\sqrt{q}\epsilon'/2-\epsilon
		-\frac{1}{2}Ct^{2} \ge 
		\Es(q_t)+\delta,
	\end{align*}
for $\epsilon'$ fixed and $\epsilon,\delta,t$ appropriately chosen, (first taking $t$ sufficiently small so that $t\sqrt{q}\epsilon'/2
-\frac{1}{2}Ct^{2}>0$ and then $\epsilon,\delta$ small enough)
	in contradiction to the event in \eqref{eq:Esqepsilon}. It follows that, for some $\epsilon>0$, \eqref{eq:ddRatGS2} does not hold. A similar argument proves the same bound for the opposite inequality (with $-\epsilon'$) and \eqref{eq:ddRatGS} follows.

%	The bound of \eqref{eq:gradatGS} can be proved by a similar idea, by showing that if there exist $\bx\in L_{N}(q,\epsilon)$ and $\by$ such that $\by\perp\bx$ and 	
%	$\frac{1}{\sqrt{N}\|\by\|}\partial_{\by}H_{N}(\bx)>\tau$, then by moving in the direction of $\by$ and projecting back to the sphere $\sqrt{q}\SN$, namely defining $\bz=\sqrt{Nq}(\bx+\eta\by)/\|\bx+\eta\by\|$, if $\epsilon,\delta$ are small then for appropriate choice of $\eta$ we have that $H_N(\bz)/N\geq \Es(q)+\delta$. 
%	
	
The bound of \eqref{eq:gradatGS} can be proved by a similar idea. Fix $\tau>0$ and assume towards contradiction that for any $\epsilon>0$,  
	\begin{align}\label{eq:ddyatGS}
		\varlimsup_{N\to\infty}\frac1N\log\bigg(\P\bigg[\sup_{\bx\in L_{N}(q,\epsilon)} \sup_{\by:\,\by\perp\bx}\Big|\frac{1}{\sqrt{N}\|\by\|}\partial_{\by}H_{N}(\bx)\Big|>\tau\bigg]\bigg)=0.
	\end{align}
Similarly to \eqref{eq:TaylorApxDerivative}, from  \eqref{eq:Lip_grad} we have that \abbr{whp} for all $\bx \in \BN$ and $\by\perp\bx$, 
	\begin{equation*}
		H_{N}\Big(\bx+t\sqrt{N}\frac{\by}{\|\by\|}\Big) 
		\ge  H_{N}(\bx)+
		t\sqrt{N} \partial_{\by/\|\by\|} H_{N}(\bx)- \frac{N}{2}  C t^2\,.
\end{equation*}
From \eqref{eq:ddyatGS}, \abbr{whp} for some $\bx\in L_{N}(q,\epsilon)$ and $\by\perp\bx$,
\begin{equation*}
	\frac1NH_{N}\Big(\bx+t\sqrt{N}\frac{\by}{\|\by\|}\Big) 
	\ge  \Es(q)-\epsilon+
	t\tau - \frac{1}{2}  C t^2\geq \Es(q+t^2)+2\delta-\epsilon\geq \Es(q+t^2)+\delta\,,
\end{equation*}
where the second inequality follows for sufficiently small $t$ and $\delta$ and the third for small $\epsilon$. Since $\bx+t\sqrt{N}\by/\|\by\|\in \sqrt{q+t^2}\SN$, we arrive at a contradiction and \eqref{eq:gradatGS} follows.
\end{proof}

Similarly to $\Cs_{m}(\epsilon)$ of \eqref{eq:xi_condition2},  let
\begin{equation}
	\label{eq:xi_condition_hat}
	\begin{aligned}
		\hat \Cs_{m}(\epsilon)
		:=\Big\{\vec \bx \in \AA_m:\,&\exists (\vec E, \vec R)\in \bar V(\epsilon)\mbox{ such that }\forall i\leq m \\
	&\Big(\frac{1}{N}H_{N}(\bx_i),\,\frac{\partial_{\bx_i}H_{N}(\bx_i)}{\|\bx_i\|^2},\,\gradt H_{N}(\bx_i)\Big)=\big(E_i,R_i,0\big) \Big\}
	\end{aligned}
\end{equation}
and
\begin{equation}
\label{eq:C0meps}
\begin{aligned}
\Cs^0_{m}(\epsilon)
		:=\Big\{\vec \bx \in \AA_m:\,&\exists E_i\in [\Es(q_i)-\epsilon, \Es(q_i)+\epsilon]  \mbox{ such that }\forall i\leq m \\
& \Big(\frac{1}{N}H_{N}(\bx_i),\,\gradt H_{N}(\bx_i)\Big)=\big(E_i,0\big) \Big\}\,.
\end{aligned}
\end{equation} 
\begin{cor}\label{cor:CshCs}
	For any $1 \le m\leq k$ and fixed  $\epsilon'>0$, for sufficiently small $\epsilon>0$,
	\[
	\varlimsup_{N\to\infty} \frac{1}{N}\log\bigg(1-\P\Big[
	 \Cs_{m}(\epsilon)\cup  \hat \Cs_{m}(\epsilon)\subset \Cs^0_{m}(\epsilon) \subset \Cs_{m}(\epsilon')\cap  \hat \Cs_{m}(\epsilon')
	\Big]\bigg)<0.
	\]
\end{cor}
\begin{proof} The first containment follows deterministically by definition of the various sets.  For $\Cs^0_{m}(\epsilon) \subset \hat\Cs_{m}(\epsilon')$,  since $(\vec E, \vec R)\in \bar V(\epsilon)$
 the bound on the probability follows directly from \eqref{eq:ddRatGS} (see \eqref{eq:V}).
	
	To prove that $\Cs^0_{m}(\epsilon) \subset \Cs_{m}(\epsilon')$ with overwhelming probability, take some small $\tau$ and assume that $\epsilon$ is sufficiently small so that \eqref{eq:gradatGS} holds.
	Let $\vec\bx=(\bx_1,\ldots,\bx_m)\in \Cs^0_{m}(\epsilon)$, so that in particular $\bx_i\in L_N(q_i,\epsilon)$.  Define the projection $\bz_i=\langle\bx_i-\bx_{i-1},\bx_i\rangle\bx_i/\|\bx_i\|^2$ and $\by_i=\bx_i-\bx_{i-1}-\bz_i\perp \bx_i$. 
	Since $\vec\bx\in\AA_m$, $\|\bx_i-\bx_{i-1}\|^2=\langle\bz_i,\bx_i\rangle=N(q_i-q_{i-1})$, $\|\bx_i\|^2=Nq_i$ and $\|\by_i\|^2=N(q_i-q_{i-1})q_{i-1}/q_i$.
	Thus,
	\begin{align*}
	\frac{\partial_{\bx_i-\bx_{i-1}}H_{N}(\bx_i)}{\|\bx_i-\bx_{i-1}\|^2}&=\frac{1}{N(q_i-q_{i-1})}\Big(\langle\bz_i,\nabla H_N(\bx_i) \rangle+\langle\by_i,\nabla H_N(\bx_i) \rangle\Big)\\
	&=\frac{1}{N(q_i-q_{i-1})}\Big(\langle\bz_i,\bx_i \rangle\frac{\partial_{\bx_i}H_{N}(\bx_i)}{\|\bx_i\|^2}+\partial_{\by_i} H_N(\bx_i) \Big)\\
	&=\frac{\partial_{\bx_i}H_{N}(\bx_i)}{\|\bx_i\|^2}+\sqrt{\frac{q_{i-1}}{q_i(q_i-q_{i-1})}}\frac{1}{\sqrt{N}\|\by_i\|}\partial_{\by_i}H_{N}(\bx_i).
	\end{align*}
	Recall that $\big|\frac{\partial_{\bx_i}H_{N}(\bx_i)}{\|\bx_i\|^2}-\Rs(q_i)\big|<\epsilon$. Hence, if $\epsilon$ and $\tau$ are sufficiently small, on the complement of the event in \eqref{eq:gradatGS}, $\vec\bx\in\Cs_{m}(\epsilon')$.
\end{proof}

For later use we also relate the ground state energy and its derivative for the mixtures $\xim$ (see \eqref{eq:xim}) to those of the mixture $\xi$ and to make explicit the dependence on $\xi$,  below we use the notations $\Es(\xi,q)$ and $\Rs(\xi,q)$.
\begin{lem}\label{lem:EsRs}
	Assuming \eqref{eq:kRSB}, we have that
	\begin{align}
		\Es(\xim,1)&=\Es(\xi,q_{m+1})-\Es(\xi,q_{m}),\label{eq:Esrelation}\\
		\Rs(\xim,1)&=(q_{m+1}-q_{m})\Rs(\xi,q_{m+1}).\label{eq:Rsrelation}
	\end{align}
\end{lem}
\begin{proof}
While	the lemma can be proved using the Parisi formula,  we instead use the \abbr{TAP} representation of \cite[Theorem 5]{FElandscape} (which by \cite[Theorem 10]{FElandscape} applies to 
overlaps in $\mathcal{S}_{P,\xi,\beta}$).  Specifically,  denote by $F(\xi)$ the limiting free 
energy of the spherical model with mixture $\xi$ at inverse-temperature $\beta$ and,  recalling the definition \eqref{eq:xiq} of $\bar\xi_q(t)$,  define
the operators
	\begin{align*}
	\mathcal{A}_q&:\xi(t)\mapsto \bar\xi_q(t), \\
	\mathcal{K}_q&:\xi(t)\mapsto \xi(qt).
	\end{align*}
The \abbr{TAP} representation of \cite[Theorems 5 and 10]{FElandscape}  state that for any $q\in[0,1)$, 
	\begin{equation}\label{eq:TAPrep1}
		F(\xi)\geq \beta \Es(\xi,q)+\frac12\log(1-q)+F(\mathcal{A}_q(\xi)),
	\end{equation}	
and for any $q\in \mathcal{S}_{P,\xi,\beta}$
		\begin{equation}\label{eq:TAPrepEquality}
			F(\xi)= \beta \Es(\xi,q)+\frac12\log(1-q)+F(\mathcal{A}_q(\xi)).
		\end{equation}
	 By applying the \abbr{TAP} representation again to the free energy on the \abbr{rhs} (namely,  to the mixture $\mathcal{A}_q(\xi)$),  we obtain that for any $\bar q\in[0,1)$,
	\begin{equation}\label{eq:TAPrep2}
		\begin{aligned}
			F(\xi)&\geq \beta \Es(\xi,q)+\frac12\log(1-q)
			%\\			&
			+\beta\Es(\mathcal{A}_q(\xi),\bar q)+\frac12\log(1-\bar q)+F(\mathcal{A}_{\bar q}\circ \mathcal{A}_q(\xi)),
		\end{aligned}
	\end{equation}
	with equality if both $q\in \mathcal{S}_{P,\xi,\beta}$ and $\bar q\in \mathcal{S}_{P,\mathcal{A}_q(\xi),\beta}$.
	
	By assumption, $q_m,q_{m+1}\in\mathcal{S}_{P,\xi,\beta} $ and by  \eqref{eq:Sbarxim} also $\bar q:= \frac{q_{m+1}-q_m}{1-q_m}\in\mathcal{S}_{P,\bxim,\beta}$. One can verify by a direct calculation that $\mathcal{A}_{\bar q}\circ \mathcal{A}_{q_m}(\xi)= \mathcal{A}_{q_{m+1}}(\xi)
	$, $\mathcal{K}_{\bar q}\mathcal{A}_{q_m}(\xi)=\xim$ and
	\begin{equation}\label{eq:log}
		\begin{aligned}
			\Es(\mathcal{A}_{q_m}(\xi),\bar q) &= \Es(\mathcal{K}_{\bar q}\mathcal{A}_{q_m}(\xi),1)=\Es(\xim,1),\\
			\log(1-\bar q)&=\log(1-q_{m+1})-\log(1-q_m).
		\end{aligned} 
	\end{equation}
	Hence, by using the first representation \eqref{eq:TAPrepEquality} for $F(\xi)$ with $q=q_{m+1}$ and the second 
	representation \eqref{eq:TAPrep2} with $q=q_m$ and $\bar q$ as above,  we arrive at  \eqref{eq:Esrelation}
	upon canceling the like-terms.

Now, for $q>q_m$, consider the second representation \eqref{eq:TAPrep2} with $q$ replaced by $q_m$ and $\bar q:= \frac{q-q_{m}}{1-q_m}$. Namely,
\begin{equation}\label{eq:TAPrep3}
	\begin{aligned}
		F(\xi)&\geq \beta \Es(\xi,q_m)+\frac12\log(1-q_m)
		%\\			&
		+\beta\Es(\mathcal{A}_{q_m}(\xi),\bar q)+\frac12\log(1-\bar q)+F(\mathcal{A}_{\bar q}\circ \mathcal{A}_{q_m}(\xi))\\
		&=\beta \Es(\xi,q_m)+\beta\Es\Big(\xim,\frac{q-q_m}{q_{m+1}-q_m}\Big)+\frac12\log(1-q)
		+F(\mathcal{A}_{q}(\xi))\\
	\end{aligned}
\end{equation}
where we used $\mathcal{A}_{\bar q}\circ \mathcal{A}_{q_m}(\xi)= \mathcal{A}_{q}(\xi)$ and similarly to \eqref{eq:log},
	\begin{equation*}
	\begin{aligned}
		\Es(\mathcal{A}_{q_m}(\xi),\bar q) &= \Es(\mathcal{K}_{\bar q}\mathcal{A}_{q_m}(\xi),1)=\Es\Big(\xim,\frac{q-q_m}{q_{m+1}-q_m}\Big),\\
		\log(1-\bar q)&=\log(1-q_{m+1})-\log(1-q_m).
	\end{aligned} 
\end{equation*}

Both representations \eqref{eq:TAPrepEquality} and \eqref{eq:TAPrep3} hold with equality for $q=q_{m+1}$. Since the the \abbr{lhs} does not depend on  $q$, the derivative in $q$ of the \abbr{rhs} of each of them is equal to zero. In particular, by comparing those two derivative and canceling like-terms we have
\[
\frac{d}{dq}\Big|_{q=q_{m+1}}\Es(\xi,q)=  \frac{d}{dq}\Big|_{q=q_m}\Es\Big(\xim,\frac{q-q_m}{q_{m+1}-q_m}\Big),
\]
from which \eqref{eq:Rsrelation} follows by the definition of $\Rs(q):=2\frac{d}{dq}\Es(q)$.
%	Now,  consider the second representation \eqref{eq:TAPrep2} with $\bar q:= \frac{q_{m+1}-q}{1-q}$
%	and $q<q_{m+1}$.  This choice still gives $\mathcal{A}_{\bar q}\circ \mathcal{A}_q(\xi)= \mathcal{A}_{q_{m+1}}(\xi)
%	$ and \eqref{eq:log} with $q_m=q$. Moreover, for those parameters, one can verify that
%	\[
%	\Es(\mathcal{A}_q(\xi),\bar q) = \Es(\mathcal{K}_{\bar q}\mathcal{A}_q(\xi),1)= \Es\Big(\xim,\frac{q_{m+1}-q}{q_{m+1}-q_m}\Big). 
%	\] 
%	We thus have
%	\begin{align*}
%		F(\xi) &
%		\geq \beta \Es(\xi,q)
%		%\\		&
%		+\beta\Es\Big(\xim,\frac{q_{m+1}-q}{q_{m+1}-q_m}\Big)+\frac12\log(1- q_{m+1})+F( \mathcal{A}_{q_{m+1}}(\xi)),
%	\end{align*}
%	with equality for $q=q_m$.  The derivative in $q$ of the \abbr{rhs} must therefore be zero at $q=q_m$.  
%	That is,
%	\[
%	0=\beta \frac{d}{dq}\Big|_{q=q_m}\Es(\xi,q)+\beta \frac{d}{dq}\Big|_{q=q_m}\Es\Big(\xim,\frac{q_{m+1}-q}{q_{m+1}-q_m}\Big),
%	\]
%	from which \eqref{eq:Rsrelation} follows by the definition of $\Rs(q):=2\frac{d}{dq}\Es(q)$.
\end{proof}

For later use, we record the following immediate corollary of Lemmas \ref{lem:EsRs} and \ref{lem:strict1RSB} and \eqref{eq:xim_strict1RSB}.
\begin{cor}\label{cor:Thetastar}
	Assuming \eqref{eq:kRSB}, for $\Es(q_m)$ and $\Rs(q_m)$ as in \eqref{eq:GS},
	\[
	\Theta_{\xim}\Big(\Es(q_{m+1})-\Es(q_{m}),(q_{m+1}-q_m)\Rs(q_{m+1})\Big)=0.
	\]
\end{cor}

\section{\label{sec:KacRice}Kac-Rice formulas}

The results of this section do not assume condition \eqref{eq:kRSB}.  Thus,  while 
we fix throughout some $0=q_0 < q_1 <\cdots <q_k<1$,  these $\{q_i\}$ do not have to be related to the Parisi measure.  
We recall that $\gradt H_{N}(\bx_i):=M_{i}\nabla H_{N}(\bx_i)$ for $M_i=M_i(\vec\bx)\in\R^{(N-i)\times N}$ an arbitrary matrix 
whose rows form an orthonormal basis of $\big(\sp\{\bx_1,\ldots,\bx_i\}\big)^{\perp}$ and 
by an abuse of notation,  for $m\leq k$ and $\vec\bx\in\AA_m$ of \eqref{eq:AA2},   define the event
\begin{equation}\label{eq:CPi}
\cpt(\vec{\bx},\vec E,\vec R) := \bigcap_{1\leq i \le m} \Big\{
	\Big(\frac{1}{N}H_{N}(\bx_i),\,\frac{\partial_{\bx_i-\bx_{i-1}}H_{N}(\bx_i)}{\|\bx_i-\bx_{i-1}\|^2},\,\gradt H_{N}(\bx_i)\Big)=\big(E_i,R_i,0\big) \Big\}\,.
\end{equation}
While the
set of indices $i$ in \eqref{eq:CPi} 
% holds in the event $\cpt(\vec{\bx},\vec E,\vec R)$ 
is determined by the size of the vector $\vec\bx$,  the vectors $\vec E$ and $\vec R$ are allowed to be longer
(and $\mathscr{C}_m(\epsilon)$ of \eqref{eq:xi_condition2} is merely 
the subset of $\AA_m$ for which $\cup_{\bar V(\epsilon)} \cpt(\vec{\bx},\vec E,\vec R)$ occurs).
Fixing $n\leq k$ we consider a non-random  
$\bar \alpha_{N,n}(\vec\bx,\varphi): \AA_n\times C^\infty(\BN) \mapsto \R_+$ which as in Theorem \ref{thm:low temp 1}
% that $\bar \alpha_{N,n}$
 is rotationally invariant in the sense that $\bar \alpha_{N,n}(\vec\bx,\varphi)=\bar \alpha_{N,n}(O\vec\bx,\varphi(O^T\cdot))$ for any orthogonal $O$ (where $O\vec\bx:=(O\bx_1,\ldots,O\bx_n)$),  and such that 
for any $\varphi,\psi_{i} \in C^\infty(\BN)$ and $d\geq1$, 
\begin{equation}
	\lim_{t_{i}\to0}\bar\alpha_{N,n}\Big(\vec\bx,\varphi+\sum_{i\leq d}t_{i}\psi_{i}\Big)=\bar\alpha_{N,n}(\vec\bx,\varphi)
	\,.
	\label{eq:alphalimits}
\end{equation}
Our goal in this section is to prove the following multi-level Kac-Rice formula. To lighten the notation we often omit $N$ from various functions,  
such as $\alpha_{n}(\vec\bx):=\bar \alpha_{N,n}(\vec\bx,H_{N}(\cdot))$.
\begin{prop}\label{prop:multilvlKR}
Suppose that $\alpha_{n}(\vec\bx)$ as above,  is 
%for any $(\vec E,\vec R)\in \bar V(\epsilon)$,   
$a.s.\ $~continuous in the topology induced on $\AA_n$ from $\R^{N\times n}$.
%conditionally on $\cpt(\vec \bx_n^\be,\vec E,\vec R)$.  
For $\epsilon>0$ and $0 \le m < n$,  set 
\begin{align}\label{eq:Im}
		I_m(\epsilon)&=[\Es(q_{m+1})-\Es(q_m)-2\epsilon,\Es(q_{m+1})-\Es(q_m)+2\epsilon],\\
\label{eq:Imprime}
		I'_m(\epsilon)&=(q_{m+1}-q_m)[\Rs(q_{m+1})-\epsilon,\Rs(q_{m+1})+\epsilon].
\end{align}
Then,  for any finite $M$,   
	\begin{align}\label{eq:multilvlKR}
		\varlimsup_{N\to\infty}\frac1N\log \E\Big[ \big(\sum_{\vec\bx\in \Cs_{n}(\epsilon)}
		\alpha_n(\vec\bx) \big) \wedge e^{MN} \Big] 
		& \leq \sum_{m=0}^{n-1}\sup_{(E,R) \in I_m(\epsilon)\times I'_m(\epsilon)}\Theta_{\xim}(E,R) \\
		& +
		\varlimsup_{N\to\infty}\frac1N\log\Big(\sup_{\bar V(\epsilon)}\E\Big[ \alpha_n(\vec\bx_{n}^{\be}) \wedge e^{MN}  \,\Big|\,\cpt(\vec \bx_{n}^\be,\vec E,\vec R) \Big]\Big)\,.  
		\nonumber
	\end{align}
\end{prop}
\begin{rem}\label{rem:Cncomplexity}
In particular,  taking $\bar \alpha_{N,n}(\cdot)\equiv 1$ we arrive at 
	\[
	\varlimsup_{N\to\infty}\frac1N\log \E\Big[ |\Cs_{n}(\epsilon)| \wedge e^{MN}\Big] \leq \sum_{m=0}^{n-1}
	\sup_{(E,R) \in I_m(\epsilon)\times I'_m(\epsilon)}\Theta_{\xim}(E,R) \,.
	\]
	For $q_i$ the atoms of the Parisi measure as in \eqref{eq:kRSB}, typically $\Cs_{n}(\epsilon)\neq\varnothing$ (for it contains the paths of the tree of pure states, 
	as  in Proposition \ref{prop:PS}),  but as $\epsilon\to0$,  the preceding bound on its 
	exponential in $N$ growth rate,  goes to zero.
\end{rem}
In Section \ref{sec:multiLvlKR} we prove Proposition \ref{prop:multilvlKR} by induction on $m$,  where in each step we 
apply the one-level Kac-Rice formulas of Section \ref{subsec-one-level-KR}.   Towards this,  we first 
provide in Section \ref{subsec-cond-law} explicit representations of $H_N(\cdot)$
and its derivatives,  on the relevant sub-spaces ($S_0(\vec \bx_m^\be)$ below),  conditional on   
$\cpt(\vec\bx_m^\be,\vec E,\vec R)$.
\subsection{Conditional laws}\label{subsec-cond-law}

We start with the conditional law of $H_N(\cdot)$ given $\cpt(\vec{\bx},\vec E,\vec R)$,  for $\vec{\bx} \in \AA_m$, 
on certain $(N-m)$-dimensional linear sub-spaces.
% $S_0(\vec{\bx})$
\begin{lem}\label{lem:conditional_law_on_band}
	Let $m\leq k$ and $\vec\bx=(\bx_1,\ldots,\bx_m)\in\AA_m$. For any 
	\begin{equation}\label{eq:yi}
		\by_1,\by_2\in S_0(\vec\bx):=\Big\{\by\in\R^N:\,\forall i\leq m,\, \frac1N\langle\by, \bx_i \rangle=q_i\Big\},
	\end{equation}
	(such that $\xi(\|\by_i\|^2/N)<\infty$) we have that 
	\begin{equation}\label{eq:conditional_law}
		\E \Big[ H_N(\by_1)\,\Big|\,  \cpt(\vec\bx,\vec E,\vec R)\Big]=N E_m
	\end{equation} 
	and
	\begin{equation}
		\label{eq:conditional_law2}
		\begin{aligned}
			&\E \Big[(H_N(\by_1)-N E_m)(H_N(\by_2)-N E_m)\,\Big|\,  \cpt(\vec\bx,\vec E,\vec R)\Big]\\ 
			&= N\Big(\xi(\langle \by_1,\by_2\rangle/N)-\xi(q_m)-\xi'(q_m)\langle \by_1-\bx_m,\by_2-\bx_m\rangle/N \Big)
			=
			 N\xi_{q_m}(\langle \by_1,\by_2\rangle/N-q_m),
		\end{aligned}
	\end{equation}
for the mixture $\xi_q$ of \eqref{eq:xiq}. 
\end{lem}
\begin{proof}
	We implicitly assume throughout that $\xi(\|\by_i\|^2/N)<\infty$.
	Consider first the case $m=1$ and set $[\by,\bx]:=\langle\by,\bx\rangle/N$.  We then have that 
	for any unit vectors $\bu\perp \bx_1$, $\bu_1,\bu_2$ and $\by_1\in S_0(\bx_1)$, 
	\begin{align}\label{eq:H-diff}
		\E\big[(H_N(\by_1)-H_N(\bx_1))H_N(\bx_1)\big]&=N(\xi( [\by_1,\bx_1])-\xi([ \bx_1,\bx_1]))=0,\\
		\label{eq:cov-H-der}
		\E\big[H_N(\bx_1) \partial_{\bu_1}H_N(\bx)|_{\bx=\bx_1}\big]&=N\partial_{\bu_1}\xi([\bx_1,\bx])|_{\bx=\bx_1}=\xi'(q_1)\langle\bx_1,\bu_1\rangle,\\
		\label{eq:cov-H-ort}
		\E\big[(H_N(\by_1)-H_N(\bx_1))\partial_{\bu}H_N(\bx)|_{\bx=\bx_1}\big]&=
		%N\partial_{\bu}(\xi([\by_1,\bx])-\xi([\bx_1,\bx]))|_{\bx=\bx_1}=
		\xi'(q_1)\langle\by_1,\bu\rangle - \xi'(q_1) \langle \bx_1,\bu \rangle 
		= \xi'(q_1)\langle\by_1,\bu\rangle,\\	
		\label{eq:cov-H-x1}
			\E\big[(H_N(\by_1)-H_N(\bx_1))\partial_{\bx_1}H_N(\bx)|_{\bx=\bx_1}\big]&=
			\xi'(q_1)\langle\by_1,\bx_1\rangle - \xi'(q_1) \langle \bx_1,\bx_1 \rangle
			%N\partial_{\bx_1}(\xi([\by_1,\bx])-\xi([\bx_1,\bx]))|_{\bx=\bx_1}
			=0,\\	
		\E\big[\partial_{\bu_1}H_N(\bx)|_{\bx=\bx_1}\partial_{\bu_2}H_N(\by)|_{\by=\bx_1}\big]
	    &=N\partial_{\bu_1}\partial_{\bu_2}\xi([\bx,\by])|_{\bx=\bx_1,\by=\bx_1} \nonumber\\	
	    &
		=\xi'(q_1)\langle\bu_1,\bu_2\rangle+\frac1N\xi''(q_1)\langle\bx_1,\bu_1\rangle\langle\bx_1,\bu_2\rangle,
		\label{eq:cov-der}
	\end{align}
	where $\partial_{\bv}f(\bx)=\langle\nabla f(\bx),\bv\rangle$ 
	and on the \abbr{rhs} we apply 
	$\partial_{\bu_1}$ \abbr{wrt} the argument $\bx$ and $\partial_{\bu_2}$ \abbr{wrt} $\by$. 
	
	Hence, $H_N(\by_1)-H_N(\bx_1)$ and $\gradt H_{N}(\bx_1)$ are independent of $H_{N}(\bx_1)$ and $\partial_{\bx_1} H_{N}(\bx_1)$. The conditional mean of $H_N(\by_1)-H_N(\bx_1)$ given $\cpt(\vec\bx,\vec E,\vec R)$ is therefore equal to its conditional mean given that $\gradt H_{N}(\bx_1)=0$. That is, it is equal to $0$ and \eqref{eq:conditional_law} follows. Moreover, the conditional covariance of $H_N(\by_1)-H_N(\bx_1)$ and $H_N(\by_2)-H_N(\bx_1)$ given $\cpt(\vec\bx,\vec E,\vec R)$ is equal to the conditional covariance only given that $\gradt H_{N}(\bx_1)=0$.
	From the well-known formulas for the conditional law of jointly Gaussian variables, 
	\begin{align*}
		&\E \big[(H_N(\by_1)-H_N(\bx_1))(H_N(\by_2)-H_N(\bx_1)) \,|\,\gradt H_{N}(\bx_1)=0 \big]\\
		&=\E \big[(H_N(\by_1)-H_N(\bx_1))(H_N(\by_2)-H_N(\bx_1))\big] \\
		&- \E\big[(H_N(\by_1)-H_N(\bx_1))\gradt H_{N}(\bx_1)^{\top}\big]
		\big(\E\big[\gradt H_{N}(\bx_1)\gradt H_{N}(\bx_1)^{\top}\big]\big)^{-1}\E\big[(H_N(\by_2)-H_N(\bx_1))\gradt H_{N}(\bx_1)\big]\,.
	\end{align*}
	Note that the first expectation on the right-side 
	is $N \xi([\by_1,\by_2])-N \xi(q_1)$.  Further,  the entries of 
	$\E\big[\gradt H_{N}(\bx_1)\gradt H_{N}(\bx_1)^{\top}\big]$ correspond to \eqref{eq:cov-der}
	for an orthonormal basis $\{\bu_j\}$ of $\bx_1^{\perp}$,  so this matrix is merely $\xi'(q_1)\mathbf{I}$.  Likewise,  by \eqref{eq:cov-H-ort},
	$$
	\E\big[(H_N(\by_i)-H_N(\bx_1))\gradt H_{N}(\bx_1)^{\top}\big] = \xi'(q_1) (\langle \by_i,\bu_1\rangle,\ldots,
	\langle \by_i,\bu_{N-1}\rangle) \,,
	$$ 
	and we arrive at \eqref{eq:conditional_law2} upon utilizing the obvious identity
	$\langle \by_1-\bx_1,\by_2-\bx_1\rangle = \sum_{j} \langle \by_1,\bu_j\rangle \langle \by_2,\bu_j\rangle$.

	We proceed by induction on $m$,  setting $\vec\bx_m=(\bx_1,\ldots,\bx_m)\in\AA_m$,
	 $\vec\bx_{m+1}=(\bx_1,\ldots,\bx_{m+1}) \in\AA_{m+1}$ and assuming that the lemma holds for $m$.  Defining $\hat H_N(\by)=H_N(\by+\bx_m)-H_N(\bx_m)$, we have by our induction hypothesis
 that conditional on $\cpt(\vec\bx_{m},\vec E,\vec R)$,  on $S_{\perp}(\vec\bx_m):=\{\by:\,\by\perp\bx_i,\forall i\leq m\}$,
	the process $\hat H_N(\by)$ has the same law as the mixed model with the mixture $\xi_{q_m}$ of \eqref{eq:xiq}.  Note that the conditional law given $\cpt(\vec\bx_{m+1},\vec E,\vec R)$ can be obtained  by first conditioning on $\cpt(\vec\bx_{m},\vec E,\vec R)$ and then further conditioning on 
	\begin{equation*}%\label{eq:CPi}
		\Big(\frac{1}{N}\hat H_{N}(\hat \bx_{m+1}),\,\frac{\partial_{\hat \bx_{m+1}}\hat H_{N}(\hat \bx_{m+1})}{\|\hat \bx_{m+1}\|^2},\,\gradt \hat H_N(\hat \bx_{m+1}) \Big)=\big(\hat{E}_{m+1},
		R_{m+1},0\big),%\label{eq:xi_condition}
	\end{equation*}
	where $\hat\bx_{m+1}:=\bx_{m+1}-\bx_{m}$,  
	$\hat E_{m+1} := E_{m+1}-E_m$ and
	$\gradt \hat H_N(\by) := M_{m+1}(\vec\bx_{m+1})\nabla \hat H_N(\by) $ (with $M_{m+1}$ as 
	in \eqref{eq:CPi}).  Hence,  by the lemma at $m=1$,  
	we have that for any $\by_1,\by_2\in S_{\perp}(\vec\bx_m)$ such that $[\by_i,\hat \bx_{m+1}]=q_{m+1}-q_m$, 
	\begin{equation*}%\label{eq:conditional_law}
		\E \Big[ \hat H_N(\by_1)\,\Big|\,  \cpt(\vec\bx_{m+1},\vec E,\vec R)\Big]=N \hat E_{m+1}
	\end{equation*} 
	and,  denoting $\hat q:=q_{m+1}-q_m$,
	\begin{equation}\label{eq:induct}
		\begin{aligned}
			&\E \Big[(\hat H_N(\by_1)-N \hat E_{m+1})(\hat H_N(\by_2)-N \hat E_{m+1})\,\Big|\,  \cpt(\vec\bx_{m+1},\vec E,\vec R)\Big]  \\
			&=N\Big(
			\xi_{q_{m}}([ \by_1,\by_2])-\xi_{q_{m}}(\hat q)-\xi_{q_{m}}'(\hat q)
			[ \by_1-\hat \bx_{m+1},\by_2-\hat \bx_{m+1}]
			\Big).
		\end{aligned}
	\end{equation}
%	Now,  note that $\check{\by}_i := \by_i+\bx_m\in S_0(\vec\bx_{m+1})$ and as
%	$[\by_1,\by_2]=[ \check{\by}_1,\check{\by}_2]-q_m$,   one can verify by straightforward algebra that the right-side of \eqref{eq:induct} equals to 
%	$N\xi_{q_{m+1}}([ \check{\by}_1,\check{\by}_2]-q_{m+1})$ (see \eqref{eq:xiq}).  This establishes
%	the lemma for $m+1$ and concludes the proof.
Now,  note that $\check{\by}_i := \by_i+\bx_m\in S_0(\vec\bx_{m+1})$, 
	$[\by_1,\by_2]=[ \check{\by}_1,\check{\by}_2]-q_m$ and 
	\[[ \by_1-\hat \bx_{m+1},\by_2-\hat \bx_{m+1}] = [ \check{\by}_1,\check{\by}_2]-q_{m+1}.\]
	By substituting the above and using \eqref{eq:xiq}, we obtain that the \abbr{RHS} of 
	\eqref{eq:induct} is equal to 
	\[
	N\Big(\xi([ \check{\by}_1,\check{\by}_2])-\xi(q_{m+1})-\xi'(q_{m+1})([ \check{\by}_1,\check{\by}_2]-q_{m+1})\Big)=N\xi_{q_{m+1}}([ \check{\by}_1,\check{\by}_2]-q_{m+1}).
	\]
	This establishes
	the lemma for $m+1$ and concludes the proof.
\end{proof}

Recall that $\vec\bx_m^\be=(\bx_1^\be,\ldots, \bx_m^\be)\in\AA_m$ for 
$\bx^{\be}_i := \sum_{j=1}^i \sqrt{N(q_j-q_{j-1})} \be_j$.
We next re-write the representation of Lemma \ref{lem:conditional_law_on_band} in terms of
the linear mapping $T^{(m)}:\R^{N-m} \to S_0(\vec\bx_m^{\be})$,  given by
% =\{\by: \langle\by,\bx_i^\be\rangle/N=q_i,\,\forall i\leq m \}$
\begin{align*}
	T^{(m)}(\bz) &= T^{(m)}((z_1,\ldots,z_{N-m})) = \bx_{m}^{\be}+\sqrt{{\textstyle \frac{N}{N-m}}(q_{m+1}-q_m)}(0,\ldots,0,z_1,\ldots,z_{N-m})\,,
\end{align*}
and note in passing that $T^{(m)}$ maps $\mathbb{S}^{N-m-1}=\{\bx\in \R^{N-m}: \|\bx\|^2=N-m\}$ to  $S(\vec\bx_{m}^{\be})$ of \eqref{eq:S(y)}.
\begin{cor}\label{cor:Hm1}
	As a process on $\mathbb{S}^{N-m-1}$, conditional on $\cpt(\vec\bx^{\be}_{m},\vec E,\vec R)$, 
	\begin{equation}\label{eq:Hm}
		H_N^{(m)}(\bz):=\sqrt{\frac{N-m}{N}} \Big( H_N\big(T^{(m)}(\bz)\big) - NE_m\Big)
	\end{equation} 
	has the same law as the mixed spherical model with the mixture $\xim$ of  \eqref{eq:xim}.

	Moreover, conditional on $\cpt(\vec\bx^{\be}_{m},\vec E,\vec R)$, if $\by=T^{(m)}(\bz)$ and 
	\begin{equation*}
		\Big(\frac{1}{N}H_{N}(\by),\,\frac{\partial_{\by-\bx^{\be}_{m}}H_{N}(\by)}{\|\by-\bx^{\be}_{m}\|^2}\Big)=\big(E,R\big),
	\end{equation*}
	then
	\begin{equation}\label{eq:ERtransform}
		\Big(\frac{1}{N-m} H^{(m)}_{N}(\bz),\,\frac{\partial_{\bz} H^{(m)}_{N}(\bz)}{\|\bz\|^2}\Big)=\sqrt{\frac{N}{N-m}}\big((E-E_m),\,(q_{m+1}-q_m)R\big).
	\end{equation}
\end{cor}
\begin{proof}
	For $\bz_1,\bz_2\in \mathbb{S}^{N-m-1}$,
	\[
	\frac1N \langle T^{(m)}(\bz_1), T^{(m)}(\bz_2) \rangle = q_m+\frac{q_{m+1}-q_m}{N-m}
	\langle \bz_1, \bz_2 \rangle\,.
	\]
	Hence, by \eqref{eq:conditional_law} and \eqref{eq:conditional_law2}, 
	$\E[ H_N^{(m)}(\bz_1)\,|\, \cpt(\vec\bx,\vec E,\vec R)]=0$ and 
	\[
	\E \big[H_N^{(m)}(\bz_1)H_N^{(m)}(\bz_2)\,\big|\,\cpt(\vec\bx,\vec E,\vec R)\big]=(N-m)\xim \Big(
	\frac{1}{N-m}\langle \bz_1,\bz_2\rangle
	\Big),
	\]
	and our first claim thus follows.
	
	Next,  from the definitions it follows that,  if $\by=T^{(m)}(\bz)$,
	then $\partial_{\by-\bx^{\be}_{m}}H_{N}(\by)=\sqrt {\frac{N}{N-m}}\partial_{\bz} H^{(m)}_{N}(\bz)$. Since $\|T^{(m)}(\bz)-\bx_m^\be\|^2/\|\bz\|^2=N(q_{m+1}-q_m)/(N-m)$, the second claim follows.
\end{proof}

Building on Corollary \ref{cor:Hm1} we proceed to represent the conditional 
on $\cpt(\vec \bx_{m}^\be,\vec E,\vec R)$ joint law of $H^{(m)}_N$ and certain
first and second derivatives of it.  Specifically,  let
$(F_i)_{i=1}^{N-m-1}$ be some piecewise smooth orthonormal frame field on the sphere $\mathbb{S}^{N-m-1}$ and for $\bz\in\mathbb{S}^{N-m-1}$ define 
$\grad H^{(m)}_{N}(\bz) := (F_iH^{(m)}_N(\bz))_{i=1}^{N-m-1}$ and $\Hess H^{(m)}_{N}(\bz) := (F_iF_jH^{(m)}_N(\bz))_{i,j=1}^{N-m-1}$.  Recall that the $d$-dimensional \abbr{GOE} matrix $\bM$ is a real symmetric matrix whose elements are centered Gaussian variables, independent of each other, up to symmetry,  and throughout
we shall use the normalization $\E\big[\bM_{ij}^2\big]=(1+\delta_{ij})/d$.
\begin{cor}\label{cor:Hm2}
	Let $\bz\in \mathbb{S}^{N-m-1}$ be an arbitrary point and define the matrix
	\begin{equation*}
		\bG := \sqrt{\frac{N-m}{(N-m-1)\xim''(1)}}\Bigg( \Hess H^{(m)}_N(\bz)+\frac{\partial_{\bz}H^{(m)}_N(\bz)}{\|\bz\|^2}\mathbf{I} \Bigg).
	\end{equation*}
	Conditional on $\cpt(\vec \bx_{m}^\be,\vec E,\vec R)$, the three random variables
	\[
	U_N(\bz):=\Big(\frac1{N-m} H^{(m)}_N(\bz),\frac{\partial_{\bz}H^{(m)}_N(\bz)}{\|\bz\|^2}\Big),\qquad\grad H^{(m)}_N(\bz)\qquad\mbox{and }\qquad\bG
	\]
	are independent. Moreover, with $\Sigma_{\xi}$ as defined in \eqref{eq:Sigmaxi},
	we have that $U_N(\bz)\sim \N(0,(N-m)^{-1}\Sigma_{\xim})$, $\grad H^{(m)}_N(\bz)\sim \N(0,\xim'(1)\mathbf{I})$ and
	$\bG$ is an $(N-m-1)$-dimensional \abbr{GOE} matrix.
\end{cor}
\begin{proof}
	In \cite[Section 4.1]{geometryMixed} this corollary is proved for the Hamiltonian of a general spherical model on $\SN$. Here we translated this result to our setting on $\mathbb{S}^{N-m-1}$ with the mixture $\xim$, which corresponds to the conditional law of $H_N^{(m)}(\bz)$ by Corollary \ref{cor:Hm1}.
\end{proof}

\subsection{One-level Kac-Rice formulas}\label{subsec-one-level-KR}
For any $\vec\by=(\by_1,\ldots,\by_m)\in \AA_m$ and $m\leq n\leq k$,  we define
\begin{equation}\label{eq:Amn-y}
	\AA_{m,n}(\vec\by):=\Big\{\vec\bx=(\bx_i)_{i=1}^{n}\in\AA_{n}:\,(\bx_1,\ldots,\bx_m)=(\by_1,\ldots,\by_m)\Big\}
\end{equation}
and similarly to \eqref{eq:xi_condition} denote
%denote by $\Cs_{m,n}(\epsilon,\vec\by)$ the set of all $\vec\bx \in\AA_{m,n}(\vec\by)$ 
%such that for some $(\vec E, \vec R)\in \bar V(\epsilon)$, property \eqref{eq:xi_condition} 
%is satisfied for any $m+1\leq i\leq n$.
\begin{equation}
	\label{eq:xi_condition_mn}
	\begin{aligned}
		\Cs_{m,n}(\epsilon,\vec\by):=\Big\{\vec \bx \in \AA_{m,n}(\vec\by):\,&\exists (\vec E, \vec R)\in \bar V(\epsilon)\mbox{ such that }\forall m+1\leq i\leq n,\\
		&\Big(\frac{1}{N}H_{N}(\bx_i),\,\frac{\partial_{\bx_i-\bx_{i-1}}H_{N}(\bx_i)}{\|\bx_i-\bx_{i-1}\|^2},\,\gradt H_{N}(\bx_i)\Big)=\big(E_i,R_i,0\big) \Big\}\,.
	\end{aligned}
\end{equation} 

By going through the proof of \eqref{eq:complexity} in \cite{geometryMixed}, one can verify that multiplying the mixture by a factor that goes to one as $N\to\infty$ or scaling $I$ and $I'$ by such a factor do not affect the asymptotic result. 
Hence, from Corollary \ref{cor:Hm1}, we have that
\begin{equation}\label{eq:1lvlComplexity}
	\lim_{N\to\infty}\frac1N\log\Big(\E\big[
	\#\Cs_{m,m+1}(\epsilon,\vec\bx_m^{\be})\,\big|\, \cpt(\vec\bx^{\be}_{m},\vec E,\vec R)
	\big]\Big)
	=
	\sup_{I_m(\epsilon,E_m)\times I'_m(\epsilon)}\Theta_{\xim}(E,R),
\end{equation}
for $I'_m(\cdot)$ of \eqref{eq:Imprime} and 
$I_m(\epsilon,E_m)=[\Es(q_{m+1})-E_m-\epsilon,\Es(q_{m+1})-E_m+\epsilon]$.

Our next lemma provides the one-level Kac-Rice formula we need for our inductive argument. 
To this end,  we further set $\vec \bx_0^\be=\bx_0^\be:=0$, $\Cs_{0,1}(\epsilon,\bx_0^\be):=\Cs_{1}(\epsilon)$, $\AA_{0,1}:=\AA_1$ with $\cpt(\vec \bx_0^\be,\vec E,\vec R)$ the entire sample space.  We recall that  $\BN\subset\R^N$ denotes the ball of radius $\sqrt {N r}$  for some $1 \le r < 1 + \epsilon_\xi$.

\begin{lem}\label{lem:1lvlInductionIndic}
	Fix $0\leq m\leq k-1$ and a non-random $\bar w_N(\bx,\varphi): 
	\sqrt{q_{m+1}}\,\SN\times C^\infty(\BN) \to \R$,  assuming as in Theorem \ref{thm:low temp 1},
	 that $\bar w_N$ is rotationally invariant and that 
	for any $\varphi,\psi_{i}\in C^\infty(\BN)$ and $d\geq1$, 
	\begin{equation}
		\lim_{t_{i}\to0}\bar w_{N}\Big(\bx,\varphi+\sum_{i\leq d}t_{i}\psi_{i}\Big)=\bar w_{N}(\bx,\varphi).
		\label{eq:glim2}
	\end{equation}
	Suppose further that for any $(\vec E,\vec R)\in \bar V(\epsilon)$, 
	the function $w_{N}(\bx):=\bar w_{N}(\bx,H_{N}(\cdot))$ is 
	 a.s.\ continuous on\footnote{\label{fn:C}In fact, as can be verified by going through the proof, it is enough to work only on the set $\{\bx_{m+1}: (\vec\bx_m^{\be},\bx_{m+1})\in \AA_{m,m+1}\}$ and require continuity and finiteness of the second moment of $w_N(\bx)$ only there.} $\sqrt {q_{m+1}}\,\SN$ conditionally on $\cpt(\vec \bx_m^\be,\vec E,\vec R)$ and
	$\E[w_{N}(\bx)^{2}|\cpt(\vec \bx_m^\be,\vec E,\vec R)]<\infty$ for any $\bx$ and $N$.  Let $I\subset \R$ be an open set and $\bar I$ its closure.  Then,  for $I_m$ and $I'_m$ of \eqref{eq:Im} and \eqref{eq:Imprime},  
	\begin{align*}
		&\varlimsup_{N\to\infty}\frac1N\log\Big(\sup_{\bar V(\epsilon)} \E\Big[ \sum_{\vec\bx\in \Cs_{m,m+1}(\epsilon,\vec\bx_m^\be)}
		\indic\big\{w_N(\bx_{m+1})\in I \big\} \,\Big|\,\cpt(\vec \bx_m^\be,\vec E,\vec R) \Big]\Big)\\
		&\leq    \sup_{(E,R) \in I_m(\epsilon)\times I'_m(\epsilon)}\Theta_{\xim}(E,R)  +
		\varlimsup_{N\to\infty}\frac1N\log\Big(\sup_{\bar V(\epsilon)}\P\Big[ w_N(\bx_{m+1}^{\be})\in \bar I  \,\Big|\,\cpt(\vec \bx_{m+1}^\be,\vec E,\vec R) \Big]\Big).
	\end{align*}
\end{lem}

\begin{proof}
	
	In the pure case of $\xi(t)=ct^{p}$ with $m=0$ the random variable $\partial_{\bx}H_{N}(\bx)=\partial_{\bx-\bx_0^{\be}}H_{N}(\bx)$
	is measurable \abbr{wrt} $H_{N}(\bx)$, while in the mixed case (that
	is, any other mixture $\xi(t)$),  or if $m\geq1$ the law of $(H_{N}(\bx),\partial_{\bx-\bx_m^{\be}}H_{N}(\bx))$
	is non-degenerate for $\bx\perp\bx_m^{\be}$. We assume hereafter either that the mixture $\xi(t)$
	is not pure or that $m\geq1$. In the case that $\xi(t)$ is pure and $m=0$ some modifications to the proof are needed. The main one being that the terms involving $\partial_{\bz} H^{(m)}_{N}(\bz)$ should be removed from the definitions of $U_N(\bz)$, $V_N(\bz)$ and $\mathcal{H}_N(\bz)$ below and,  since $\partial_{\bz} H^{(m)}_{N}(\bz)/\|\bz\|^2=\alpha_{p,q_1} H^{(m)}_{N}(\bz)/(N-m)$ for $\alpha_{p,q}=\sqrt{q}p=q\Rs(q)/\Es(q)$, the set $I_1\times I_2$ should be replaced by $\bar I_1:=\{E\in I_1, \alpha_{p,q_1}E\in I_2\}$.  Noting that the degeneracy only simplifies the proof and that $\cpt(\vec \bx_m^\be,\vec E,\vec R)=\Omega_0$ for $m=0$ where $\Omega_0$ is the whole sample space,  other modifications in this case are straightforward and left to the reader.
	
	Throughout the proof,  we denote expectation conditional on $\cpt(\vec \bx_m^\be,\vec E,\vec R)$ 
	for fixed values of $(E_i,R_i)_{i\leq m}$  by $\E_0$.
	 Recalling
	 the definition \eqref{eq:Hm} of $H^{(m)}_N(\bz)$,  we have from Corollary \ref{cor:Hm1}
	 that
\begin{align}
	\E_0\Big[ \sum_{\vec\bx\in \Cs_{m,m+1}(\epsilon,\vec\bx_m^\be)} \!\!\!\!\!\!\!\!
	\indic\big\{w_N(\bx_{m+1}) \in I\big\}  \Big]=\E_0\Big[ \sum_{\substack{\bz\in \mathbb{S}^{N-m-1}:\\ \grad H_{N}^{(m)}(\bz)=0}} \!\!\!\!
	\indic\big\{ w_N(T^{(m)}(\bz))\in I,\,  U_N(\bz)\in I_1\times I_2  \big\}    \Big],\label{eq:meanw}
\end{align}
	where 
	\begin{equation}\label{eq:I1I2}
		\begin{aligned}
			U_N(\bz)&:=\Big(\frac{1}{N-m} H^{(m)}_{N}(\bz),\,\frac{\partial_{\bz} H^{(m)}_{N}(\bz)}{\|\bz\|^2}\Big),\\%\label{eq:UN}
			I_1&:=\sqrt{\frac{N}{N-m}}\big(\Es(q_{m+1})-E_m-\epsilon,\Es(q_{m+1})-E_m+\epsilon\big),\\%\label{eq:I1}
			I_2&:=\sqrt{\frac{N}{N-m}}(q_{m+1}-q_m)\big(\Rs(q_{m+1})-\epsilon,\Rs(q_{m+1})+\epsilon\big).%\label{eq:I2}
		\end{aligned}
	\end{equation}
	We thus proceed to apply the Kac-Rice formula in order
	to bound the expectation in \eqref{eq:meanw}.  To this end,  fixing $\epsilon>0$ we set
	the process $w_{N}^{\epsilon}(\bx)=w_{N}(\bx)+\epsilon Z$
	where $Z\sim \N (0,1)$ is independent of all other random variables.  We further define
	the vector-valued processes
	\begin{align*}
		V_{N}(\bz)&:=\left( H^{(m)}_{N}(\bz),\partial_{\bz} H^{(m)}_{N}(\bz),\grad H^{(m)}_{N}(\bz),\mbox{vec}(\Hess H^{(m)}_{N}(\bz)),w_{N}^{\epsilon}(T^{(m)}(\bz))\right), \\%\label{eq:gvec}
		\mathcal{H}_{N}(\bz)&:=\left( H^{(m)}_{N}(\bz),\partial_{\bz} H^{(m)}_{N}(\bz),\grad H^{(m)}_{N}(\bz),\mbox{vec}(\Hess H^{(m)}_{N}(\bz))\right),%\label{eq:gvec2}
	\end{align*}
	where $\mbox{vec}(\Hess H^{(m)}_{N}(\bz))$ means the vectorization of the on-and-above elements of $\Hess H^{(m)}_{N}(\bz)$. Note that by rotational invariance, the laws of $V_{N}(\bz)$ and $\mathcal{H}_{N}(\bz)$ do not depend on $\bz\in \mathbb{S}^{N-m-1}$.
	
	Set $\oN:=N-m-1$ and $D=2+\oN+\frac12\oN(\oN+1)$.  By Corollary \ref{cor:Hm2},  conditional on  $\cpt(\vec \bx_m^\be,\vec E,\vec R)$,   
	per fixed $\bz$,  the $D$-dimensional vector $\mathcal{H}_{N}(\bz)$ has non-degenerate Gaussian density $\varphi_{\mathcal{H}_{N}(\bz)}(\cdot)$ on $\R^{D}$. 
	Since $H_{N}$ is a Gaussian process, conditional on $\cpt(\vec \bx_m^\be,\vec E,\vec R)$ and on $\mathcal{H}_{N}(\bz)=\underline{t}=(t_{1},\ldots,t_{D})$, the process $H_{N}(\cdot)$ has the same law as
	\[
	\E_0\left[H_{N}(\cdot)\,|\,\mathcal{H}_{N}(\bz)=\underline{t}\right]+Y_{N}(\cdot)=\sum_{i=1}^{D}t_{i}\psi_{N}^{(i)}(\cdot)+Y_{N}(\cdot),
	\]
	for some deterministic functions $\psi_{N}^{(i)}\in C^\infty(\BN)$, where $Y_{N}\in C^\infty(\BN)$ is a Gaussian process whose law is that of $H_N$ given that $\mathcal{H}_{N}(\bz)=0$.  In particular, 
	conditionally on $\cpt(\vec \bx_m^\be,\vec E,\vec R)$,  the vector $V_{N}(\bz)$ has at any $(\underline{t},s)$
	the density
	$p_{\epsilon}(s|\underline{t}) \varphi_{\mathcal{H}_{N}(\bz)}(\underline{t})$,  where 
	\begin{equation*}
		\label{eq:pst}
		\begin{aligned}
			p_{\epsilon}(s|\underline{t})&:=\frac{1}{\sqrt{2\pi\epsilon^{2}}}\E_0\bigg(\exp\Big[-\frac{1}{2\epsilon^{2}}\big(s-\bar w_{N}\big(T^{(m)}(\bz),H_{N}(\cdot)\big)\big)^{2}\Big] \,\Big|\, \mathcal{H}_N(\bz)=\underline{t}\bigg)
			\\
			&=\frac{1}{\sqrt{2\pi\epsilon^{2}}}\E_0\exp\Big[-\frac{1}{2\epsilon^{2}}\Big(s-\bar w_{N}\big(T^{(m)}(\bz),\sum_{i=1}^{D}t_{i}\psi_{N}^{(i)}(\cdot)+Y_{N}(\cdot)\big)\Big)^{2}\Big].
		\end{aligned}
	\end{equation*}
	From the bounded convergence theorem and \eqref{eq:glim2},  the density $p_{\epsilon}(s|\underline{t})$
	is continuous in $(\underline{t},s)$,  and hence,  by the preceding, 
	conditionally,  $V_{N}(\bz)$ has a strictly positive,  continuous density.
	  This
    implies in turn,  by the same reasoning as in \cite[Proposition 5.1]{DemboSubag2020},   that
	the \abbr{rhs} of \eqref{eq:meanw} is bounded above by $K(\Omega_0)$ where $\Omega_0$ 
	is the whole sample space and  
\begin{align*}
	K(\cdot):=C_{N,m}
	\E_0\bigg[ \left|\det\left(\Hess  H^{(m)}_{N}(\bz)\right)\right|\indic\{\cdot\}\indic\big\{  w_N(T^{(m)}(\bz))\in \bar I,\, U_N(\bz)\in I_1\times I_2  \big\} \, \bigg| \, \grad H^{(m)}_{N}(\bz)=0 \bigg],
\end{align*} 	
%	\begin{equation}\label{eq:KR}
%		\begin{aligned}
%			&C_{N,m}
%			\E_0\bigg[ \left|\det\left(\Hess  H^{(m)}_{N}(\bz)\right)\right|\indic\big\{  w_N(T^{(m)}(\bz))\in \bar I,\, U_N(\bz)\in I_1\times I_2  \big\} \, \bigg| \, \grad H^{(m)}_{N}(\bz)=0 \bigg]			,
%		\end{aligned}
%	\end{equation}		
%	where 
with
$\bz\in\mathbb{S}^\oN$
% is 
an arbitrary point, 
	\begin{equation}\label{eq:CNm}
	C_{N,m}:=(\sqrt{ N-m})^{N-m-1}
	\frac{2\pi^{(N-m)/2}}{\Gamma\left((N-m)/2\right)} \widetilde{\varphi}_{N,m} (0)
	\end{equation}
	and $\widetilde{\varphi}_{N,m}(\cdot)$ denotes the Gaussian density of $\grad H^{(m)}_{N}(\bz)$
	conditional on $\cpt(\vec \bx_m^\be,\vec E,\vec R)$,  which by Corollary \ref{cor:Hm2}
	is merely the density of a $\N(0,\xim'(1)\mathbf{I}_{\oN})$ vector.  
	To bound $K(\Omega_0)=K(\calE)+K(\calE^c)$ we will bound each of the two summands 
	for $\calE$ to be determined below.
	
	Now,  by Corollary \ref{cor:Hm2} the density of $U_N(\bz)$ conditional on 
	$\cpt(\vec \bx_{m}^\be,\vec E,\vec R)$ and $\grad H^{(m)}_{N}(\bz)=0$,  is merely the density 
	$\varphi_{N,m}$ of $\N(0,\frac{1}{N-m}\Sigma_{\xim})$,  while upon conditioning also on  
% 	$\cpt(\vec \bx_{m}^\be,\vec E,\vec R)$,  $\grad H^{(m)}_{N}(\bz)=0$ and
	 $U_N(\bz)=(E,R)$,  the matrix $\Hess  H^{(m)}_{N}(\bz)$ has the law of $b_m \bG - R \,\mathbf{I}$
	where $b_m^2 := \xim''(1)$ and $\bG$ is
	an $\oN$-dimensional \abbr{GOE} matrix,  scaled by the factor $
	%c_\oN:=
	\sqrt{\frac{\oN}{\oN +1}}$.  To bound $K(\Omega_0)$,  we set
\begin{equation}\label{eq:amirLD2}
Z_{\oN,R}:=\frac{1}{\oN}\log \big|\det(\bG-(R/b_m)\mathbf{I})\big| = 
% \log b_m +
 \int \log |\lambda - R/b_m| d\mu_{\oN}(\lambda) \,,
\end{equation}
where $\mu_{\oN}$ is the empirical measure of eigenvalues of $\bG$. Considering for $t>0$ the event
$\calE:=\big\{ Z_{\oN,R} \le \Omega(R/b_m) + t 
%\det\big( \Hess  H^{(m)}_{N}(\bz) \big)\big| \le \Gamma_{N,R}(t) 
\big\}$ and decomposing according to the value of $U_N(\bz)$,  we see that
%thus bound the \abbr{rhs} of \eqref{eq:meanw} by
	\[
	K(\calE) \le 
	C_{N,m} \int_{I_1\times I_2} \Gamma_{\oN,R}(t) Q_{E,R} (\bar I) \varphi_{N,m}(E,R)dE dR \,,
	\]	
	where for $\Omega(\cdot)$ of \eqref{eq:Omega},
	\begin{align}\label{eq:Gamma}
	\Gamma_{\oN,R}(t) &: = b_m^{\oN} \exp\big( \oN(\Omega(R/b_m)+t)\big)
\qquad \text{and} \\
		Q_{E,R} (\bar I)&:=
\mathbb{P}_0\Big[ w_N(T^{(m)}(\bz))\in \bar I  \, \Big| \, \grad H^{(m)}_{N}(\bz)=0,\, U_N(\bz)=(E,R) \Big]\,.\nonumber
\end{align}
Similarly,  
\begin{equation*}
%\label{eq:amirLD1}
K(\calE^c) \le  C_{N,m} \int_{I_1 \times I_2}  b_m^{\oN} 
 \E \Big[ 
e^{\oN Z_{\oN,R}}\cdot  \indic\{Z_{\oN,R} > \Omega(R/b_m)+ t \}
%\big| \det(b_m \bG - R\mathbf{I}) \big| ;  
% |\det(b_m \bG - R\mathbf{I})| > \Gamma_{N,R}(t) 
\Big] \varphi_{N,m}(E,R)dE dR \,,
\end{equation*}
from which we deduce by Cauchy-Schwarz,  that 
$K(\calE^c) \le  C_{N,m} A(t) |I_1||I_2|$,  for
\begin{align*}
A (t)^2 := \sup_{R\in I_2} \Big\{ b_m^{2 \oN} \E \Big[ e^{2 \oN Z_{\oN,R}} 
% \big|\det(b_m \bG - R\mathbf{I})\big|^2  
\Big]  \Big\} 
\sup_{R\in I_2} \Big\{ \P  \big( Z_{\oN,R} > \Omega(R/b_m) + t \big) \Big\} \,.
\end{align*}
Denoting by $\lambda_i$ the eigenvalues of $\bG-(R/b_m)\mathbf{I}$ and setting $\bar\lambda:=\max_i |\lambda_i|$, we have that $Z_{\oN,R} \leq  \log \bar\lambda$.  Recall from \cite[Lemma 6.3]{BDG},  that for 
some universal constant $x_0$ and any 
$x \ge x_0$ 
\begin{equation}\label{eq:BDG}
\P\big( |\bar \lambda| > x + |R|/b_m \big) \le e^{-\oN x^2/9} \qquad \forall \oN,  R \,,
\end{equation}
hence,
\begin{equation}\label{eq:LDP-bd1}
\limsup_{N \to \infty} \sup_{R \in I_2} \frac{1}{\oN} \log \Big( b_m^{2\oN} 
\E \big[ e^{2 \oN Z_{\oN,R}}  \big] \Big) < \infty \,.
\end{equation}
Further,  fixing $t>0$,  define for $\delta>0$  the bounded,  continuous truncated functions 
	\[
	\log_\delta(t) := \log \Big( ( |t| \vee \delta)  \wedge (1/\delta) \Big)\,,
	% \begin{cases}
		%	\log \epsilon &|t|\leq \epsilon\\
		%	\log |t| 	& \epsilon<|t|\leq 1/\epsilon\\
		%	\log (1/\epsilon) &|t|>1/\epsilon
		% \end{cases}
	\]  
whereby
	\begin{align}\label{eq:LDP-bd2}
		\P  \big( Z_{\oN,R} > \Omega(R/b_m) + t \big)&\leq \P  \Big( \int \log_\delta (\lambda - R/b_m) d\mu_{\oN}(\lambda)  > \Omega(R/b_m) + t \Big)+\P\big(\bar\lambda > \frac1\delta\big) \,.
	\end{align}
	% where the second inequality holds for large $N$, uniformly in $R\in I_2$.
Note that there exists a number $\delta_o(t)>0$ depending on $t$, such that for any $\delta \le \delta_o(t)$ and all $R\in I_2$,  
\[
\frac{1}{2\pi}  \int_{-2}^{2} \log_{\delta} (\lambda-R/b_m) \sqrt{4-\lambda^2}d\lambda - \Omega(R/b_m) \le \frac{t}{2} \,,
\]
in which case,  by the large deviation principle at rate $\oN^2$ for $\mu_\oN$
% for the empirical measure of eigenvalues of $\bG$,  
(see \cite[Theorem 2.1.1]{BAG97}),  we bound the first term on the \abbr{rhs} of \eqref{eq:LDP-bd2}
by  $\exp(-C_\delta \oN^2)$ for some $C_\delta>0$,  all $\oN$ large enough and 
any $R \in I_2$.   Bounding the second term 
in \eqref{eq:LDP-bd2} by \eqref{eq:BDG},  then taking $N \to \infty$ followed by $\delta \to 0$ we conclude that for any $t>0$, 
	\begin{align*}
		\lim_{N \to \infty} \sup_{R \in I_2} \frac{1}{\oN} \log \P  \big( Z_{\oN,R} > \Omega(R/b_m) + t \big)  &= -\infty \,.
		%\label{eq:LDP-bd2}
	\end{align*}
	This,  together with \eqref{eq:LDP-bd1}
	implies in turn that $A(t_N)\leq e^{-C_N N}$ for some $t_N \to 0$ and $C_N \to \infty$.

Assume \abbr{wlog} that $\bz=(T^{(m)})^{-1}(\bx_{m+1}^\be)$ and set $(E_{m+1},R_{m+1})$ by the inverse of \eqref{eq:ERtransform},  namely
	\begin{equation}\label{eq:ERinversetransform}
		E_{m+1}:= \sqrt{\frac{N-m}{N}}E+E_m,\qquad R_{m+1}:= \sqrt{\frac{N-m}{N}}\frac{1}{q_{m+1}-q_m}R,
	\end{equation}
	to get that
	\[
	Q_{E,R}(\bar I) =
\mathbb{P}\left(  w_N(\bx_{m+1}^{\be})\in \bar I \, \Big| \, \cpt(\vec \bx_{m+1}^\be,\vec E,\vec R) \right)\,.
\]	
Consequently,  subject to \eqref{eq:ERinversetransform},  we have that for some $t_N \to 0$ and $C_N \to \infty$,  
\begin{align}\label{eq:A-new}
& \E\Big[ \sum_{\vec\bx\in \Cs_{m,m+1}(\epsilon,\vec\bx_m^\be)}
		\indic\big\{w_N(\bx_{m+1})\in I \big\} \,\Big|\,\cpt(\vec \bx_m^\be,\vec E,\vec R) \Big] \nonumber \\
		& \le C_{N,m} \Big[ e^{-C_N N} +  \int_{I_1\times I_2} \Gamma_{\oN,R}(t_N)
		\mathbb{P}\left(  w_N(\bx_{m+1}^{\be})\in \bar I \, \Big| \, \cpt(\vec \bx_{m+1}^\be,\vec E,\vec R) \right)
		 \varphi_{N,m}(E,R)dEdR  \Big]\,.
\end{align}
%	\[
%	\E_0\Big[ \cdots  \, \Big| \, \grad H^{(m)}_{N}(\bz)=0,\, U_N(\bz)=(E,R) \Big] = \E\Big[ \cdots  \, \Big| \,  \cpt(\vec \bx_{m+1}^\be,\vec E,\vec R) \Big].
%	\]
By definition,  the mapping \eqref{eq:ERinversetransform} is a bijection between $(E,R) \in I_1\times I_2$ and
\[
\big\{(E_{m+1},R_{m+1}):\, |E_{m+1}-\Es(q_{m+1})|<\epsilon,\,|R_{m+1}-\Rs(q_{m+1})|<\epsilon\big\}\,.
\] 
 Thus,  if $(E_i,R_i)_{i \le m} \in \bar V(\epsilon)$ and $(E,R) \in I_1 \times I_2$,  then 
also $(E_i,R_i)_{i \le m+1} \in \bar V(\epsilon)$ (more precisely, in both cases, $(E_i,R_i)_i$ can be completed to a length $k$ vector in $\bar V(\epsilon)$).  Hence,
	\begin{align}%\label{eq:finalKRbound}
			& \varlimsup_{N\to\infty} \frac1N\log\Big(\sup_{\bar V(\epsilon)} \E\big[ \sum_{\vec\bx\in \Cs_{m,m+1}(\epsilon,\vec\bx_m^\be)}
			\indic\big\{w_N(\bx_{m+1})\in I \big\} \,\Big|\,\cpt(\vec \bx_m^\be,\vec E,\vec R) \big]\Big) \nonumber \\
			&\quad\leq   \varlimsup_{N\to\infty}  \sup_{E_m}\sup_{I_1\times I_2} \frac1N\log\Big( C_{N,m}
			\Gamma_{\oN,R}(0) \varphi_{N,m}(E,R)
			\Big)  \label{eq:supEm} \\
			&\quad+
			\varlimsup_{N\to\infty}\frac1N\log\Big(\sup_{\bar V(\epsilon)}\P\Big[ w_N(\bx_{m+1}^{\be})\in \bar I  \,\Big|\,\cpt(\vec \bx_{m+1}^\be,\vec E,\vec R) \Big]\Big)=:A+B\,,\nonumber
	\end{align}
	where the supremum in \eqref{eq:supEm} is over $E_m$ such that $|E_m - \Es(q_m)|<\epsilon$. Noting that $B$ is the same term as in the bound stated in the lemma, it remains to show that 
	\begin{equation}\label{eq:Abound}
		A \le  \sup_{(E,R) \in I_m(\epsilon)\times I'_m(\epsilon)}\Theta_{\xim}(E,R)\,.
	\end{equation}
	
	The function whose supremum is taken in \eqref{eq:supEm} only depends on $(E,R)$. The dependence on $E_m$ is through the set $I_1$  defined in \eqref{eq:I1I2}. Instead of the double supremum, we can therefore take in \eqref{eq:supEm} a supremum over $(E,R)$ in 
	\[
	J_N(\epsilon):=\big\{ (E,R):\, \exists E_m \mbox{\ \ s.t.\ \ }|E_m - \Es(q_m)|<\epsilon,\,(E,R)\in I_1\times I_2\big\}\,.
	\]
	For $I_m(\epsilon)$ and $I'_m(\epsilon)$ of \eqref{eq:Im}, the decreasing sequence $J_{N+1}(\epsilon)\subset J_N(\epsilon)$ converges to 
	\[
	\lim_{N \to \infty} J_N(\epsilon):=\cap_{N>m} J_N(\epsilon) = I_m(\epsilon)\times I'_m(\epsilon)\,.
	\]

%	Recalling that, by Corollary \ref{cor:Hm2}, $\varphi_{\grad H^{(m)}_{N}(\bz)}$ is the density of a random vector with law $\N(0,\xim'(1)\mathbf{I}_{N-m-1})$ and $\varphi_{N,m}(E,R)$ is the density a random vector with law 
%	$\N(0,\frac{1}{N-m}\Sigma_{\xim})$, 
	Finally,  for $C_{N,m}$,  $\Gamma_{\oN,R}(0)$,  $\Sigma_{\xim}$ of \eqref{eq:CNm}, \eqref{eq:Gamma} and \eqref{eq:Sigmaxi},  respectively,  with
	\[
	\widetilde{\varphi}_{N,m}(0) = (2\pi \xi_{(m)}'(1))^{-\frac{N-m-1}{2}},\quad \varphi_{N,m}(E,R)=\frac{N-m}{2\pi\sqrt{\det(\Sigma_{\xim})}}\exp\Big(-\frac{N-m}2{(E,R)\Sigma_{\xim}^{-1}(E,R)^T}\Big),
	\] 
we confirm via Stirling's approximation (for $\Gamma((N-m)/2)$ in \eqref{eq:CNm}), 
 that uniformly on compacts,  
	\[
	 \lim_{N\to\infty} \frac1N \log\Big( C_{N,m} \Gamma_{\oN,R}(0) \varphi_{N,m}(E,R) \Big) 
	= \Theta_{\xim}(E,R) \,,
	 \]
	 and thereby conclude the proof of \eqref{eq:Abound}, and thus the lemma. 
\end{proof}

Building on Lemma \ref{lem:1lvlInductionIndic},  we further  control the corresponding conditional 
first moment of $w_N(\cdot)$.
\begin{cor}\label{cor:1lvlKR_E}
	Assume in the setting of Lemma \ref{lem:1lvlInductionIndic},  that in addition,  $w_N(\bx) \in [0, e^{NM}]$ for  finite $M$,  any $\bx$ and $N$. Then,
	\begin{align}\label{eq:Ew}
		& \varlimsup_{N\to\infty}\frac1N\log\Big(\sup_{\bar V(\epsilon)} \E\Big[ \sum_{\vec\bx\in \Cs_{m,m+1}(\epsilon,\vec\bx_m^\be)}
		w_N(\bx_{m+1}) \,\Big|\,\cpt(\vec \bx_m^\be,\vec E,\vec R) \Big]\Big) \nonumber \\
		&\leq    \sup_{I_m(\epsilon)\times I'_m(\epsilon)}\Theta_{\xim}(E,R)  +
		\varlimsup_{N\to\infty}\frac1N\log\Big(\sup_{\bar V(\epsilon)}\E \Big[ w_N(\bx_{m+1}^{\be}) \,\Big|\,\cpt(\vec \bx_{m+1}^\be,\vec E,\vec R) \Big]\Big). 
	\end{align}
\end{cor}
\begin{proof} We fix the values of $(E_i,R_i)_{i\leq m}$ and cover $[0,e^{MN}]$ by the union of 
open intervals $B_1,\ldots,B_\kappa$ (allowing $B_i$ and $\kappa$ to depend on $N$),  with 
$b_i:=\sup B_i$,  such that    
\[
w  \le \sum_{i=1}^\kappa b_i \indic_{B_i} (w) \le \sum_{i=1}^\kappa b_i \indic_{\bar B_i}(w) \stackrel{(\star)}{\le}
5 w + e^{-N^2}  \quad  and \quad \sum_{i=1}^\kappa b_i
\stackrel{(\star\star)}{\le}  6 e^{M N}  \qquad \forall w \in [0,e^{MN}] \,.  
\]
Indeed,  the only non-trivial inequalities here are $(\star)$ and $(\star\star)$,  both of which are
satisfied for example by $B_1=(-e^{-N^2},e^{-N^2})$ and $B_i = (2^{i-2.5} e^{-N^2},2^{i-1} e^{-N^2})$,  $i \ge 2$.
Utilizing the preceding bounds,  we get from \eqref{eq:A-new} for $I=B_i$,  $1 \le i \le \kappa$,  that for some $t_N \to 0$ and $C_N \to \infty$,  under the relation \eqref{eq:ERinversetransform},  
\begin{align}\label{eq:KRunifrombd}
\E\Big[  \sum_{\vec\bx\in \Cs_{m,m+1}(\epsilon,\vec\bx_m^\be)} & \!\!\!\!
		w_N(\bx_{m+1}) \,\Big|\,\cpt(\vec \bx_m^\be,\vec E,\vec R) \Big] \le C_{N,m} \bigg\{  6 e^{M N} e^{-C_N N} 
		\nonumber \\
		& +  \int_{I_1\times I_2} \Gamma_{\bar N,R}(t_N) 
		\Big( e^{-N^2} + 5 \E \big[  w_N(\bx_{m+1}^{\be}) \, \big| \, \cpt(\vec \bx_{m+1}^\be,\vec E,\vec R) \big] \Big)
		 \varphi_{N,m}(E,R)dEdR  \bigg\}\,,
\end{align}
and \eqref{eq:Ew} follows by the same argument as after \eqref{eq:A-new}.
For later use,  we note that the bound \eqref{eq:KRunifrombd} holds uniformly for any $w_N$ as in the corollary 
(since $C_N$ depends only on $A(t_N)$ whose definition does not involve $w_N$).
\end{proof}

\subsection{Proof of Proposition \ref{prop:multilvlKR}}\label{sec:multiLvlKR}
To prove Proposition \ref{prop:multilvlKR},  we set for $m\leq n$ and $\vec\by\in\AA_m$,
\begin{equation}\label{eq:alpha^eps}
	\alpha_{m,n}^{\epsilon}(\vec\by):=\sum_{\vec\bx\in \mathscr{C}_{m,n}(\epsilon,\vec\by)} \alpha_n(\vec\bx)\,,
\end{equation}
where $\mathscr{C}_{m,n}(\epsilon,\vec\by)$ is defined in \eqref{eq:xi_condition_mn}.
With $\vec\bx_0^{\be}=0$,  $\mathscr{C}_{0,n}(\epsilon,0)=\mathscr{C}_{n}(\epsilon)$
and $\cpt(\vec \bx_{0}^\be,\vec E,\vec R)$ the entire sample space,  the 
\abbr{lhs} of \eqref{eq:multilvlKR} is merely
	\begin{align*}
	& \varlimsup_{N\to\infty}\frac1N \log \E\Big[ \alpha_{0,n}^{\epsilon}(\vec\bx_0^{\be}) \wedge e^{MN} 
\,\big|\,	\cpt(\vec \bx_{0}^\be,\vec E,\vec R) \Big] \,.
	\end{align*}
Further,  as $\alpha_{n,n}^{\epsilon}(\cdot)=\alpha_n(\cdot)$,
the proposition follows upon showing that,  for any $0\leq m\leq n-1$,
\begin{equation}\label{eq:KRalphaeps}
\begin{aligned}
	&\varlimsup_{N\to\infty}\frac1N\log\Big(\sup_{\bar V(\epsilon)} \E\Big[ \alpha_{m,n}^{\epsilon}(\vec\bx_m^{\be}) \wedge e^{MN} \,\big|\,\cpt(\vec \bx_{m}^\be,\vec E,\vec R) \Big]\Big)\\
	&\leq    \sup_{I_m(\epsilon)\times I'_m(\epsilon)}\Theta_{\xim}(E,R)  +
	\varlimsup_{N\to\infty}\frac1N\log\Big(\sup_{\bar V(\epsilon)}\E\Big[ \alpha_{m+1,n}^{\epsilon}(\vec\bx_{m+1}^{\be}) \wedge e^{MN}  \,\big|\,\cpt(\vec \bx_{m+1}^\be,\vec E,\vec R) \Big]\Big).
\end{aligned}
\end{equation}	
To this end,  we plan to apply the one-level Kac-Rice formula of Corollary \ref{cor:1lvlKR_E}.  However,  as we move $\vec\by$, new points may `enter' $\mathscr{C}_{m,n}(\epsilon,\vec\by)$ while others may `disappear' from it.  With 
$\alpha_{m,n}^{\epsilon}(\vec \by)$ discontinuous at those $\vec\by$ where this happens,  some 
preliminary work is needed to approximate $\alpha_{m,n}^{\epsilon}(\cdot)$,  $m \le n$ by continuous functions
of the form 
\[
\alpha_{m,n}^{\epsilon,\delta}(\vec\by):=\sum_{\vec\bx\in \mathscr{C}_{m,n}(\epsilon,\vec\by)}
\Psi^{\epsilon,\delta}_{m,n}(\vec\bx)
\alpha_n(\vec\bx)\,.
\]
Specifically,  to define $\gradt H_{N}(\bx_i):=M_{i}\nabla H_{N}(\bx_i)$ we used an arbitrary matrix $M_i=M_i(\vec\bx)\in\R^{N-i\times N}$
whose rows form an orthonormal basis of $\big(\sp\{\bx_1,\ldots,\bx_i\}\big)^{\perp}$. Using the same matrix, for any $i\leq m\leq k$ and $\vec\bx=(\bx_1,\ldots,\bx_m)$, define 
$\Hesst H_{N}(\bx_i):=M_{i}\nabla^2 H_{N}(\bx_i)M_i^{\top}$ and 
\begin{equation}\label{eq:etai}
	\eta_i(\vec\bx):=\min_j \bigg|\lambda_j(\Hesst H_{N}(\bx_i))-\frac{\partial_{\bx_i-\bx_{i-1}}H_{N}(\bx_i)}{\|\bx_i-\bx_{i-1}\|^2}\bigg|,
\end{equation} 
where $\lambda_j(A)$ are the eigenvalues of a matrix $A$.  For any $\frac{\epsilon}{2}>\delta>0$,  let 
$\bar h_\delta(t)=1 \wedge (\frac{t}{\delta}-1)_+$ and set
\begin{align*}
	\Psi^{\epsilon,\delta}_{m,n}(\vec\bx):=\prod_{i=m+1}^n&\bar h_\delta\Big(\epsilon -\big | \frac1NH_N(\bx_i)-\Es(q_i)\big|\Big)
	\bar h_\delta\Big(\epsilon - \big| \frac{\partial_{\bx_i-\bx_{i-1}}H_{N}(\bx_i)}{\|\bx_i-\bx_{i-1}\|^2}-\Rs(q_i)\big| \Big)
	\bar h_\delta\big(\eta_i(\vec\bx)\big).
\end{align*}
Note that $\bar h_\delta(t) \nearrow \indic_{(0,\infty)}(t)$ as $\delta \downarrow 0$,  and therefore 
$\Psi^{\epsilon,\delta}_{m,n}(\vec\bx) \nearrow 1$ for any 
$\vec \bx \in \mathscr{C}_{m,n}(\epsilon,\vec\bx_m^{\be})$ with $\min_{m<j \le n} \{ \eta_j(\vec\bx)\} > 0$.
Recall that any such $\vec \bx$ is in $\AA_{m,n}(\vec\bx_m^{\be})$ and must satisfy the condition as in \eqref{eq:xi_condition_mn} 
% $\mathscr{C}_{m,n}(\epsilon,\vec\by)$ requires 
at any $m< i\leq n$.  In particular,  for such $\vec\bx$ we have that 
$\gradt H_{N}(\bx_i)=0$ for all $i \in (m,n]$.  The same applies if 
$\vec\bx\in \mathscr{C}_{m+1,n}(\epsilon,\vec\by)$ for some 
$\vec\by\in \mathscr{C}_{m,m+1}(\epsilon,\vec\bx_m^{\be})$,  so from Lemma \ref{lem:etapositive_conditional} 
we deduce 
that conditionally on $\cpt(\vec\bx_m^{\be},\vec E,\vec R)$,  almost surely,  as $\delta \downarrow 0$, 
$$
\sum_{\substack{\vec\bx\in \mathscr{C}_{m+1,n}(\epsilon,\vec\by)\\ \vec\by\in \mathscr{C}_{m,m+1}(\epsilon,\vec\bx_m^{\be})}}
\Psi^{\epsilon,\delta}_{m+1,n}(\vec\bx)
\alpha_n(\vec\bx)\,
\nearrow 
\sum_{\substack{\vec\bx\in \mathscr{C}_{m+1,n}(\epsilon,\vec\by)\\ \vec\by\in \mathscr{C}_{m,m+1}(\epsilon,\vec\bx_m^{\be})}} \alpha_n(\vec\bx) \,,
$$
or equivalently,  that 
\begin{equation}\label{eq:alpha-del}
\sum_{\vec\by\in \mathscr{C}_{m,m+1}(\epsilon,\vec\bx_m^{\be})} \alpha_{m+1,n}^{\epsilon,\delta}(\vec\by) 
\nearrow 
\sum_{\vec\by \in \mathscr{C}_{m,m+1}(\epsilon,\vec\bx_m^{\be})}
\alpha_{m+1,n}^{\epsilon}(\vec\by) \,.
\end{equation}

\begin{lem}\label{lem:etapositive_conditional}
For any $0 \le m < n \leq N$ and $\vec E, \vec R$,  conditionally on $\cpt(\vec\bx_m^{\be},\vec E,\vec R)$,  almost surely,
\begin{equation*}\label{eq:eta0}
\Big\{ \vec\bx\in \AA_{m,n}(\vec\bx_m^{\be}):\, \forall m< i\leq n,\, \gradt H_{N}(\bx_i)=0,\, \min_{m < j\leq n}\eta_j(\vec\bx)= 0  \Big\} = \varnothing.
\end{equation*}
\end{lem}
\begin{proof} Fixing $N \ge n > m \ge 0$ and $\vec E,\vec R$,  it suffices
%  for \eqref{eq:eta0},  
to show that conditionally on $\cpt(\vec\bx_m^{\be},\vec E,\vec R)$,  a.s.,
\[
\Big\{ \vec\bx\in \AA_{m,n}(\vec\bx_m^{\be}):\, \forall m < i\leq n,\, \gradt H_{N}(\bx_i)=0,\, \eta_{n}(\vec\bx)= 0  \Big\} = \varnothing.
\]
We start with the unconditional version of $m=0$,  whereby setting 
\[
\Cs_n := \Big\{ \vec\bx\in \AA_{n}:\, \forall i\leq n,\, \gradt H_{N}(\bx_i)=0  \Big\},
\]
we proceed to show that $\{\vec\bx\in\Cs_n:\,\eta_n(\vec\bx)=0\}=\varnothing$ a.s. 
	
Indeed,  note that $\AA_n\subset \R^{N\times n}$ is a compact manifold of dimension 
$d_0 := \sum_{i=1}^n (N-i)$ (which is diffeomorphic to the Stiefel manifold $V_n(\R^N)$ of orthonormal $n$-frames 
in $\R^N$) and that the Euclidean structure on $\R^{N\times n}$ induces 
a  Riemannian  metric on $\AA_n$,  whose distance function and volume form we denote by 
$d(\cdot,\cdot)$  and $dv$.  Let $\mbox{Vol}(\AA_n)$ 
	and $\mbox{Vol}(\epsilon)$ denote the volume \abbr{WRT} $dv$ of $\AA_n$ and of  a ball of radius 
	$\epsilon$ in $\AA_n$ around $\vec\bx\in\AA_n$,  which by symmetry does not depend on $\vec\bx$.

Note that the mappings $\vec\bx\mapsto\|\gradt H_{N}(\bx_i)\|$,  $m< i \le n$ and $\vec\bx\mapsto \eta_n(\vec\bx)$ 
are Lipschitz on $\AA_n$ with some random Lipschitz constant $L_n=L_{N,n}$ which depends on the smooth function $H_N(\cdot)$ (but is independent of our choice of $M_i(\vec\bx)$). 
Thus,  denoting by $\mathcal{E}(c)$ the event that $L_n <c$ and $\{\vec\bx\in\Cs_n:\,\eta_n(\vec\bx)=0\}\neq \varnothing$,  to finish the proof,  it suffices to show that $\P\{\mathcal{E}(c)\}=0$ for any fixed,  finite $c$.  

	For any $\vec\bx'\in\AA_n$, define the event
	\[
	\DD_{\vec\bx'} := \{  \forall m< i\leq n,\, \|\gradt H_{N}(\bx'_i)\|<c\epsilon,\, \eta_n(\vec\bx')<c\epsilon \} \,.
	\]
	Note that by rotational invariance,
	$\P(\DD_{\vec\bx'})$ is the same for any fixed,  non-random $\vec\bx'\in\AA_n$. In particular, $\P(\DD_{\vec\bx'})=\P(\DD_{\vec\bx^{\be}_n})$ for such $\vec\bx'\in\AA_n$.
	Now,  if $\mathcal{E}(c)$ occurs then: first,  there exists some $\vec\bx\in\AA_n$ such that $\gradt H_{N}(\bx_i)=0$ for $m<i\leq n$ and $\eta_n(\vec\bx)=0$; second, for any $\vec\bx'\in\AA_n$ within distance $\epsilon$ from $\vec\bx$ the event
	$\DD_{\vec\bx'}$ holds .  
	It follows that the expected total volume 
		of the subset of $\vec\bx' \in \AA_n$ for which $\DD_{\vec\bx'}$ holds,  is at least
		$\P\{\mathcal{E}(c)\} \mbox{Vol}(\epsilon)$. On the other hand, by Fubini's theorem, that expected total volume is $\P(\DD_{\vec\bx^{\be}_n})\mbox{Vol}(\AA_n)$. Hence,
	\begin{equation}\label{eq:PEc}
		\P\{\mathcal{E}(c)\} \mbox{Vol}(\epsilon) \leq \P(\DD_{\vec\bx^{\be}_n})\mbox{Vol}(\AA_n) \,.
	\end{equation}	
	%Since the event $\DD_{\vec\bx'}$ is rotationally invariant,  by symmetry 
	
	Invoking  rotational invariance once more, we note that 
	$\P(\DD_{\vec\bx^{\be}_n})$ 
	does not depend on the specific choice of  $M_{i}=M_{i}(\vec\bx^{\be}_n)$,  as long as its
	rows form an orthonormal basis of $\big(\sp\{\bx^{\be}_1,\ldots,\bx^{\be}_{i}\}\big)^{\perp}$. 
	Suppose that the $j$-th row of $M_i$ is the standard basis element $\be_{j+i}\in\R^N$.
	Let $m\leq n-1$ and take $\bz=(T^{(m)})^{-1}(\bx^{\be}_{m+1})=\sqrt{N-m}\,\be_1\in\R^{N-m}$, where we use the same notation as in Corollary \ref{cor:Hm1} (and by abuse of notation, denote by $\be_i$ the basis elements in both $\R^N$ and $\R^{N-m}$). By Corollary \ref{cor:Hm1}, 
	\[
	\frac{\partial_{\bx^{\be}_{m+1}-\bx^{\be}_{m}}H_{N}(\bx^{\be}_{m+1})}{\|\bx^{\be}_{m+1}-\bx^{\be}_{m}\|^2} = \sqrt{\frac{N-m}{N}}\frac{1}{q_{m+1}-q_{m}} \frac{\partial_{z_1} H^{(m)}_{N}(\bz)}{\|\bz\|}.
	\]
	By working with the definitions, one can also check that 	 
	\begin{align*}
		(\gradt H_N(\bx^{\be}_{m+1}))_{i} &= \frac{1}{\sqrt{q_{m+1}-q_{m}}} \partial_{z_{i+1}} H^{(m)}_{N}(\bz),\\
	(\Hesst H_N(\bx^{\be}_{m+1}))_{ij} &= \sqrt{\frac{N-m}{N}}\frac{1}{q_{m+1}-q_{m}} \partial_{z_{i+1}}\partial_{z_{j+1}} H^{(m)}_{N}(\bz)  .
	\end{align*}
	Finally, in the notation of Corollary \ref{cor:Hm2}, for an appropriate choice of $F_i$, (see e.g.\ \cite[Eq.~(7.11)]{geometryMixed})
	\begin{align*}
		\grad H^{(m)}_{N}(\bz) &= \big(\partial_{z_{i+1}} H^{(m)}_{N}(\bz)   \big)_{i=1}^{N-m-1},
		\\
	\Hess H^{(m)}_{N}(\bz) &= \big(\partial_{z_{i+1}}\partial_{z_{j+1}} H^{(m)}_{N}(\bz)   \big)_{i,j=1}^{N-m-1}-\frac{\partial_{z_1} H^{(m)}_{N}(\bz)}{\|\bz\|}\mathbf{I}.
	\end{align*}

By Corollary \ref{cor:Hm2},  given $\cpt(\vec\bx_m^{\be},\vec E,\vec R)$ which includes $\gradt H_N(\bx_i^{\be})$, $i\leq m$, the vector $\gradt H_N(\bx_{m+1}^{\be})$ has positive density on $\R^{N-(m+1)}$.  Proceeding 
inductively over $1 \le m \le n$,  we find that unconditionally,  the Gaussian 
$X:=\big(\gradt H_N(\bx_i^{\be})\big)_{i\leq n}$ has positive density on $\R^{d_0}$.  
From Corollary \ref{cor:Hm2} we further have that 
conditionally on $\cpt(\vec \bx_n^{\be},\vec E,\vec R)$ which includes $X$,  the variable $\eta_n(\vec\bx_n^{\be})$
is the minimal eigenvalue (in absolute value),  of a scaled $(N-n-1)$-dimensional \abbr{GOE} matrix,  shifted
by an independent random multiple of the identity.  As such,  this variable clearly has a positive density.  Thus,  unconditionally,  $(X,\eta_n(\vec\bx_n^{\be}))$ has a positive 
joint density on $\R^{d_0+1}$ and consequently,  the probability on the \abbr{rhs} of
\eqref{eq:PEc} is of order $O(\epsilon^{d_0+1})$.  Of course,  $\mbox{Vol}(\epsilon)=\Theta(\epsilon^{d_0})$,  
so upon dividing both sides by $\mbox{Vol}(\epsilon)$ and
taking $\epsilon\to0$,  we conclude that $\P\{\mathcal{E}(c)\}=0$.

For $m \ge 1$ we run a similar argument,  for the conditional law 
$\P_0$ on the compact manifold $\AA_{m,n}(\vec\bx_m^{\be})$ of dimension 
$d_m:=\sum_{i>m}^n (N-i)$,  given $\cpt(\vec\bx^\be_m,\vec E,\vec R)$,  now considering
$$
\Cs_{m,n} := \big\{ \vec\bx\in \AA_{m,n}(\vec\bx_m^{\be}):\, \forall m< i\leq n,\, \gradt H_{N}(\bx_i)=0\big\} \,,
$$ 
with $\mathcal{E}(c)$ the event that $L_n <c$ and $\{\vec\bx\in\Cs_{m,n}:\,\eta_n(\vec\bx)=0\}\neq \varnothing$ 
(where now $L_n$ is measured
%of $\vec\bx\mapsto\|\gradt H_{N}(\bx_i)\|$,  $m<i \le n$ and $\vec\bx\mapsto \eta_n(\vec\bx)$,  
\abbr{wrt} the distance on $\AA_{m,n}(\vec\bx_m^{\be})$).  Having under $\P_0$ the relevant symmetries for $\DD_{\vec\bx'}$ within 
$\AA_{m,n}(\vec\bx_m^{\be})$,  we get similarly to the derivation of \eqref{eq:PEc},  that 
\begin{equation*}\label{eq:PEc2}
\P_0\{\mathcal{E}(c)\} \mbox{Vol}(\epsilon) \leq \P_0(\DD_{\vec\bx_n^{\be}})\mbox{Vol}(\AA_{m,n}(\vec\bx_m^{\be})) \,.
\end{equation*}
As $\mbox{Vol}(\epsilon)=\Theta(\epsilon^{d_m})$,  it thus suffices to show that for $m<n$,  the collection 
$\big(\big(\gradt H_N(\bx_i^{\be})\big)_{m < i\leq n},\eta_{n}(\vec\bx_{n}^{\be})\big)$ has  
strictly positive density under $\P_0$.  For the latter task we utilize as before
Corollary \ref{cor:Hm2} conditional on $\cpt(\vec\bx_m^{\be},\vec E,\vec R)$,  now
using the fact that $\cpt(\vec\bx_{m+1}^{\be},\vec E,\vec R) \subset \cpt(\vec\bx_{m}^{\be},\vec E,\vec R)$
in order to remain throughout the proof under the conditional law $\P_0$.
\end{proof}

\begin{lem}\label{lem:alpha_continuity}
For any $\delta \in (0,\epsilon/2)$,  the function
	$\vec\by\mapsto \alpha^{\epsilon,\delta}_{m,n}(\vec\by)$ is a.s.\  continuous on $\AA_m$. 
\end{lem}
\begin{proof}
Recall that $H_N(\cdot)$ is a.s.\ continuously differentiable up to order three over the ball of radius $\sqrt N$, e.g.\ by \cite[Corollary C.2]{geometryMixed}. Hence, it will be enough to prove the stated continuity for fixed deterministic $N$ and $H_N(\cdot)$ with the same property, as we shall.  Denote by $C_N'$ the maximal directional derivative up to order $3$ over the ball, which is finite by compactness.
	Since  rotations of $H_N(\cdot)$ preserve the above differentiability property, do not change the value of $C_N'$, and act as rotations also on the function
	$\vec\by\mapsto \alpha^{\epsilon,\delta}_{m,n}(\vec\by)$, \abbr{wlog} it is sufficient to prove continuity only at $\vec\by=\vec\by^\be_m=(\by_1^{\be},\ldots,\by_m^{\be}) \in \AA_m$ where
	$\by^{\be}_i := \sum_{j=1}^i \sqrt{N(q_j-q_{j-1})} \be_j$.

	Recall that $M_{i}=M_{i}(\vec\bx)\in\R^{(N-i)\times N}$ is a matrix
	whose rows form an orthonormal basis of $\big(\sp\{\bx_1,\ldots,\bx_{i}\}\big)^{\perp}$ and define, for $n>m$,
	\begin{equation*}\label{def:Amn-s}
		\AA_{m,n}^*:=\Big\{\vec\bx\in\AA_{n}:\, M_{i}(\vec\bx)\nabla H_N(\bx_{i}) = 0,\,\forall m+1\leq i\leq n\Big\}.
	\end{equation*}
	The main step of the proof is to show using the implicit function theorem that for any
	$m<n\leq k$
	% $\vec\by\in\AA_m$
	and $\vec\bx\in\AA^*_{m,n}(\vec\by):=\AA_{m,n}(\vec\by)\cap \AA^*_{m,n}$ 
	(for $\AA_{m,n}(\vec\by)$ of \eqref{eq:Amn-y}),  with
	$\eta_i(\vec\bx)>\delta$ for $\eta_i(\vec \bx)$ of \eqref{eq:etai} and	
	all $m<i\le n$,  there exists a function $f_{\vec\by,\vec\bx}$ from a small neighborhood of $\vec\by$ in $\AA_m$ (\abbr{wrt} the distance $\|\vec\by-\vec\by'\|^2:=\sum_{i=1}^m \|\by_i-\by'_i\|^2$)  to $\R^{(n-m)\times N}$,  such that:
	\begin{equation}\label{eq:Lip_gN}
		\begin{aligned}
			\vec \bx':= (\vec\by',f_{\vec\by,\vec\bx}(\vec\by'))& \in \AA^*_{m,n}(\vec\by')\,,  \\
			\|f_{\vec\by,\vec\bx}(\vec\by')-(\bx_{m+1},\ldots,\bx_{n})\| & \leq C \|\vec\by'-\vec\by\|,
		\end{aligned}
	\end{equation}
	for some $C$ that depends only on $\delta$, $N$, $(q_i)_{i=1}^{k}$ and $C_N'$. 
		
	Note that, in fact,  by induction 
	it suffices to establish \eqref{eq:Lip_gN} only for $n=m+1$, because we can compose such functions that increase the index by $1$ to obtain the required function with general $n>m$ (here it is important 
	that the assumed bound $\eta_i(\vec\bx)>\delta$ holds  for all $m<i\leq n$).
	Moreover,  since rotations of $H_N(\cdot)$ also operate as rotations on 
	$\eta_{m+1}(\vec\bx)$,  to prove the existence of such a  function $f_{\vec\by,\vec\bx}$ for general $\vec\bx\in\AA^*_{m,m+1}(\vec\by)$ it is sufficient to prove it under the assumption that $\vec\bx\in\AA^*_{m,m+1}(\vec\by)$ and $\vec\bx=\vec\bx^\be_{m+1}=(\bx_1^{\be},\ldots,\bx_{m+1}^{\be})$ where
	$\bx^{\be}_i := \sum_{j=1}^i \sqrt{N(q_j-q_{j-1})} \be_j$, as we shall.  Note that in this case, we may also assume that the rows $M_i(\vec\bx)$ are the standard basis vectors $\be_{i+1},\ldots,\be_{N}$.

	Further,  fixing $n=m+1$,  $m \ge 1$,  we claim that \eqref{eq:Lip_gN} then follows from the implicit function theorem if,  given $\vec\by=\vec\by_m^\be$ and $\vec\bx=\vec\bx_n^\be\in\AA^*_{m,m+1}(\vec\by)$ as above,  we 
	have on some neighborhoods $A$ and $B$ of the origin in $\R^{N-1+\cdots+N-m}$ and $\R^{N-(m+1)}$,  respectively,  deterministic functions $\psi_m:A\to\AA_m$ and  $\Gamma_m:A\times B\to \AA_{m+1}$,  
	with $\Theta_m(\ba,\bb) := H_N((\Gamma_m(\ba,\bb))_{m+1})$
	(where $(\Gamma_m(\ba,\bb))_{i}$ is the $i$-th element of  $\Gamma_m(\ba,\bb)$),  such that:
	\begin{enumerate}
		\item\label{enum:pt1} $\psi_m$ is invertible on $A$ and $\psi_m(0)=\vec\by$.
		\item\label{enum:pt2} $ \Gamma_m(\ba,\bb) \in \AA_{m,m+1}(\psi_m(\ba))$ is twice continuously differentiable on $A\times B$ and $\Gamma_m(0,0)=\vec\bx$.
		\item \label{enum:pt6} The functions  $\vec\by'\mapsto \psi_m^{-1}(\vec\by')$  and  $(\vec\by',\bb)\mapsto (\Gamma_m( \psi_m^{-1}(\vec\by'),\bb))_{m+1}$ are Lipschitz on a small neighborhood of $(\vec\by,0)$ in $\AA_m\times B$ and the function $\Gamma_m$ is Lipschitz on a small neighborhood of the origin with constants that depend only on $N$ and $(q_i)_{i=1}^{k}$. 
		%	with constants that do not depend on $\vec\by$ and $\vec\bx$.
		% \item\label{enum:pt3} 
		\item \label{enum:pt4} $\Gamma_m(\ba,\bb)\in \AA^*_{m,m+1}(\psi_m(\ba))\iff \frac{d}{db_i}\Theta_m(\ba,\bb)=0,\,\forall i\leq N-(m+1)$.
		\item \label{enum:pt5} The Hessian matrix of $\Theta_m$ in $\bb$ at the origin,  namely, $\Big(\frac{d}{db_i}\frac{d}{db_j}\Theta_m(0,0)\Big)_{i,j=1}^{N-(m+1)}$ is equal to
		\begin{equation} \label{eq:HessTheta}
			M_{m+1}(\vec\bx)\nabla^2 H_N(\bx_{m+1}) M_{m+1}(\vec\bx)^{\top} - \frac{\langle \nabla H_N(\bx_{m+1}), \bx_{m+1}-\bx_{m}\rangle}{\|\bx_{m+1}-\bx_m\|^2}\mathbf{I}.
		\end{equation}
	\end{enumerate}
	Indeed,  by Point \eqref{enum:pt2} the function $\big(\frac{d}{db_i}\Theta_m(\ba,\bb)\big)_{i\leq N-(m+1)}$ is continuously differentiable on $A\times B$. 
	Moreover,  since $\vec\by\in\AA_m$ and $\vec\bx\in\AA_{m,m+1}^*(\vec\by)$ with $\eta_{m+1}(\vec\bx)>\delta/2$,  by Points \eqref{enum:pt1}, \eqref{enum:pt2} and \eqref{enum:pt4}, $\big(\frac{d}{db_i}\Theta_m(0,0)\big)_{i\leq N-(m+1)}=0$,  whereas by point \eqref{enum:pt5},  the matrix $\big(\frac{d}{db_i}\frac{d}{db_j}\Theta_m(0,0)\big)_{i,j\leq N-(m+1)}$ is invertible.  Therefore,  the 	
	implicit function theorem defines for any $\ba\in A$ in a small neighborhood of the origin,  some $\bb(\ba)\in B$ such that the condition in Point \eqref{enum:pt4} holds at $(\ba,\bb(\ba))$.  For
	any $\vec\by'$ close to $\vec\by$ we can thus define $\ba'= \psi_m^{-1}(\vec\by')$, $\bb'=\bb(\ba')$ and		
	$\vec\bx' = \Gamma_m(\ba',\bb')\in\AA_{m,m+1}^*(\vec\by')$. 		Precisely, this defines 
	\[
	f_{\vec\by,\vec\bx}(\vec\by')= \Big(\Gamma_m\big(\psi_m^{-1}(\vec\by'),\bb(\psi_m^{-1}(\vec\by'))\big)\Big)_{m+1}\,.
	\]
	
	By Point \eqref{enum:pt5} and the assumption that $\eta_{m+1}(\vec\bx)>\delta$,  from the implicit function theorem, on a small neighborhood of $0\in A$, the gradient of $\ba\mapsto\bb(\ba)$ is bounded by some constant $C_0$ that depends on $\delta$ and $C_N'$. 
	Combining this bound on the gradient of $\bb(\ba)$ with Point \eqref{enum:pt6}  
	then yields the required bound as in \eqref{eq:Lip_gN}.
	
	We proceed to complete the proof of \eqref{eq:Lip_gN} by explicitly constructing 
	the functions $\psi_m$ and $\Gamma_m$ as above.  To this end,  note first
	that we may parameterize $\AA_m$ by $\chi_m:=\prod_{i=1}^m \mathbb{S}^{N-i}(1)$ where $\mathbb{S}^{N-i}(1):=\{\bx\in\R^{N-i+1}: \|\bx\|=1\}$. Precisely,  map any $\vec\bz\in\chi_m$ to $\Phi_m(\vec\bz)=(\Phi_{m,i}(\vec\bz))_{i=1}^m\in\AA_m$ defined by\footnote{Let $i<j$  and note that $T\circ O_{i+1}\circ T\cdots O_{j-1}\circ T(\bz_j)$ is orthogonal to $\be_{N-i+1}$ (simply by the operation of the left-most $T$). Hence, $O_{i}\circ T\circ O_{i+1}\circ T\cdots O_{j-1}\circ T(\bz_j)$ is orthogonal to $O_{i}(\be_{N-i+1})=\bz_i$ and $O_{1}\circ T\circ O_{2}\circ T\cdots O_{j-1}\circ T(\bz_j)$ is orthogonal  to
		$O_{1}\circ T\circ O_{2}\circ T\cdots O_{i-1}\circ T(\bz_i)$. From this we have that indeed $\Phi_m(\vec\bz)\in\AA_m$. Of course it is a bijection since if $\vec\bz$ and $\vec\bz'$ are such that $\bz_i=\bz_i'$ for all $i\leq j-1$ and $\bz_j\neq \bz_j'$ then $\Phi_{m,j}(\vec\bz)\neq \Phi_{m,j}(\vec\bz')$.
	}
	\begin{equation}\label{eq:Phi}
		\begin{aligned}
			\Phi_{m,i}(\vec\bz) &= \sqrt{Nq_1}\bz_1 + \sqrt{N(q_2-q_1)}O_1\circ T(\bz_2)+\cdots\\
			&+\sqrt{N(q_i-q_{i-1})}O_{1}\circ T\circ O_{2}\circ T\cdots O_{i-1}\circ T(\bz_i) ,
		\end{aligned}
	\end{equation}
	where, for any $\ell$,  set $T(x_1,\ldots,x_\ell):=(x_1,\ldots,x_\ell,0)$ and $O_i$ is the $(N-i+1)\times(N-i+1)$ rotation matrix that maps $\be_{N-i+1}$ to $\bz_i$ and acts as the identity on any vector in the orthogonal space to $\be_{N-i+1}$ and $\bz_i$.  Having done this,  let $A\subset \R^{N-1+\cdots+N-m}$ be some neighborhood of the origin and fixing some smooth chart $\bar\psi_m:A\to\chi_m$,  define $\psi_m:A\to \AA_m$ by 
	$\psi_m=\Phi_m\circ\bar\psi_m$,  choosing the chart $\bar\psi_m$ so that $\vec\by=\vec\by_m^\be=\psi_m(0)$. 
	
	With $\vec\by=\vec\by_m^\be\in\AA_m$ and $\vec\bx=\vec\bx_{m+1}^\be=(\bx^\be_1,\ldots,\bx^\be_{m+1})\in\AA^*_{m,m+1}(\vec\by)$,   
	note that
	\begin{equation}\label{eq:S(y)}
		\begin{aligned}
			S(\vec\by)&
			:=\big\{\bx: (\vec\by,\bx)\in \AA_{m,m+1}(\vec\by)\big\}\\
			&\,
			=\Big\{\bx^\be_{m}+\by:\,\frac1N\|\by\|^2=q_{m+1}-q_{m},\,\langle\by,\be_i\rangle=0,\forall i\leq m\Big\}\,,
		\end{aligned}
	\end{equation}
	is an $(N-m-1)$-dimensional sphere of radius $\sqrt{N(q_{m+1}-q_m)}=\|\bx^\be_{m+1}-\bx^\be_m\|$.  For any $\vec\by'\in\AA_m$ we map $\vec\bx\in\AA_{m,m+1}(\vec\by)$ to $\vec\bx'=\Phi_{m+1}(\vec\bz',\bz_{m+1})\in\AA_{m,m+1}(\vec\by')$
	where $\vec\bz=(\bz_1,\ldots,\bz_{m+1})=\Phi^{-1}_{m+1}(\vec\bx)$ and
	$\vec\bz'= \Phi^{-1}_{m}(\vec\by')$.  This induces a mapping $\pi_{\vec\by'}:S(\vec\by)\to S(\vec\by')$, which is an isometry. % Consider the map $\AA_m\times S(\vec\by)$ given by $(\vec\by',\bx)\mapsto H_N\circ \pi_{\vec\by'}(\bx)$. 
	Let $B\subset \R^{N-(m+1)}$ be some small  neighborhood of the origin and define
	$\Gamma_m(\ba,\bb)= (\psi_m(\ba),\pi_{\psi_m(\ba)}(\psi^m(\bb)))$,  where	
	\begin{equation}\label{eq:psim}
		\psi^m:\, B \ni \bb \mapsto M_{m+1}^{\top} (\vec\bx_{m+1}^\be)\bb+\alpha(\bx^\be_{m+1}-\bx^\be_{m}) + \bx^\be_{m}\in S(\vec\by),
	\end{equation}
	for the unique $\alpha>0$ such that $N(q_{m+1}-q_{m})\alpha^2+\|\bb\|^2=N(q_{m+1}-q_{m})$. 
	Finally,  recall that 
	\begin{equation}\label{eq:Thetam2}
		\Theta_m(\ba,\bb)=H_N((\Gamma_m(\ba,\bb))_{m+1})= H_N(\pi_{\psi_m(\ba)}(\psi^m(\bb)))\,.
	\end{equation}
	
	With these definitions all the properties we required above from the functions $\psi_m$, $\Gamma_m$ and $\Theta_m$ can be checked either directly from the definitions or by simple geometric considerations (with an appropriate choice of the chart $\bar\psi_m$).  The functions $\vec\by'\mapsto \psi_m^{-1}(\vec\by')$  and  $(\vec\by',\bb)\mapsto (\Gamma_m( \psi_m^{-1}(\vec\by'),\bb))_{m+1}$ defined after \eqref{eq:S(y)},  which map between manifolds and do not depend on $H_N(\cdot)$,  are easily seen to satisfy the Lipschitz bounds of Point \eqref{enum:pt6}.  Point \eqref{enum:pt4} follows since the push-forward of the standard basis
		$(\frac{d}{db_i})_{i\leq N-(m+1)}$ of the tangent space to $\bb\in B$ by
		$\pi_{\psi_m(\ba)}\circ \psi^m$ is a basis $(\bv_i)_{i\leq N-(m+1)}$ of the tangent space to $\pi_{\psi_m(\ba)}(\psi^m(\bb))\in S(\psi_m(\ba))$.
	%Point \eqref{enum:pt4} follows since the \textcolor{green}{standard Euclidean basis 
%	$(\frac{d}{db_i})_{i\leq N-(m+1)}$} is pushed-forward by 
%	$\psi^m$ to a basis $(\bv_i)_{i\leq N-(m+1)}$ of the tangent space to $\psi^m(\bb)$ in $S(\vec\by)$ and the latter basis is pushed-forward by $\pi_{\psi_m(\ba)}$ to a basis $(\bu_i)_{i\leq N-(m+1)}$ of the tangent space to $\pi_{\psi_m(\ba)}(\psi^m(\bb))$ in $S(\psi_m(\ba))$. 
Therefore,  using \eqref{eq:Thetam2},  the condition on the \abbr{rhs} of Point \eqref{enum:pt4} holds if and only if the derivative of $H_N(\cdot)$ at
	$(\Gamma_m(\ba,\bb))_{m+1}$ in any direction in the tangent space of $(\Gamma_m(\ba,\bb))_{m+1}$ in $S(\psi_m(\ba))$ is zero.
	For Point \eqref{enum:pt5} we note that upon substituting $\ba=0$ we obtain $\Theta_m(0,\bb)=H_N(\psi^m(\bb))$ and the second term in \eqref{eq:HessTheta} comes from  the term involving 
	$\alpha=\alpha(\bb)$ in the definition of $\psi^m$.
		
We proceed to show 
	that for any $\vec\by\in\AA_m$,  the number of points $\vec \bx  \in \AA^*_{m,m+1}(\vec\by)$ such that $\eta_{m+1}(\vec\bx)> \delta$ is bounded from above by a random constant $c$ that depends only 
	on $N$, $(q_i)_{i=1}^{k}$, $\delta$ and $C_N'$.  To this end,  let $\vec\bx=(\vec\by,\bx_{m+1})$ be such a point. Defining $\psi^m$ as in \eqref{eq:psim} (but now with general $\vec{\bx}$), we write
	\begin{align*}
		&H_N(\psi^m(\bb))=H_N\Big(\bx_{m+1}+M_{m+1}^{\top} (\vec\bx)\bb+(\bx_{m+1}-\bx_{m})\Big(-\frac12\frac{\|\bb\|^2}{{N(q_{m+1}-q_{m})}}+O\Big(\frac{\|\bb\|^4}{{N^2(q_{m+1}-q_{m})^2}}\Big)\Big) \Big)\\
		&= H_N(\bx_{m+1})+\frac12\Big\langle M_{m+1}(\vec\bx)\nabla^2 H_N(\bx_{m+1}) M_{m+1}(\vec\bx)^{\top} - \frac{\partial_{\bx_{m+1}-\bx_{m}}H_{N}(\bx_{m+1})}{\|\bx_{m+1}-\bx_{m}\|^2}
		\mathbf{I}, \bb \bb^\top\Big\rangle +O(\|\bb\|^3),
	\end{align*}
	where the linear term in $\bb$ was dropped since $M_{m+1}(\vec\bx)\nabla H_N(\bx_{m+1})=0$ and the constant of the $O(\|\bb\|^3)$ term depends only on $N$, $(q_i)_{i=1}^{k}$ and $C_N'$.
	Recall \eqref{eq:etai} and note that all the eigenvalues of the matrix have absolute value at least $\delta$. Thus, there is a ball  $B_0\subset B$ centered around the origin $0\in \R^{N-(m+1)}$ such that:
	\begin{itemize}
	\item  The center point $0\in B_0$ is a critical point of $H_N(\psi^m(\bb))$.
	\item  Any $\bb\in B_0\setminus\{0\}$ is not a critical point of $H_N(\psi^m(\bb))$.
%	and
	 \item The radius of $B_0$ only depends on $N$, $(q_i)_{i=1}^{k}$, $\delta$ and $C_N'$. 
	 \end{itemize}
		The image $\psi^m(B_0)$ is an open set in $S(\vec{\by})$. It contains a ball $B_0'\subset S(\vec{\by})$ with $\bx_{m+1}\in B_0'$ such that:
		\begin{itemize}
		 \item The point $\bx_{m+1}$ is a critical point of $H_N(\cdot)$ as a function on $S(\vec{\by})$.
		 \item Any $\bx\in B_0'\setminus\{\bx_{m+1}\}$ is not a critical point of $H_N(\cdot)$ (again, as a function on $S(\vec{\by})$).
		 % and
		  \item  The  radius of $B_0'$ only depends on $N$, $(q_i)_{i=1}^{k}$, $\delta$ and $C_N'$. 
		  \end{itemize}
		  To make the dependence on $\vec\bx$ explicit we write  $B_0'(\vec\bx)$ and define $B_0''(\vec\bx)$ as the same ball with half the radius.  Note that for any
		$\vec \bx  \in \AA^*_{m,m+1}(\vec\by)$ such that $\eta_{m+1}(\vec\bx)> \delta$ we have associated a ball $B_0''(\vec\bx)$.  Those balls are disjoint,  they all have the same radius and the number of such disjoint balls bounded by a random constant $c$ that depends only 
		on $N$, $(q_i)_{i=1}^{k}$, $\delta$ and $C_N'$. Hence, the same bound follows for the number of points $\vec\bx$, as we claimed above.
		
Note that by induction we also have that for any $n>m$,  the number of points 
$\vec\bx\in\AA^*_{m,n}(\vec\by)$
such that $\eta_{i}(\vec\bx)> \delta$ for all $i=m+1,\ldots,n$ is at most $c^{n-m}$. 
Further,  $\eta_i(\vec\bx)$ of \eqref{eq:etai} is a.s.\  continuous,  so with $\bar h_\delta(\cdot)$ continuous,  $\Psi^{\epsilon,\delta}_{m,n}(\vec\bx)$ is also a.s.\  continuous. 
	Recall that $\alpha_{n}(\vec\bx)$ is assumed to be  a.s.\  continuous,  so by compactness
	$\Psi^{\epsilon,\delta}_{m,n}(\vec\bx)\alpha_n(\vec\bx)$ is uniformly continuous.  Fix $\zeta>0$ and $\tau>0$ small enough so 
	$\|\Psi^{\epsilon,\delta}_{m,n}(\vec\bx)\alpha_n(\vec\bx)-\Psi^{\epsilon,\delta}_{m,n}(\vec\bx')\alpha_n(\vec\bx')
	\|<\zeta$ whenever $\|\vec\bx-\vec\bx'\|<\tau$.   Only at most $c^{n-m}$ 
	points $\vec\bx \in \AA^*_{m,n}(\vec\by)$ with $\min_{i>m}^n \{\eta_i(\vec\bx)\}>\delta$ can contribute 
	to $\alpha^{\epsilon,\delta}_{m,n}(\vec\by)$,  and if $\|\vec\by'-\vec\by\|<\tau /C$,  then by utilizing
	$f_{\vec\by,\vec\bx}$ of \eqref{eq:Lip_gN},  there
	exists some $\vec\bx' \in \AA^*_{m,n}(\vec\by')$ such that $\|\vec\bx-\vec\bx'\| < \tau$.  Consequently,  
	for any $\vec\by,\vec\by'\in\AA_m$ such that $\|\vec\by'-\vec\by\|<\tau /C$,
	\[
	\alpha^{\epsilon,\delta}_{m,n}(\vec\by')  \geq \alpha^{\epsilon,\delta}_{m,n}(\vec\by)-\zeta c^{n-m}.
	\]
	Since the same inequality holds if we interchange $\vec\by$ and $\vec\by'$,  the lemma follows. 
\end{proof}

\begin{proof}[Proof of Proposition \ref{prop:multilvlKR}]
Equipped with Lemmas \ref{lem:etapositive_conditional} and \ref{lem:alpha_continuity},  
we fix $0 \leq m \leq n-1$ and proceed to derive \eqref{eq:KRalphaeps}.
To this end,  recall first that by definition (see \eqref{eq:alpha^eps}),
\[
\alpha_{m,n}^{\epsilon}(\vec\bx_m^{\be})\wedge e^{MN} 
%=	 \Big( \sum_{\vec\by\in \mathscr{C}_{m,m+1}(\epsilon,\vec\bx_m^{\be})}
%\alpha_{m+1,n}^{\epsilon}(\vec\by)\Big)\wedge e^{MN}
\leq 
\sum_{\vec\by \in \mathscr{C}_{m,m+1}(\epsilon,\vec\bx_m^{\be})}\big(
\alpha_{m+1,n}^{\epsilon}(\vec\by)\wedge e^{MN}\big).
\]
Denoting hereafter by $\E_0$ the expectation conditional on $\cpt(\vec \bx_{m}^\be,\vec E,\vec R)$,  we thus have that
\begin{equation*}
\E_0\Big[ \alpha_{m,n}^{\epsilon}(\vec\bx_m^{\be}) \wedge e^{MN} \,\Big] \leq 
\E_0 \Big[ \sum_{\vec\by\in \mathscr{C}_{m,m+1}(\epsilon,\vec\bx_m^{\be})}\big(
\alpha_{m+1,n}^{\epsilon}(\vec\by)\wedge e^{MN}\big) \Big].
\end{equation*}
Hence,  by our consequence \eqref{eq:alpha-del} of Lemma  \ref{lem:etapositive_conditional} and the  
monotone convergence theorem,    
\begin{equation}\label{eq:KRm1}
\E_0\Big[ \alpha_{m,n}^{\epsilon}(\vec\bx_m^{\be}) \wedge e^{MN} \,\Big] \le
%\E_0 \Big[ \sum_{\vec\by\in \mathscr{C}_{m,m+1}(\epsilon,\vec\bx_m^{\be})}\big(
% \alpha_{m+1,n}^{\epsilon}(\vec\by)\wedge e^{MN}\big) \Big] =
\lim_{\delta\to0}\E_0\Big[ \sum_{\vec\by\in \mathscr{C}_{m,m+1}(\epsilon,\vec\bx_m^{\be})}\big(
\alpha_{m+1,n}^{\epsilon,\delta}(\vec\by)\wedge e^{MN}\big) \Big].
\end{equation}

Lemma \ref{lem:alpha_continuity} about the a.s.\  continuity of 
$\vec \by \mapsto \alpha_{m+1,n}^{\epsilon,\delta}(\vec \by)$ is basically deterministic. 
Precisely, its proof only uses that 
$\bs \mapsto H_N(\bs)$ is twice continuously differentable.  It therefore also holds conditionally on $\cpt(\vec \bx_{m}^\be,\vec E,\vec R)$. Of course, $\Psi^{\epsilon,\delta}_{m+1,n}(\vec\bx)=\bar\Psi^{\epsilon,\delta}_{m+1,n}(\vec\bx,H_N(\cdot))$ for some rotationally invariant deterministic $\bar\Psi^{\epsilon,\delta}_{m+1,n}$ and one can therefore easily verify that $\alpha_{m+1,n}^{\epsilon,\delta}(\vec\bx) =\bar\alpha_{m+1,n}^{\epsilon,\delta}(\vec\bx,H_N(\cdot))$ for rotationally invariant $\bar\alpha_{m+1,n}^{\epsilon,\delta}$.
Since taking maximum with $e^{MN}$ preserves both continuity and rotational invariance,  subject to \eqref{eq:ERinversetransform},  we have 
from \eqref{eq:KRunifrombd},\footnote{We note that to use \eqref{eq:KRunifrombd}, here we view $\alpha_{m+1,n}^{\epsilon,\delta}(\vec\bx)$ with $\vec\bx=(\vec\bx_m^{\be},\bx_{m+1})$ as a function of $\bx_{m+1}$. Indeed, in the proofs of Lemma \ref{lem:1lvlInductionIndic} and Corollary \ref{cor:1lvlKR_E}, we always work with $\bx_{m+1}$ such that $(\vec\bx_m^{\be},\bx_{m+1})\in\Cs_{m,m+1}(\epsilon,\vec\bx_m^\be)$.  See Footnote \ref{fn:C}.} that  
\begin{equation}\label{eq:KRmulti2}
\begin{aligned}
& \E_0\Big[   \sum_{\vec\by \in \Cs_{m,m+1}(\epsilon,\vec\bx_m^\be)} \!\!
	\big(
	\alpha_{m+1,n}^{\epsilon,\delta}(\vec\by)\wedge e^{MN}\big)
	 \, \Big] \le C_{N,m} \bigg\{  6 e^{M N} e^{-C_N N} 
	 \\
		& +  \int_{I_1\times I_2} \Gamma_{N,R}(t_N) 
		\Big( e^{-N^2} + 5 \E \big[  \alpha_{m+1,n}^{\epsilon,\delta}(\vec\bx_{m+1}^{\be})\wedge e^{MN} 
		 \, \big| \, \cpt(\vec \bx_{m+1}^\be,\vec E,\vec R) \big] \Big)
		 \varphi_{N,m}(E,R)dEdR  \bigg\}\,,
\end{aligned}
\end{equation}
where $C_N\to\infty$ and $t_N \to 0$ do not depend on $\delta$ or $(\vec E,\vec R,E,R)$
(see remark after \eqref{eq:KRunifrombd}).  
Since $\alpha_{m+1,n}^{\epsilon,\delta}(\cdot)\leq \alpha_{m+1,n}^{\epsilon}(\cdot)$,  we can and shall replace
$\alpha_{m+1,n}^{\epsilon,\delta}(\vec\by)$ by $\alpha_{m+1,n}^{\epsilon}(\vec\by)$ on the \abbr{rhs} of 
 \eqref{eq:KRmulti2}.  At this point we take $\delta \downarrow 0$,  and combining \eqref{eq:KRm1} with 
 \eqref{eq:KRmulti2},  arrive at
\begin{equation*}
% \label{eq:KRmulti3}
\begin{aligned}
\E \big[ & \alpha_{m,n}^{\epsilon}(\vec\bx_m^{\be} ) \wedge e^{MN} \, \big| \, \cpt(\vec \bx_{m}^\be,\vec E,\vec R) \big]
	  \le C_{N,m} \bigg\{  6 e^{M N} e^{-C_N N} 
	 \\
		& +  \int_{I_1\times I_2} \Gamma_{N,R}(t_N) 
		\Big( e^{-N^2} + 5 \E \big[  \alpha_{m+1,n}^{\epsilon}(\vec\bx^{\be}_{m+1})\wedge e^{MN} 
		 \, \big| \, \cpt(\vec \bx_{m+1}^\be,\vec E,\vec R) \big] \Big)
		 \varphi_{N,m}(E,R)dEdR  \bigg\}\,,
\end{aligned}
\end{equation*}
from which we get \eqref{eq:KRalphaeps} by the same reasoning as we have done after \eqref{eq:A-new}. 
\end{proof}

\section{Proof of Property \eqref{enum:PS7} of Proposition \ref{prop:PS}}\label{sec-pf-prop:PS}

Recall from Section \ref{sec:putestates} that fixing some $c_N \to 0$ slowly enough,  
the vertex sets $\V_N \subset \R^N$ of full regular trees $\T_N$ of degrees $d_N \to \infty$ 
and depth $k$ can with probability going to $1$ as $N \to \infty$,  be chosen so that 
% be made into a 'good' pure state decomposition where
 for some non-random $\delta_N,\epsilon_N \to 0$,
the collections $\V_N$ and the disjoint subsets $\{B_i, i \le d_N^k\}$ of $\SN$ associated
with the leaves of $\T_N$,  have all the properties
stated in Proposition \ref{prop:PS},  except possibly Property \eqref{enum:PS7}.
Hereafter, we call such $(\T_N,(B_i)_i,c_N,d_N,\delta_{N},\epsilon_{N})$ a \emph{good pure state decomposition}. 
For any $m \le k$,  
denote by $\V_{N, m}$ the subset of $\V_N$ of depth $m$.  In addition,  for the root $\V_{N,0}$ 
% of $\T_N$
we let $\vec{\bv}=\bv={\bf 0}$,  the origin in $\R^N$,  
while for any $\bv\in \V_{N,m}$ (with $|\bv|=m\geq1$),  we denote by 
$\vec{\bv}=(\bv_1,\ldots,\bv_m)$ the path from the root to $\bv=\bv_m$,  excluding the root.  
Recalling the definition \eqref{eq:C0meps} of $\Cs^0_m(\epsilon)$,  it suffices to show by induction 
on $m \le k$,  that 
\begin{equation}\label{eq:claim}
	\begin{aligned}
		&\mbox{there exists a good pure states decomposition with 
		$\vec\bv \in \Cs^0_m(\epsilon_N)$ for all $\bv \in \V_{N,m}$.}
	\end{aligned}	
\end{equation}
Indeed,  from Corollary \ref{cor:CshCs} we know that for any fixed $\epsilon'>0$,  for some  $\eta_N=\eta_N(\epsilon')\to0$,
	\[
	\P (	 \Cs^0_{k}(\epsilon_{N}) \not\subset \Cs_{k}(\epsilon') ) \le \eta_N\,.
	\]
Hence,  after further enlarging $\epsilon_N$,  Proposition \ref{prop:PS} follows from \eqref{eq:claim} at $m=k$.  

In case $m=0$,  we obviously have \eqref{eq:claim}.  Thus,  we assume hereafter that \eqref{eq:claim} holds for some 
$m \ge 0$ and show that it then must hold also at $m+1$,  by finding non-random $a_N,  b_N \to 0$
(with $a_N \ge \epsilon_N$),  and a (random) mapping $\bv_{m+1} \mapsto \bv'_{m+1}$ of $\V_{N,m+1}$,  such that with probability going to 1 as $N \to \infty$,
\begin{align}\label{eq:claim-new}
 \sup_{\V_{N,m+1}}  & \| \bv'_{m+1}  - \bv_{m+1} \| \le b_N  \sqrt{N} \quad \textrm{and} \quad
\{ \vec \bv'_{m+1}  \} \subset  \Cs^0_{m+1}(a_N) \,.
\end{align}
Indeed,  
%\abbr{wlog} take $a_N \ge \epsilon_N+4 b_N$ and 
keeping the disjoint sets $\{B_i\}$ of our good pure state decomposition,  while 
increasing both $\delta_N$ and $\epsilon_N$ by $4 b_N$,  induces Properties 
(\ref{enum:PS1})-(\ref{enum:PS4}) of Proposition \ref{prop:PS} for our modified vertex set.
% (1) - {enum:PS1} and (4) - {enum:PS4} increasing in \epsilon_N and require only {B_i}
% (2) - {enum:PS2} given {B_i} verify that |\bx| \le \sqrt{N},  |<\bx-\bu,\bu>|< N \delta and |\bv-\bu| \le b \sqrt{N} 
%                                                                                           ==> |<\bx-\bv,\bv>| \le N (\delta + 4 b)
% (3) - {enum:Pt3}  verify that |<\bv',\bv> - <\bu',\bu>| \le 4 b N when |\bv'-\bv| \le b \sqrt{N}
% (5) - {enum:PS7} need  \Cs^0_{m+1}(\epsilon_N) \subset \Cs_{m+1} (\epsilon_N')
Likewise,  \eqref{eq:energyontree} transfers to the modified vertex set,  since \abbr{whp}
$\bv \mapsto H_N(\bv)$ is $C\sqrt N$-Lipschitz and enlarging $\epsilon_N$ to $a_N$,  
we now have that \eqref{eq:claim} holds at $m+1$. 

Turning to establish \eqref{eq:claim-new},  
recall that $\Cs_m^0(\epsilon) \subset \AA_m$ of \eqref{eq:AA} and the definition 
\eqref{eq:S(y)} of $S(\vec\by)$.  Setting $r_{m+1} := \sqrt{N (q_{m+1}-q_m)}$,  
for each vertex $\bv_{m}\in\V_{N,m}$ at level $m$ we denote by $p(\cdot)$ the orthogonal projection onto 
$\mbox{span}\{\bv_1,\ldots,\bv_{m}\}^{\perp}$ and consider the decomposition obtained by replacing 
each of its descendants  $\bv_{m+1}\in\V_{N,m+1}$ by its image $\bu_{m+1}$ under the mapping 
\begin{equation*}\label{eq:proj}
\bv_{m+1} \mapsto \bu_{m+1} := \bv_{m}+\frac{r_{m+1}}{\|p(\bx)\|} p(\bx)
 \in S(\vec \bv_m) \,,  \quad \textrm{where} \quad \bx:=\bv_{m+1}-\bv_m \,.
\end{equation*}
Now, with $\{ \bv_i-\bv_{i-1},  i \le m\}$,  orthogonal of norms $r_i$,  we have for some finite 
 $c_m=c_m(\{q_i\})$,
%c_m =2m/\min_{i \le m} (q_i-q_{i-1})
\[
\frac{1}{\sqrt{N}} \|\bx - p(\bx)\| = \Big[ \sum_{i=1}^m \frac{1}{N r_i^2} \langle \bx, \bv_i-\bv_{i-1} \rangle^2 \Big]^{1/2} 
\le \frac{c_m}{N} \max_{i=1}^m |\langle \bx,\bv_i\rangle| \le 2c_m \epsilon_N \,,
\]
since by Property \eqref{eq:Pt3} of Proposition \ref{prop:PS}, 
$|\langle \bx,\bv_i \rangle| 
%=|\langle \bv_{m+1}- \bv_m, \bv_i \rangle| 
= |\langle \bv_{m+1},\bv_i\rangle - \langle \bv_{m},\bv_i\rangle | \le 2\epsilon_N N$ for all $i \le m$ (where for $i=1$, we define above $\bv_{i-1}=\bv_{0}=0$).
Similarly,  by our induction hypothesis,  \eqref{eq:claim} holds.  Hence,  $\|\bv_m\|^2 = N q_m$ with
\[
\big|  \| \bx \|^2 - r_{m+1}^2 \big| \le \big| \|\bv_{m+1}\|^2 - N q_{m+1} \big| +
2 |\langle \bv_{m+1}, \bv_m \rangle - N q_m | \le 3 \epsilon_N N 
\]
and since by the triangle inequality
\begin{align*}
	\|\bv_{m+1}-\bu_{m+1}\| &\le \|\bx-p(\bx)\|+\|p(\bx)-r_{m+1} \frac{p(\bx)}{\|p(\bx)\|} \| \\
	&= \|\bx-p(\bx)\|+\big|\|p(\bx)\|-r_{m+1} \big|  \le 2 \|\bx-p(\bx)\| + \big|\|\bx\|-r_{m+1}\big|  \,,
\end{align*}
%\[
%\|\bv_{m+1}-\bu_{m+1}\|^2 = \|\bx-r_{m+1} \frac{p(\bx)}{\|p(\bx)\|} \|^2 
% = (r_{m+1} - \|p(\bx)\|)^2 + \| \bx - p(\bx) \|^2 
%\le 2 r_{m+1} \|\bx - p(\bx)\| + \big|  \|\bx\|^2-r_{m+1}^2 \big|  \,,
%\]
we deduce that $\|\bv_{m+1}-\bu_{m+1}\| \le b'_N \sqrt{N}$ 
% (b'_N)^2 = (2 c_m+3) \epsilon_N
for some $b'_N=b'_N(m,\epsilon_N) \to 0$.  As explained above,  \abbr{whp}  
$H_N(\cdot)$ is $C\sqrt N$-Lipschitz and
$\big|\frac{1}{N}H_N(\bv_{m+1})-\Es(q_{m+1})\big| \le a_N$
for some non-random $a_N \to 0$.  Hence,  after suitably increasing $a_N$,  we have that \abbr{whp}
\begin{equation}\label{eq:sup-level-set}
\sup_{\V_{N,m+1}} \Big|\frac{1}{N}H_N(\bu_{m+1})-\Es(q_{m+1})\Big| \leq a_N.
\end{equation}
Now,  having $\vec \bv_m \in \Cs_m^0(\epsilon_N)$ and $\bu_{m+1} \in S(\vec \bv_m)$ 
% near $\bv_{m+1}$
satisfying \eqref{eq:sup-level-set},  we apply our next lemma at $\epsilon = a_N$ (and some
$\tau=\tau_N \to 0$),  to get with probability going to $1$ 
as $N \to \infty$,  a perturbation $\bv'_{m+1} \in S(\vec \bv_m)$ of $\bu_{m+1}$ 
for which  \eqref{eq:claim-new} holds (say,  with $b_N=b'_N+\tau_N$).
\begin{lem}
	For any $m\leq k-1$ and $\tau>0$, if $\epsilon>0$ is sufficiently small,
	\begin{align*}
	\varlimsup_{N\to\infty} \frac1N\log\bigg(1-	\P\Big[
	\forall \vec\bx_m\in\Cs^0_m(\epsilon),\,&\bx\in S(\vec\bx_m):\,\frac{1}{N}H_N(\bx)\geq\Es(q_{m+1})-\epsilon \implies \\
	&\exists \bx'\in S(\vec\bx_m),\,\|\bx-\bx'\|\leq \tau\sqrt{N}, (\vec\bx_m,\bx')\in\Cs^0_{m+1}(\epsilon)
	\Big]
	\bigg)<0.
	\end{align*}
\end{lem}
\begin{proof}
	By the Borell-TIS inequality, with probability at least $1-e^{-cN}$ for some $c=c(\xi,\epsilon)$,
	\begin{equation}\label{eq:Ebdlvlset}
		\sup_{\bx\in \sqrt{q_{m+1}}\SN}  \Big\{ \frac1N	H_N(\bx) \Big\} \leq \Es(q_{m+1})+\epsilon.
	\end{equation}
	Suppose $\vec\bx_m=(\bx_1,\ldots,\bx_m)\in\Cs^0_m(\epsilon)$ and 
	recall that $S(\vec\bx_m)$ is a sphere of  dimension  $N-m-1$ and radius $\sqrt{N(q_{m+1}-q_m)}$. 
	Consider the subsets
	\begin{align*}
	S^+_\epsilon(\vec\bx_m)&:=\Big\{
	\bx\in S(\vec\bx_m):\,\frac{1}{N}H_N(\bx)\geq\Es(q_{m+1})-\epsilon
	\Big\},\\
	S_\epsilon(\vec\bx_m)&:=\Big\{
	\bx\in S(\vec\bx_m):\,\Big|\frac{1}{N}H_N(\bx)-\Es(q_{m+1})\Big|\leq \epsilon
	\Big\}.
	\end{align*}
	Assume that the event in \eqref{eq:Ebdlvlset} occurs, so that $S_\epsilon(\vec\bx_m)=S^+_\epsilon(\vec\bx_m)$, since $S(\vec\bx_m)\subset\sqrt{q_{m+1}}\SN$.
	
	Each of the connected components of $S_\epsilon(\vec\bx_m)=S^+_\epsilon(\vec\bx_m)$ contains at least one local maximum and, in particular, a critical point $\bx_{m+1}$ of $H_N(\bx)$.  Namely, a point $\vec\bx_{m+1}=(\bx_1,\ldots,\bx_{m+1})$ such that $\gradt H_N(\bx_{m+1})=0$, where the latter gradient is defined above \eqref{eq:xi_condition}. To prove the implication as in the lemma, assuming \eqref{eq:Ebdlvlset}, it is therefore sufficient to show that for any critical point $\bx_{m+1}$ in $S_\epsilon(\vec\bx_m)$, 
	%with $H_N(\bx_{m+1})/N\geq \Es(q_{m+1})-\epsilon$, 
	the connected component of $S_\epsilon(\vec \bx_m)$ containing $\bx_{m+1}$ 
	has diameter less than $\sqrt{N}\delta$ for some 
	$\delta=\delta(\epsilon) \to 0$ as $\epsilon \to 0$.  The latter claim about the diameter follows if for some $\eta=\eta(\epsilon) \to 0$ and any such
	%critical point 
	$\bx_{m+1}$
	\begin{equation}\label{dfn:Ups}
	\Upsilon_{N}(\vec\bx_{m+1}):=\frac1N\sup_{\substack{\bx\in S(\vec\bx_m):\\ \frac1N\langle\bx,\bx_{m+1}\rangle=q_{m+1}-\eta}} \Big\{ H_N(\bx) \Big\} < \Es(q_{m+1})-\epsilon\,,
	\end{equation}
since this implies that any $\bx\in S(\vec\bx_m)$ with $\frac1N\langle\bx,\bx_{m+1}\rangle=q_{m+1}-\eta$, equivalently with $\frac1N\|\bx-\bx_{m+1}\|^2=2\eta$, does not belong to $S^+_\epsilon(\vec\bx_m)$.
	We conclude that in order to prove the lemma, it suffices to show that for any fixed $\eta>0$ and $\epsilon<\epsilon_0(\eta)$,
	% , for some $\epsilon_0(\eta)>0$. 
	\begin{align}\label{eq:topLemE1}
		\varlimsup_{N\to\infty} \frac1N\log  \E\Big[1 \wedge \sum_{\vec\bx \in\Cs^0_{m+1}(\epsilon)} \indic
		\big\{ \Upsilon_{N}(\vec\bx)\ge   \Es(q_{m+1})- \epsilon \big\} \Big]&<0.
	\end{align}
By Corollary \ref{cor:CshCs},  with overwhelming probability, $\Cs^0_{m+1}(\epsilon)\subset \Cs_{m+1}(\epsilon')$ whenever $\epsilon<\epsilon_1(\epsilon')$.  Hence,  \eqref{eq:topLemE1} holds if for any fixed $\eta$,  
small enough $\epsilon'=\epsilon'(\eta)$ and any $\epsilon<\epsilon_1(\epsilon')$,
\begin{align}\label{eq:topLemE2}
	\varlimsup_{N\to\infty} \frac1N\log
	 \E\Big[1 \wedge \sum_{\vec\bx \in\Cs_{m+1}(\epsilon')} \indic
		\big\{ \Upsilon_{N}(\vec\bx) \ge \Es(q_{m+1})- \epsilon  \big\} \Big]
	&<0.
\end{align}
The bounded,  continuous function $\tilde h_\epsilon(t) := 1 \wedge (\frac{t}{\epsilon} + 1)_+$ exceeds  
the indicator on $\R_+$,  hence the \abbr{lhs}
%Obviously,  it would also be sufficient to prove either of the two bounds \eqref{eq:topLemE1}-\eqref{eq:topLemE2} with 
%the expectation of the minimum between the cardinality of the set and $e^{MN}$, where $M>0$ is an arbitrary absolute constant (e.g., $M=1$). With the minimum added, 
of \eqref{eq:topLemE2} is bounded from above by
\begin{align}\label{eq:topLemE3}
	\varlimsup_{N\to\infty} \frac1N\log \E \Big[\Big(\sum_{\vec\bx\in\Cs_{m+1}(\epsilon') } \alpha_{m+1,\epsilon}(\vec\bx)\Big)\wedge 1 \Big],
\end{align}
where we have set
\[
\alpha_{m+1,\epsilon}(\vec\bx) := \tilde h_\epsilon(\Upsilon_{N}(\vec\bx) - \Es(q_{m+1}) + \epsilon)\,.
\]
The $[0,1]$-valued $\alpha_{m+1,\epsilon}(\vec\bx)=\bar \alpha_{N,m+1} (\vec\bx,H_N(\cdot))$ is a.s.\  
continuous on $\AA_{m+1}$,  where the non-random,  rotationally invariant $\bar \alpha_{N,m+1}(\vec\bx,\varphi)$
that corresponds to replacing $H_N(\bx)$ by $\varphi(\bx)$ in \eqref{dfn:Ups},  satisfies \eqref{eq:alphalimits}.  
Consequently,  by Proposition \ref{prop:multilvlKR},  the limit in \eqref{eq:topLemE3} is bounded from above by 
	\begin{align}\label{eq:PSKRfinalform}
	\sum_{n=0}^{m}\sup_{I_n(\epsilon')\times I'_n(\epsilon')}\Theta_{\xin}(E,R)  +
	\varlimsup_{N\to\infty}\frac1N\log\Big(\sup_{\bar V(\epsilon')}\E\big[ \alpha_{m+1,\epsilon}(\vec\bx_{m+1}^{\be})   \,\big|\,\cpt(\vec \bx_{m+1}^\be,\vec E,\vec R) \big]\Big),
\end{align}
% where we dropped the minimum since $\alpha_{m+1}(\vec\bx)\leq1$ 
and we recall that
\begin{equation*}
	\begin{aligned}
		I_m(\epsilon')&=[\Es(q_{m+1})-\Es(q_m)-2\epsilon',\Es(q_{m+1})-\Es(q_m)+2\epsilon'],\\
		I'_m(\epsilon')&=(q_{m+1}-q_m)[\Rs(q_{m+1})-\epsilon',\Rs(q_{m+1})+\epsilon'].
	\end{aligned}
\end{equation*}
Since $(E,R)\mapsto\Theta_{\xi}(E,R)$ is continuous, by Corollary \ref{cor:Thetastar},  the sum over $n$
in \eqref{eq:PSKRfinalform} can be made as small as we wish by taking $\epsilon'$ small enough.

By our choice of $\tilde h_\epsilon$,  the expectation in \eqref{eq:PSKRfinalform} is bounded from above by the conditional probability that $\Upsilon_{N}(\vec\bx_{m+1}^{\be})\geq\Es(q_{m+1})-2\epsilon$.
By Corollary \ref{cor:Hm1},  the latter
conditional probability is equal to 
\begin{align*}
\P\Big(
\frac{1}{N-m} \sup_{\substack{\bz\in\mathbb{S}^{N-m-1}:\\ \frac{1}{N-m}\langle \bz,\bs\rangle=1-\eta'}} H_N^{(m)}(\bz)\geq E'\,\Big|\,\bs\in \Crt_{N-m,\xim}(E,R)
\Big)\,,
\end{align*}
where $\bs\in\mathbb{S}^{N-m-1}$ is arbitrary, $\eta'=\frac{\eta}{q_{m+1}-q_m}$,
\begin{align*}
	E'&= \sqrt{\frac{N}{N-m}}\Big(
	\Es(q_{m+1})-E_m-2\epsilon\Big),\\
	E&= \sqrt{\frac{N}{N-m}}\big(
	E_{m+1}-E_m\big),\\
	R&=\sqrt{\frac{N}{N-m}}\big(q_{m+1}-q_m\big) R_{m+1},
\end{align*}
and $\Crt_{N-m,\xim}$ is defined as in Section \ref{sec:complexity},  
albeit now for the spherical model $H_N^{(m)}(\bz)$ of mixture $\xim$.

For $N$ sufficiently large,  by choosing $\epsilon'$ sufficiently small we can make sure that uniformly over $\bar V(\epsilon')$, $E'$, $E$ and $R$ are as close as we wish to $\Es(q_{m+1})-\Es(q_m)-2\epsilon$, $\Es(q_{m+1})-\Es(q_m)$ and $(q_{m+1}-q_m)\Rs(q_{m+1})$, respectively.
Hence, the proof is completed by Lemma \ref{lem:EsRs} and  Proposition \ref{prop:HS2}, noting that the bound in 
Proposition \ref{prop:HS2} only improves as we decrease $a$ and $b$.
\end{proof}

\section{\label{sec:LowTempPf} Low temperature: proof of Theorem \ref{thm:low temp 1}}

Recall the measure 
\begin{equation*}\label{eq:Pi}
\Pi^\epsilon_{N,\beta}(\,\cdot\,):= 
\sum_{\vec\bx\in \mathscr{C}_{k}(\epsilon)} G_{N,\beta}\big(\,\cdot\,\big|\, B(\vec\bx,\epsilon)\big)
=\sum_{\vec\bx\in \mathscr{C}_{k}(\epsilon)} \frac{G_{N,\beta}\big(\,\cdot\,\cap\, B(\vec\bx,\epsilon)\big)}{G_{N,\beta}\big(B(\vec\bx,\epsilon)\big)},
\end{equation*}
and subset 
\begin{align*}
	%\label{eq:Band}
	B(\vec \bx,\epsilon)&:=\left\{ \bs\in\SN:\,|\langle\bs,\bx_i\rangle-\langle\bx_i,\bx_i\rangle|\leq N\epsilon,\,\forall i\leq k \right\} .
\end{align*} 
which we defined in \eqref{eq:barPi} and \eqref{eq:Band}. We denote by $\mathcal{B}(\SN)$ the set of Borel measurable subsets of $\SN$. 
From the discussion of Section \ref{subsec:PfSketch_purestates}, for $\epsilon_N \to 0$ of 
Proposition \ref{prop:PS} and any fixed $\epsilon>0$,
\begin{equation*}%\label{eq:SG}
	\lim_{N\to\infty} \P \Big(\forall M\in \mathcal{B}(\SN):\, \Pi_{N,\beta}^\epsilon(M)\wedge 1\geq G_{N,\beta}(M)-\epsilon_{N}\Big)=1
\end{equation*}
and thus to prove \eqref{eq:GAverageLowT}, it is enough to show  that
\begin{equation}
	\varlimsup_{\epsilon\to0}\varlimsup_{N\to\infty}\frac{1}{N}\log \Big(\E \Big[\Pi^{\epsilon}_{N,\beta}\big(f_{N}(\bs)\in A \big)\wedge 1\Big]\Big)<0.\label{eq:PiAverageMin}
\end{equation}
In contrast to Section \ref{subsec:PfSketch_purestates}, here we take the minimum with $1$ which is justified since $G_{N,\beta}(\SN)= 1$.

For a notation consistent with Proposition \ref{prop:multilvlKR},  we define, for $\vec\bx\in\AA_k$,
\begin{equation}\label{dfn:alpha-k}
\alpha_k(\vec\bx) =\alpha_{k,\epsilon} (\vec\bx) = G_{N,\beta}\Big(\,\big\{ \bs\in B(\vec\bx,\epsilon):\,f_{N}(\bs)\in A  \big\} \Big)/ G_{N,\beta}(B(\vec\bx,\epsilon)),
\end{equation}
so that, in the notation of Section \ref{sec:multiLvlKR},
\[
\Pi^{\epsilon}_{N,\beta}\big(f_{N}(\bs)\in A \big)\wedge 1 = \alpha^{\epsilon}_{0,k}(0) \wedge 1.
\]
Written more explicitly in terms of $H_N(\bs)$ and $\bar f_N(\bs,H_N)$,  the function 
$\alpha_{k,\epsilon} (\vec\bx)$ is given by
\begin{equation}\label{eq:alphaPS}
\alpha_{k,\epsilon} (\vec\bx) 
%= \frac{\int_{B(\vec\bx,\epsilon)}\indic\{ f_{N}(\bs)\in A \}e^{\beta H_N(\bs)}d\bs}{\int_{B(\vec\bx,\epsilon)}e^{\beta H_N(\bs)}d\bs}
=\frac{\int_{B(\vec\bx,\epsilon)}\indic\{ \bar f_{N}(\bs,H_N)\in A \}e^{\beta H_N(\bs)}d\bs}{\int_{B(\vec\bx,\epsilon)}e^{\beta H_N(\bs)}d\bs}.
\end{equation}

Note that $\alpha_{k,\epsilon}$ is a.s.\ continuous on $\AA_k$ by bounded convergence, since the volume of the symmetric difference $B(\vec\bx,\epsilon)\vartriangle B(\vec\bx',\epsilon)$ goes to zero as $\vec\bx\to\vec\bx'$. Bounded convergence, together with \eqref{eq:flim},  also implies the continuity property of \eqref{eq:alphalimits}. In \eqref{eq:alphaPS} we have $\alpha_{k,\epsilon}(\vec\bx)$ in the form $\bar\alpha_{k,\epsilon}(\vec\bx,H_N(\cdot))$ for  deterministic $\bar\alpha_{k,\epsilon}$.  By a change of variables, its rotational invariance is inherited from $\bar f_N$.  Further,  $\alpha_{k,\epsilon}(\vec\bx_{k}^{\be})\leq1$ a.s.\  and 
hence,  by Proposition \ref{prop:multilvlKR}, we have that the \abbr{lhs} of \eqref{eq:PiAverageMin} is bounded by the sum of
\begin{align}\label{eq:vanishingComplexity}
	\varlimsup_{\epsilon\to0}
	\sum_{m=0}^{k-1}\sup_{I_m(\epsilon)\times I'_m(\epsilon)}\Theta_{\xim}(E,R) 
\end{align}
and 
\begin{align}\label{eq:1ptE}
	\varlimsup_{\epsilon\to0}
	\varlimsup_{N\to\infty}\frac1N\log\Big(\sup_{\bar V(\epsilon)}\E\big[ \alpha_{k,\epsilon}(\vec\bx_{k}^{\be}) \,\big|\,\cpt(\vec \bx_{k}^\be,\vec E,\vec R) \big]\Big).
\end{align}

Recall that, since $(E,R)\mapsto\Theta_{\xi}(E,R)$ is continuous and \eqref{eq:kRSB} holds,  by Corollary \ref{cor:Thetastar}, \eqref{eq:vanishingComplexity} is zero,  so we finish the proof upon showing that 
\eqref{eq:1ptE} is negative.
% prove the following.
%\begin{lem}
%	In the setting above,
%	\begin{align*}%\label{eq:1ptE2}
%		\varlimsup_{\epsilon\to0}
%		\varlimsup_{N\to\infty}\frac1N\log\bigg(\sup_{\bar V(\epsilon)}\E\bigg[ G_{N,\beta}\Big(\,\big\{ \bs\in B(\vec\bx_k^{\be},\epsilon):\,f_{N}(\bs)\in A  \big\} \Big) /G_{N,\beta}( B(\vec\bx_k^{\be},\epsilon)) \,\bigg|\,\cpt(\vec \bx_{k}^\be,\vec E,\vec R) \bigg]\bigg)<0.
%	\end{align*}
%\end{lem}
% \begin{proof}
%	If the bound in the lemma does not hold, then 
By standard diagonalization argument,  it suffices to consider in \eqref{eq:1ptE} only the limit superior in $N$,  for a
fixed,  arbitrary sequence $\epsilon_{N}\to0$.  That is,  it suffices to show that
	\begin{align}\label{eq:1ptE3}
		\varlimsup_{N\to\infty}\frac1N\log\Big(\sup_{\bar V(\epsilon_N)}
		\E\big[ \alpha_{k,\epsilon_N}(\vec\bx_{k}^{\be}) \,\big|\,\cpt(\vec \bx_{k}^\be,\vec E,\vec R) \big]
	%	\E\Big[ G_{N,\beta}\Big(\,\big\{ \bs\in B(\vec\bx_k^{\be},\epsilon_N):\,f_{N}(\bs)\in A  \big\} \Big) /G_{N,\beta}( %B(\vec\bx_k^{\be},\epsilon_N)) \,\Big|\,\cpt(\vec \bx_{k}^\be,\vec E,\vec R) \Big]
	\Big)<0.
	\end{align} 
	
	To this end,  suppose that $F_{N,\beta}^\gamma$ is the free energy of the Gibbs measure
	$G^{\gamma}_{N,\beta}$ associated to $H_N^\gamma$,  where 
	$\big(H^\gamma_N(\bs)\big)_{\gamma\in\Gamma}$ is,  for each $N$,   
	a family of Gaussian fields on the same probability reference space $(\Sigma_N,\cF_N,\mu_N)$,
	such that for some fixed $v>0$,
	\begin{align}\label{eq:cond-Gamma}
	\lim_{N\to\infty}\sup_{\gamma\in\Gamma,\bs\in\bar\Sigma_N}\Big|\frac1N\E\big[H^\gamma_N(\bs)^2\big]-v\Big| &= \lim_{N\to\infty}\sup_{\gamma\in\Gamma,\bs\in\bar\Sigma_N}\Big|\frac1N\E\big[H^\gamma_N(\bs)\big]\Big| \nonumber \\
	& = \lim_{N\to\infty} \sup_{\gamma \in \Gamma} \, \big|\E F_{N,\beta}^{\gamma}-\frac12\beta^2v \big|=0\,,
	\end{align}
	for some measurable subset $\bar\Sigma_N\subset \Sigma_N$ such that $\mu_N(\bar\Sigma_N)=1$.
Let $\bar f_N(\bs,\varphi)$ be a deterministic function from $\bar\Sigma_{N}\times \R^{\Sigma_{N}}$ to $\R$ 
such that $\bar f_N(\bs,H_N^\gamma)$ is a.s.\ measurable.  By essentially the same proof as for 
Theorem \ref{thm:high-temp}, if, 
for some measurable $A\subset \R$, 
\begin{equation}\label{eq:HighTemp-Gamma}
	\varlimsup_{\epsilon\to0}\varlimsup_{N\to\infty}\frac{1}{N}\log\Big\{\sup_{\substack{\gamma\in\Gamma\\\bs\in\bar\Sigma_{N}}}
	\sup_{|E-\beta v|<\epsilon}\P\big(\bar f_{N}(\bs,H_N^\gamma)\in A\,\big|\,\frac{1}{N}H^{\gamma}_{N}(\bs)=E\big)\Big\}<0,
\end{equation}
	then 
	\begin{equation}\label{eq:Gamma-concl}
	\varlimsup_{N\to\infty}\frac{1}{N}\log \Big( \sup_{\gamma\in\Gamma}\E \big[ G^{\gamma}_{N,\beta}\left(\bar f_{N}(\bs,H_N^{\gamma})\in A\right) \big] \Big)  <0\,.
	\end{equation}
Indeed,  with $v=|\Gamma|=1$ and $\E H^\gamma_N(\bs)=0$,  the bound \eqref{eq:Gamma-concl} follows from Theorem \ref{thm:high-temp} (see \eqref{eq:HighTempCondUnif}) applied with $\bar \Sigma_N$. Generalizing to general $v>0$ is just a matter of scaling, see \eqref{eq:variance-bd}.  By going over the proof, one can check that allowing for vanishing expectation and uniform bound over $\gamma$ for $|\Gamma|>1$ only requires straightforward modifications.

Our next lemma,  whose proof is deferred to the end of the section,  verifies \eqref{eq:cond-Gamma}
and thereby allows us to employ the above,  for $v=\bxik(1)$ of \eqref{eq:xim}, $\Gamma=\bar V(\epsilon_{N})$,
$\bar\Sigma_N=B(\vec \bx_{k}^\be,\epsilon_N)$, $\Sigma_{N}=\SN$ where $H_N^\gamma(\cdot)$ at $\gamma=(\vec E,\vec R)$ 
follows the conditional law of  $H_N(\cdot)-N\Es(q_k)$ given $\cpt(\vec \bx_{k}^\be,\vec E,\vec R)$.
\begin{lem}\label{lem:Hhat}
		Suppose that $\hat H_N(\bs;\vec E,\vec R)$ is a Gaussian field having the law of $H_N(\bs)-N\Es(q_k)$ conditional on $\cpt(\vec \bx_{k}^\be,\vec E,\vec R)$. Then,
		\begin{align}
			\lim_{N\to\infty}\sup_{\bar V(\epsilon_{N})}\sup_{\bs\in B(\vec \bx_{k}^\be,\epsilon_N)}&\Big|\frac1N \E\big[ \hat H_N(\bs;\vec E,\vec R) \big]\Big| =0,\label{eq:Hhat1}\\
			\lim_{N\to\infty}\sup_{\bar V(\epsilon_{N})}\sup_{\bs\in B(\vec \bx_{k}^\be,\epsilon_N)}&\Big|\frac1N \E\big[\hat H_N(\bs;\vec E,\vec R)^2\big]- \bar\xi_{k}(1)\,\Big|=0,\label{eq:Hhat2}\\
			\lim_{N\to\infty}\sup_{\bar V(\epsilon_{N})}&\bigg|\, \E\bigg[\frac1N\log\Big( \int_{B(\vec \bx_{k}^\be,\epsilon_N)} e^{\beta \hat H_N(\bs;\vec E,\vec R)}d\bs\Big) \bigg]-\frac12\beta^2\bar \xi_k(1)\, \bigg|=0,\label{eq:Hhat3}
		\end{align}
	where here $d\bs$ denotes integration \abbr{wrt} the normalized to $1$ measure on $B(\vec\bx_k^{\be},\epsilon_N)$.
	\end{lem}
	
	From \eqref{eq:alphaPS},  we may write the expectation in \eqref{eq:1ptE3} as
	\[
	\E\Bigg[ \frac{\int_{B(\vec \bx_{k}^\be,\epsilon_N)}\indic\Big\{ \bar f_{N}(\bs, \hat H_N(\,\cdot\,;\vec E,\vec R) + N\Es(q_k))\in A \Big\}e^{\beta \hat H_N(\bs;\vec E,\vec R)}d\bs}
	{\int_{B(\vec \bx_{k}^\be,\epsilon_N)}e^{\beta  \hat H_N(\bs;\vec E,\vec R)}d\bs}  \Bigg] \,.
	\]
	%(where we eliminated the common factor $N\Es(q_k)$} from the two exponents).
    	Thus,  \eqref{eq:1ptE3},  which is of the form of \eqref{eq:Gamma-concl},  holds,  once we 
    	verify the relevant version of	\eqref{eq:HighTemp-Gamma}. That is,  the proof will be complete upon showing that 
	\begin{equation}\label{eq:conditional_prob_lowtemp}
		\varlimsup_{\epsilon\to0}\varlimsup_{N\to\infty}\frac{1}{N}\log\Big\{ \sup_{\bar V(\epsilon_N)}\sup_{\bs\in B(\vec \bx_{k}^\be,\epsilon_N)}\sup_{|E-\beta \bar\xi_k(1)|<\epsilon} P(\bs,A;E,\vec E,\vec R) \Big\} <0,
	\end{equation} 
where,  written in terms of $\tilde H_N(\cdot) := \hat H_N(\cdot)+N \Es(q_k)$,  we have that
% has the law of $H_N(\bs)$ conditional on $\cpt(\vec \bx_{k}^\be,\vec E,\vec R)$ and
\begin{align*}
	P(\bs,A;E,\vec E,\vec R)	&:=
	%\P\bigg(\bar f_{N}(\bs, \hat H_N(\,\cdot\,;\vec E,\vec R) + N\Es(q_k)) \in A\,\Big|\,\frac{1}{N} \hat H_N(\bs;\vec E,\vec R)=E\bigg)\\	\,=&
	\P\big(\bar f_{N}(\bs, \tilde H_N(\,\cdot\,;\vec E,\vec R) ) \in A\,\big|\,\frac{1}{N} \tilde H_N(\bs;\vec E,\vec R)=E+\Es(q_k)\big) \\
	&\hphantom{:}= \P\big(\bar f_{N}(\bs, H_N) \in A\,\big|\, \cpt(\bs, \vec \bx_{k}^\be,E+\Es(q_k),\vec E,\vec R)
	\big). 
\end{align*}
%Of course, the conditional law of $\tilde H_N(\bs;\vec E,\vec R)$ as above is the same as the law of $H_N(\bs)$ conditional on $\cpt(\vec \bx_{k}^\be,\vec E,\vec R)$ and $\frac1NH_N(\bs) = E+\Es(q_k)$. 
Note that \eqref{eq:conditional_prob_lowtemp} follows directly 
from our assumption \eqref{eq:bd_Gaussian_average-11} (see also \eqref{eq:V}),
%of Theorem \ref{thm:low temp 1},  
provided that 
\begin{equation}\label{eq:Fp-bxi}
F'(\beta,\xi)=\Es(q_k)+\beta \bxik(1) 
\end{equation}
(where we denote by $F(\beta,\xi)$ the limiting free energy of the spherical model with mixture $\xi$).
To verify \eqref{eq:Fp-bxi},  recall \eqref{eq:kRSB},  \eqref{eq:xiq},  and 
that by the \abbr{TAP} representation of 
\cite[Theorem 5]{FElandscape},  for any $q\in[0,1)$
\[
F(\beta,\xi)  \ge  \beta \Es(q) +\frac12\log(1-q)+F(\beta,\bar\xi_{q}) \,,
\]
with equality for $q\in \mathcal{S}_P$.  Consequently,  $F'(\beta,\xi)$ equals 
the partial derivative in $\beta$ of the \abbr{rhs} at $q_k\in\mathcal{S}_P$.  That is,
\begin{equation}\label{eq:TAPbetaderivative}
F'(\beta,\xi) = \Es(q_k) +\frac{d}{d\beta }F(\beta,\bxik)
\end{equation}
(as $\bxik=\bar\xi_{q_k}$,   see \eqref{eq:xim}).  From \cite[Corollary 6]{FElandscape}  we have that
$F(\beta,\bxik)=\frac12\beta^2\bxik(1)$,  so $\beta$ is bounded by the critical inverse-temperature of the mixture $\bxik(t)$. If $\beta$ is strictly less than the critical inverse-temperature then $F(\beta',\bxik)=\frac12\beta'^2\bxik(1)$ for all $\beta'$ in a neighborhood of $\beta$ and obviously $\frac{d}{d\beta }F(\beta,\bxik)=\beta \bxik (1)$.  In fact, 
the last identity holds even if $\beta$ is the critical inverse-temperature,  see the beginning of Section \ref{sec:HighTempPf}.  Upon comparing \eqref{eq:Fp-bxi} with \eqref{eq:TAPbetaderivative},  it remains to prove Lemma \ref{lem:Hhat}.
\qed

\begin{proof}[Proof of Lemma \ref{lem:Hhat}]
Recall that $B(\vec\bx_k^{\be},0) = \SN \cap S_0(\vec\bx_k^{\be})$ for $S_0(\cdot)$ of \eqref{eq:yi}.  
	Thus,  by Lemma \ref{lem:conditional_law_on_band} we have that conditional on 
	$\cpt(\vec \bx_{k}^\be,\vec E,\vec R)$,  the Gaussian field $H_N(\bs)-NE_k$ has on $B(\vec\bx_k^{\be},0)$
	zero mean and covariance $N \xi_{q_k}(\langle \bs,\bs'\rangle/N - q_k)$.  Mapping 
	$B(\vec\bx_k^{\be},0)$
	to $\mathbb{S}^{N-k-1}$ by $\bs \mapsto \frac{1}{\sqrt{1-q_k}} (\bs - \bx_k^{\be})$,  yields a
	conditional centered field which is a spherical model of the mixture $\bxik$ of \eqref{eq:xim}.  
	This spherical model has normalized variance $\bxik(1)$ and the limiting free energy $\frac12\beta^2\bxik(1)$
	(by \cite[Corollary 6]{FElandscape}),  resulting with \eqref{eq:Hhat1}-\eqref{eq:Hhat3} when $\bs$ is 
	restricted from $B(\vec\bx_k^{\be},\epsilon_N)$ to $B(\vec\bx_k^{\be},0)$.  
	
	To deal with the larger bands $B(\vec\bx_k^{\be},\epsilon_N)$, 
	note that conditionally on $\cpt(\vec \bx_{k}^\be,\vec E,\vec R)$,  the field $H_N(\bs)$ can be written 
	as the sum of two parts
	\begin{equation*}\label{eq:conditionalPS}
		N \varPhi_N(\bs; \vec E,\vec R) +  H_N^0(\bs) := \E\Big[ H_N(\bs) \,\Big|\,\cpt(\vec \bx_{k}^\be,\vec E,\vec R) \Big] + H_N^0(\bs),
	\end{equation*}
	where the centered field $H_N^0(\bs)$ has the law as $H_N(\bs)$ conditional on 
	$\cpt(\vec \bx_{k}^\be,\vec 0,\vec 0)$.  By Anderson's inequality,  see Remark \ref{rem:Anderson},  the Lipschitz constant of $H_N^0(\bs)$ over the sphere is stochastically dominated by that of $H_N(\bs)$.  
The latter is bounded by $L\sqrt N$ with probability at least $1-e^{-CN}$ for some constants $L=L(\xi)$ and $C=C(\xi)$, 
see e.g.\ \cite[Lemma 6.1]{TAPChenPanchenkoSubag} or \cite[Lemma A.3]{MontanariSubag2023}. 	
Thus, for some $\delta_N\to0$, on $B(\vec\bx_k^{\be},\epsilon_N)$,  the normalized conditional variance 
$\frac1N\E\big[ H^0_N(\bs)^2]$ is uniformly within $\delta_N$ from its value $\bxik(1)$ on $B(\vec\bx_k^{\be},0)$.
Further,  by the same Lipschitz property,  one easily sees that in the $N\to\infty$ limit,  the free energy of $H^0_N(\bs)$ 
on the band $B(\vec\bx_k^{\be},\epsilon_N)$ is the same as the one computed over $B(\vec\bx_k^{\be},0)$ 
(treated as a sphere with the uniform measure on it).  The proof is thus completed 
by showing that for some $\delta_N \to 0$,
% (which depends on $\epsilon_N$),  
\begin{equation}\label{eq:cont-cond-mean}
\sup_{(\vec E,\vec R)\in \bar V(\epsilon_{N})}
\sup_{\bs\in B(\vec\bx_k^{\be},\epsilon_N)}  \big|\varPhi_N(\bs; \vec E,\vec R)-E_k \big| \le \delta_N \,,
\end{equation}
where $E_k$ denotes the $k$-th element of $\vec E$. 
Proceeding to prove \eqref{eq:cont-cond-mean},  
we set $U_{\bs}=\mbox{span}\{U,\bs\}$ for
$U=\mbox{span}\{\bx_i^{\be}\}_{i\leq k}=\mbox{span}\{\be_i\}_{i\leq k}$,  letting
$\bar W_N=\{\partial_\bv H_N(\bx_a^{\be})\}_{a\leq k,\bv\in U_{\bs}^{\perp}}$ and 
\begin{align*}%\label{eq:WN}
		W_N&=\bigg\{\frac1N H_N(\bx_i^{\be}), \frac{\partial_{\bx_i^{\be}-\bx_{i-1}^{\be}} H_N(\bx_i^{\be})}{\|\bx_i^{\be}-\bx_{i-1}^{\be}\|^2},  \frac{\partial_{\bs-\bx_i^{\be}}H_N(\bx_i^{\be})}{\|\bs-\bx_i^{\be}\|^2} \bigg\}_{1 \leq i\leq k} 
		\bigcup 
		\bigg\{\frac{1}{\sqrt{N}} \partial_{\be_j} H_N(\bx_i^{\be}) \bigg\}_{1 \leq i< j\leq k} \,.
	\end{align*}
	We claim that the Gaussian vectors $(W_N,H_N(\bs)/N)$ and $\bar W_N$  are independent.  Indeed,  
	similarly to \eqref{eq:cov-H-der} and \eqref{eq:cov-der},  we have that
	\begin{equation}\label{eq:cov-idn}
	\begin{aligned}
		\E \left\{
		\partial_\bv H_N(\bx_a^{\be}) H_N(\bx) \right\} & = \xi'([ \bx_a^{\be},\bx ]) \langle \bv,\bx \rangle,\\
		\E \left\{
		\partial_\bv H_N(\bx_a^{\be}) \partial_{\bu} H_N(\bx_i^{\be}) \right\} & =  \xi'([ \bx_a^{\be},\bx_i^{\be} ]) \langle \bv,\bu \rangle + N^{-1}\xi''([ \bx_a^{\be},\bx_i^{\be} ]) \langle \bv,\bx_i^{\be} \rangle  \langle \bx_a^{\be},\bu \rangle,
	\end{aligned}
	\end{equation}
	and the inner products involving $\bv$ are zero for $\bx\in\{\bs,\bx_i^{\be}\}_{i\leq k}$ and $\bu\in\{\bs,\be_j,\bx_i^{\be}\}_{i,j\leq k}$.  
	%Taking 
	%for concreteness as rows of $M_i$ the vectors $\{ \be_j,  i<j \le k\}$,  the projection of $\bs-\bx^{\be}_i$ onto $U^{\perp}$ and an orthonormal basis of $U_{\bs}^{\perp}$,}
	This means that although the event $\cpt(\vec \bx_{k}^\be,\vec E,\vec R)$ specifies all variables in $W_N\cup \bar W_N$, it is only those in $W_N$ that affect the value of $\varPhi_N(\bs; \vec E,\vec R)$.
	
	Using the well-known formula for the conditional expectation of Gaussian variables,
	\begin{equation}\label{eq:Phi-N}
	\varPhi_N(\bs; \vec E,\vec R) = \Big[\E \big(W_N H_N(\bs) \big)\Big]^{\top} \Big[ N \E \big( W_N W_N^{\top}\big)\Big]^{-1} \varPsi (\vec E,\vec R),
	\end{equation}
	where here we treat $W_N$ as a vector of length $\frac{k(k+5)}{2}$
	and $\varPsi (\vec E,\vec R)$ is a vector whose values are $E_i$ and  $R_i$ in the entries corresponding to $\frac1N H_N(\bx_i^{\be})$ and $\frac{\partial_{\bx_i^{\be}-\bx_{i-1}^{\be}} H_N(\bx_i^{\be})}{\|\bx_i^{\be}-\bx_{i-1}^{\be}\|^2}$ 
%	in the vectorization of $W_N$, 
	and $0$ elsewhere.
		
		% Define $\hat W_N$ similarly to $W_N$, with the last element in \eqref{eq:WN} removed.   
	From the proof of
	%Assume that $\bs\in B(\vec\bx_k^{\be},\epsilon_N)$ with $N$ sufficiently large.  By either of
	 Lemma \ref{lem:conditional_law_on_band} or directly from \eqref{eq:cov-idn},
	 %or Corollary \ref{cor:Hm1} or \ref{cor:Hm2}, 
	 one sees that for $\bs \in B(\vec \bx_k^{\be},0)$ the vector $\E \big(W_N H_N(\bs) \big)$ and the
	  invertible\footnote{Here we use the fact that $\xi$ is not a pure mixture, since we assume it is generic. 
	  If it is pure,  then $k=1$ (see Remark \ref{rem:pure}),  and
$\partial_{\bx_1^{\be}}H_N(\bx_1^{\be})$ is measurable \abbr{wrt} $H_N(\bx_1^{\be})$. }
	 matrix $N \E (W_N W_N^{\top})$ depend neither on $\bs$ nor on $N$.
	 For $\bs \in B(\vec \bx_k^{\be},\epsilon_N)$,  the vector $\E \big(W_N H_N(\bs) \big)$
	 and the matrix $N \E (W_N W_N^{\top})$ match their values on $B(\vec \bx_k^{\be},0)$,
	 up to a small perturbation which depends on the vector  $\big(\langle \bs,\bx_i^{\be}\rangle/N-q_i\big)_{i\leq k}$ 
	 whose elements are bounded by $\epsilon_N$ in absolute value.  With $\varPsi (\vec E,\vec R)$
	 independent of $\bs$ and $N$,  this extends to  the \abbr{lhs} of \eqref{eq:Phi-N},  thereby yielding 
	 \eqref{eq:cont-cond-mean} (where we have also utilized the compactness of $\bar V(\epsilon_{N})$).	 
%there exists $c,\delta>0$ 
%(depending only on $\{q_1,\ldots,q_k\}$),  
%such that 
%the projection of $B(\vec\bx^{\be}_k,c)$ onto $U^{\perp}$ has norm at least $\delta \sqrt{N}$. 
%Assume \abbr{wlog} that $\epsilon_N<c$,  fix $\bs \in B(\vec\bx^{\be}_k,\epsilon_N)$,  	 
\end{proof}

\appendix
\section{\label{sec:Appendix}Statistics relative to pure states}

To extend
the \abbr{ckchs}-equations (which as derived apply only for \abbr{iid} initial conditions),
for Langevin dynamics \eqref{diffusion} initialized according to the Gibbs measure,  one must deal 
with more general statistics that are measured relative to the path $\vec\bx$ with which a pure state 
is defined (see Proposition \ref{prop:PS}).  We  thus derive here the corresponding generalization
of Theorem \ref{thm:low temp 1} 
(which is indeed used in \cite{DS2},  see Section \ref{subsec:Langevin}).

Assume the strict $k$-\abbr{RSB} condition \eqref{eq:kRSB} and suppose that  $\bar f_N(\bs,\vec\bx,\varphi)$ is a deterministic function from $\SN\times\AA_k\times C^\infty(\BN)$ to $\R$, which is  rotationally invariant in the sense that $\bar f_N(\bs,\vec \bx, \varphi(\cdot))= \bar f_N(O\bs,O\vec\bx,\varphi(O^T\cdot))$ for any orthogonal $O\in \R^{N\times N}$. The  continuity property of \eqref{eq:flim} is replaced here by having for  smooth $\varphi,\psi_i\in C^{\infty}(\BN)$,
\begin{equation}
	\lim_{t_{i}\to0}\bar f_{N}\Big(\bs,\vec\bx,\varphi+\sum_{i\leq k}t_{i}\psi_{i}\Big)=\bar f_{N}(\bs,\vec\bx,\varphi)\mbox{\quad and\quad} \lim_{\vec\bx'\to\vec\bx}\bar f_{N}(\bs,\vec\bx',\varphi)=\bar f_{N}(\bs,\vec\bx,\varphi).
\label{eq:flim2}
\end{equation}

Define the subset $\mathscr{B}(\epsilon)=\cup_{\vec\bx\in \mathscr{C}_k(\epsilon)}B(\vec\bx,\epsilon)$.
By Proposition \ref{prop:PS}, there exists a sequence $\epsilon_{N}\to0$ such that $\lim_{N\to\infty}G_{N,\beta}(\mathscr{B}(\epsilon_N))=1$, in probability. Suppose $\vec\bx(\cdot)$ is a random function from $\mathscr{B}(\epsilon_N)$ to $\mathscr{C}_k(\epsilon_N)$ such that $\bs\in B(\vec\bx(\bs),\epsilon_{N})$.

\begin{cor}
	\label{cor:low temp PS} Consider the spherical mixed $p$-spin model
	with a generic mixture $\xi(t)$ and $\beta>\beta_{c}$ at which $\xi(t)$ is strict $k$-\abbr{rsb}.
	% such that \eqref{eq:kRSB} holds.
	Suppose that $f_N(\bs,\vec\bx)=\bar f_N(\bs,\vec\bx,H_N(\cdot))$ for some deterministic rotationally invariant $\bar f_N(\bs,\vec\bx,\varphi)$ which satisfies the continuity property of \eqref{eq:flim2} and that $f_N(\bs,\vec\bx(\bs))$ is $\P$-a.s.\  a measurable function on $\SN$.  If for some measurable $A\subset\R$,
	\begin{equation}
		\varlimsup_{\epsilon\to0}\varlimsup_{N\to\infty}\sup_{B(\vec\bx^{\be}_k,\epsilon)\times V(\epsilon)}\frac{1}{N}\log\P\left(\, f_{N}(\bs,\vec\bx^{\be}_k)\in  A\,\Big|\,\cpt(\bs,\vec{\bx}^{\be}_k,E,\vec E,\vec R)\,\right)<0,\label{eq:bd_Gaussian_average-PS}
	\end{equation}
	%where $A_{\epsilon}:=\{t\in \R:\,\exists a\in A,\,|t-a|<\epsilon\}$,} 
	then 
	\begin{equation*}
		\lim_{N\to\infty}\E G_{N,\beta}\Big(\big\{ \bs\in \mathscr{B}(\epsilon_{N}):\,f_{N}(\bs,\vec\bx(\bs))\in A\big\}\Big)=0.\label{eq:GAverageLowT-PS}
	\end{equation*}
\end{cor}
\begin{proof}
	On $\mathscr{B}(\epsilon_{N})$,
	\[
	\indic\{f(\bs,\vec\bx(\bs))\in A\}\leq \sum_{\vec\bx\in\mathscr{C}_k(\epsilon_{N})} \indic\{f(\bs,\vec\bx)\in A,\,\bs\in B(\vec\bx,\epsilon_{N}) \}
	\]
	and therefore
	\begin{align*}
	G_{N,\beta}\Big(\big\{ \bs\in \mathscr{B}(\epsilon_{N}):\,f_{N}(\bs,\vec\bx(\bs))\in A\big\}\Big)
	&\leq 
	\sum_{\vec\bx\in\mathscr{C}_k(\epsilon_{N})} G_{N,\beta}\Big(\big\{ \bs\in B(\vec\bx,\epsilon_{N}):\,f_{N}(\bs,\vec\bx)\in A\big\}\Big)
	\\
	&\leq 
	\sum_{\vec\bx\in\mathscr{C}_k(\epsilon_{N})} G_{N,\beta}\Big(  f_{N}(\bs,\vec\bx)\in A\,\Big|\,B(\vec\bx,\epsilon_{N})\Big).
	\end{align*}

	Defining,  analogously to \eqref{dfn:alpha-k}, 
	\[
	\alpha_{k}(\vec\bx) = \alpha_{k,\epsilon}(\vec\bx)= G_{N,\beta}\Big(  f_{N}(\bs,\vec\bx)\in A\,\Big|\,B(\vec\bx,\epsilon)\Big),
	\]
	and using the notation of Section \ref{sec:multiLvlKR},  we thus complete the proof 
	by showing that for any $\epsilon_N \to 0$, 
	\begin{equation*}
		\varlimsup_{N\to\infty}\frac1N\log \E \Big[ \alpha^{\epsilon_N}_{0,k}(0)\wedge 1\Big] 
		<0.
	\end{equation*}
Similarly to the proof of Theorem \ref{thm:low temp 1},  the preceding can be shown from \eqref{eq:bd_Gaussian_average-PS} using Proposition \ref{prop:multilvlKR}. We note that the second assumption in \eqref{eq:flim2} is required here for the continuity of $\vec\bx\mapsto \alpha_{k,\epsilon}(\vec\bx)$.
\end{proof}

\bibliographystyle{plain}
\bibliography{master}
\end{document}